\documentclass[a4paper,11pt,reqno,noindent]{amsart}
\usepackage[centertags]{amsmath}
\usepackage{mathtools} 
\usepackage{amsfonts,amssymb,amsthm,dsfont,cases,amscd,esint,enumerate,comment}
\usepackage[T1]{fontenc}
\usepackage[english]{babel}
\usepackage[applemac]{inputenc}
\usepackage{newlfont,cancel}
\usepackage{color}
\usepackage[body={15cm,21.5cm},centering]{geometry} 
\usepackage{fancyhdr}
\usepackage{enumerate} 

\theoremstyle{plain}
\newtheorem{lem}{Lemma}[section]
\newtheorem{theor}[lem]{Theorem}
\newtheorem{prop}[lem]{Proposition}

\newtheorem{rem}[lem]{Remark}
\newtheorem{rems}[lem]{Remarks}
\newtheorem{defin}[lem]{Definition}

\numberwithin{equation}{section}

\newcommand{\e}{\varepsilon}
\newcommand{\eps}{\varepsilon}
\newcommand{\calM}{\mathcal M}
\newcommand{\Ic}{\mathcal I}
\newcommand{\N}{\mathbb N}
\newcommand{\Dd}{\mathbb D}
\newcommand{\R}{\mathbb R}
\newcommand{\E}{\mathbb E}
\newcommand{\C}{\mathbb C}

\newcommand{\Pc}{\mathcal P}
\newcommand{\Gc}{\mathcal G}
\newcommand{\PP}{\mathbb P}
\newcommand{\Lc}{\mathcal L}
\newcommand{\Dc}{\mathcal D}

\newcommand{\Id}{\operatorname{Id}}

\newcommand{\Div}{{\operatorname{div}}}
\newcommand{\loc}{{\operatorname{loc}}}
\newcommand{\Sym}{{\operatorname{sym}}}
\newcommand{\cvf}{\rightharpoonup}

\usepackage[colorlinks,linkcolor=black,citecolor=black,urlcolor=black]{hyperref}

\title[2D Vlasov-Poisson: derivation and fluctuations]{A duality method for mean-field limits\\with singular interactions:\\[3mm]The 2D Vlasov-Poisson equation and fluctuations}
\author[M.~Duerinckx]{Mitia Duerinckx}
\address[Mitia Duerinckx]{Universit\'e Libre de Bruxelles, D\'epartement de Ma\-th\'e\-matiques, 1050 Brussels, Belgium}
\email{mitia.duerinckx@ulb.be}
\author[P.-E.~Jabin]{Pierre-Emmanuel Jabin}
\address[Pierre-Emmanuel Jabin]{Penn State University, Department of Mathematics, State College, PA 16802, USA}
\email{pejabin@psu.edu}

\begin{document}
\begin{abstract}
We derive for the first time the 2D Vlasov-Poisson equation from the dynamics of classical particles with Coulomb interactions, without microscopic cutoff. The proof builds on our recent dual hierarchical approach to mean-field limits, combined with a refined analysis of the dual BBGKY hierarchy exploiting kinetic regularization. More generally, our result applies in any dimension to singular interaction forces $K\in L^{2-\eta}_\loc$ with~$\eta>0$ small enough, for as long as the mean-field solution remains sufficiently regular, and it extends to particle systems with Brownian forcing in velocity. We further show that the same method characterizes empirical fluctuations around the mean-field limit, through the fluctuating Vlasov equation, under the same singularity assumptions.
\end{abstract}

\maketitle
\setcounter{tocdepth}{1}
\tableofcontents
\allowdisplaybreaks

\section{Introduction}

\subsection{General overview}
This paper is a companion to~\cite{BDJ-24}. We consider a system of $N\ge2$ exchangeable classical particles in phase space $\Dd:=\R^d\times\R^d$, interacting through a pairwise force kernel $K:\R^d\to\R^d$. Denoting by $Z_{i,N}=(X_{i,N},V_{i,N})$ the position and velocity of particle $i$, the dynamics is given by
\begin{equation}\label{eq:manyparticles}
\left\{\begin{array}{ll}
dX_{i,N}=V_{i,N}\,dt,\\
dV_{i,N}=\frac1{N-1}\sum_{j:j\ne i}^NK(X_{i,N}-X_{j,N})\,dt+\sqrt{2\alpha}\,dB_{i,N},
\end{array}\right.
\end{equation}
where $\alpha\ge0$ and the $B_{i,N}$'s are independent Brownian motions. The deterministic Newton dynamics is recovered for $\alpha=0$. The associated mean-field equation for the one-particle density $f$ is the Vlasov equation
\begin{equation}\label{eq:Vlasov}
\partial_t f+v\cdot\nabla_x f-\alpha\triangle_v f+(K\ast f)\cdot\nabla_v f=0,
\end{equation}
where $(K\ast f)(x):=\int_\Dd K(x-x')f(x',v')\,dx'dv'$. A central challenge in mathematical physics has been to justify this mean-field approximation for unbounded forces, especially for Coulomb interactions without microscopic cutoff, and to characterize the fluctuations around the limiting dynamics, which capture the leading stochastic finite-$N$ corrections.

In the sequel, we extend the dual hierarchical method introduced in~\cite{BDJ-24} by developing two essential new ingredients: kinetic regularization, based on velocity averaging when~$\alpha=0$ and on hypoellipticity when $\alpha>0$; and a probabilistic interpretation of the limiting dual hierarchy in terms of Wiener-It\^o chaos.
The first provides a gain of regularity on the hierarchy of dual correlations, allowing us to treat more singular interaction forces than in the $L^2$ framework of~\cite{BDJ-24}. The second allows us to identify the limiting fluctuations through the same hierarchical analysis, thus making it possible to derive fluctuations alongside the mean-field limit at no extra cost.

In particular, we obtain the first ever rigorous derivation of the famous Vlasov-Poisson system in dimension $d=2$ from the deterministic particle dynamics~\eqref{eq:manyparticles} with Coulomb forces $K(x)=\pm x|x|^{-d}$, both repulsive or attractive, without microscopic cutoff.
More generally, we establish the mean-field limit from~\eqref{eq:manyparticles} to~\eqref{eq:Vlasov} in any dimension for singular forces beyond the $L^2$ threshold of~\cite{BDJ-24}, allowing for $K\in L^{2-\eta}_\loc(\R^d)$ with sufficiently small $\eta>0$; see~\eqref{eq:restr-gamma} below. As in~\cite{BDJ-24}, we impose no structural assumption on the interaction force, such as repulsiveness, antisymmetry, or a potential structure.
The result holds both with and without diffusion, and remains valid for as long as the mean-field solution remains sufficiently regular. At the same time, we identify empirical fluctuations at the natural central-limit scale, in terms of the expected fluctuating Vlasov equation, under exactly the same assumptions on the interaction force.

\vspace{0.4cm}
The Vlasov-Poisson system is arguably the earliest kinetic mean-field model that has been formulated in the physics literature: it was introduced in the gravitational setting by Jeans~\cite{Jeans-15} in 1915 to describe self-gravitating stellar systems, and subsequently in the electrostatic setting by Vlasov~\cite{Vlasov-38} in the 1930s to describe collisionless plasmas. Despite this long history and the considerable progress in other mean-field settings, its derivation from the full Coulomb particle dynamics had frustratingly remained open in dimensions~\mbox{$d\ge2$}.

For Lipschitz interaction forces, the mean-field approximation goes back to the classical works~\cite{Neunzert-Wick-74,Braun-Hepp-77,Dobrushin-79,Sznitman-91}, and was subsequently extended to arbitrary bounded forces by a relative entropy method~\cite{Jabin-Wang-16} . For unbounded interactions, by contrast, the problem is substantially more delicate: singularities prevent the direct use of classical stability estimates, while close particle encounters may generate large microscopic forces. As a consequence, results have remained relatively scarce.

In the deterministic setting $\alpha=0$, the only results for unbounded forces have long remained those of~\cite{HaurayJabin-07,HaurayJabin-15}, covering mildly singular forces with $|K(x)|\lesssim|x|^{-\gamma}$ and $|\nabla K(x)|\lesssim|x|^{-\gamma-1}$ for $\gamma<1$. Further progress was made in~\cite{Grass-19}, improving those results possibly to~$\gamma\le\frac43$ in 3D. This was only recently extended through the dual hierarchical method that we developed with D.~Bresch in~\cite{BDJ-24}, allowing us to cover arbitrary forces $K\in L^2_{\loc}$, both with and without diffusion. This substantially enlarged the admissible class of singularities, but still excluded the 2D Coulomb force, which belongs to $L^{2,\infty}_\loc$ but not $L^2_\loc$.

In the diffusive setting $\alpha>0$, a direct hierarchical approach was developed in~\cite{BJS-22} for repulsive interactions deriving from a nonnegative potential with suitable exponential integrability. Exploiting the velocity diffusion to compensate for the derivative loss in the BBGKY hierarchy, this yielded the first derivation of the repulsive 2D Vlasov-Poisson-Fokker-Planck equation, and the corresponding 3D result was recently completed in~\cite{Feng-Wang-26}. Although these results reach stronger singularities, they rely on velocity diffusion, repulsiveness, and the potential structure of the interactions, and are restricted to short times. In this setting, the present work provides the first derivation of the \emph{attractive} 2D Vlasov-Poisson-Fokker-Planck equation, and removes the short-time restriction in the repulsive case as long as the mean-field solution remains regular enough.

After the completion of this work, Nam and Nguyen~\cite{NamNguyen26} developed a different extension of our duality method based on Morawetz-type estimates. Their result covers more singular interactions including the 2D Coulomb case, but it only applies to the diffusive setting, with a velocity diffusion $\alpha$ possibly tending to $0$ at a controlled rate as $N\uparrow\infty$. In contrast, our approach also applies to the purely deterministic setting $\alpha=0$.

A complementary literature considers the mean-field limit for Coulomb interactions regularized at some $N$-dependent microscopic scale, e.g.~\cite{BoersPickl-16,Lazarovici-16,LazaroviciPickl-17,CarrilloChoiSalem-19,HuangLiuPickl-20}. In the deterministic setting, recent advances have reached cutoff radii arbitrarily close to $N^{-5/12}$ in 3D~\cite{FeistlHeldPickl-25-3D}, and of order $N^{-2}$ in 2D~\cite{FeistlHeldPickl-25-2D}. These results approximate the Coulomb force remarkably closely, but retain a nonzero cutoff for every finite $N$. We also mention that the derivation of the mean-field limit is of course much simpler in a perturbative regime near Gibbs equilibrium: we have for example derived in~\cite{DJ-26} the mean-field limit toward the linearized Vlasov equation for a broad class of singular interactions, including Coulomb forces in dimensions $d\le3$.

The difficulty of treating singular forces in the kinetic second-order dynamics~\eqref{eq:manyparticles} contrasts with the much more extensive theory available for \emph{first-order} (inertialess) dynamics. To cite only a few recent results, in that case, the modulated energy method yields mean-field limits for Coulomb and more singular Riesz interactions, and also applies to second-order dynamics in the monokinetic regime~\cite{Serfaty-20}. Attractive Coulomb interactions with diffusion were treated in~\cite{Bresch-Jabin-Wang-20}, leading to the Keller-Segel equation. We also mention the particularization of our dual hierarchical method, see~\cite[Section~1.2.1]{BDJ-24}, as well as the very recent multiscale mollification method of~\cite{Serfaty-Nguyen-26}: both apply to broad classes of first-order dynamics, attractive or repulsive, without requiring a potential structure; see~\cite[Section~1.4]{BDJ-24} for a detailed comparison. While these other approaches are specific to first-order dynamics, our dual hierarchical method also applies to the kinetic particle dynamics~\eqref{eq:manyparticles}, and the present extension to the full 2D Coulomb singularity thus narrows the gap between our understanding of first- and second-order systems.

\vspace{0.3cm}
Regarding fluctuations, the classical central-limit picture and the corresponding fluctuating Vlasov dynamics were rigorously justified already in the early days of mean-field theory for Lipschitz interactions~\cite{Braun-Hepp-77,Sznitman-84,Sznitman-91}, and the theory was extended more recently to arbitrary bounded forces~\cite{Wang-Zhao-Zhu-23}. For unbounded forces, by contrast, results have remained much scarcer and have so far been restricted to the diffusive setting $\alpha>0$: in~\cite{DJ-25}, we obtained a fixed-time central limit theorem for square-integrable forces deriving from a potential with suitable exponential integrability, while a process-level central limit theorem was established in~\cite{HZZ-26} for Kato-class interactions, including $K\in L^p$ with $p>3d$. Although these works extend fluctuation theory beyond bounded forces, the admissible singularities are far from the 2D Coulomb case, which we can now treat both with and without diffusion, for repulsive as well as attractive forces.

We emphasize that even for first-order dynamics fluctuation results for unbounded forces had remained remarkably scarce compared with the much more extensive mean-field theory: to our knowledge, the only available results were those of~\cite{Wang-Zhao-Zhu-23} covering the 2D Biot-Savart kernel with or without diffusion, and those of~\cite{CecchinNikolaev-25} for repulsive Coulomb interactions with diffusion. In particular, the deterministic 2D Coulomb case had remained open even for first-order dynamics, whereas we now cover it in the second-order setting.

A particularly unusual feature of our result is that we obtain fluctuations under exactly the same assumptions on the interaction force as for the mean-field limit, while previous approaches had always required much stronger assumptions. More remarkably still, our proof does not rely on any quantitative rate of convergence for particle marginals, nor on preliminary uniform moment bounds or trajectory tightness estimates for the empirical fluctuation field. Conversely, although we identify fluctuations at the central-limit scale, our result does not yield a corresponding convergence rate for marginals; see Remark~\ref{rem:fluct} below. These features reflect a fundamental difference between our dual hierarchical approach and more traditional proofs of fluctuation limits~\cite{Sznitman-84,Sznitman-91,Wang-Zhao-Zhu-23,HZZ-26}.

\subsection{Main result~1: mean-field limit}
We turn to a statistical description of the particle system. Denoting by $F_N$ the joint density on the $N$-particle phase space $\Dd^N$, the dynamics~\eqref{eq:manyparticles} formally leads to the Liouville equation
\begin{equation}\label{eq:Liouville}
\partial_tF_N+\sum_{i=1}^N(v_i\cdot\nabla_{x_i}-\alpha\triangle_{v_i})F_N+\frac1{N-1}\sum_{i\ne j}^NK(x_i-x_j)\cdot\nabla_{v_i}F_N\,=\,0.
\end{equation}
Exchangeability of the particles amounts to the symmetry of $F_N$ under permutations of the particle variables $z_i:=(x_i,v_i)\in\Dd$, $1\le i\le N$. In the macroscopic limit $N\uparrow\infty$, we aim at an averaged description of finite subsets of typical particles, as encoded by marginals
\begin{equation}\label{eq:def-marginal}
F_{N,k}(z_1,\ldots,z_k):=\int_{\Dd^{N-k}}F_N(z_1,\ldots,z_N)\,dz_{k+1}\ldots dz_N,\qquad1\le k\le N.
\end{equation}
For asymptotically independent initial data with one-particle density $f_\circ$, one expects propagation of chaos,
\begin{equation}\label{eq:prop-chaos}
F_{N,k}(t)\to f(t)^{\otimes k}\qquad\text{for all $k\ge1$ and $t\ge0$},
\end{equation}
where $f$ solves the Vlasov equation~\eqref{eq:Vlasov} with $f|_{t=0}=f_\circ$.

Our first main result establishes this conclusion for singular interactions beyond the $L^2$ threshold of~\cite{BDJ-24}, including 2D Coulomb forces, both repulsive and attractive, for as long as the mean-field solution has the required regularity. Compared with~\cite{BDJ-24}, stronger regularity is required for the mean-field solution: we assume
\begin{equation}\label{eq:reg-req-0}
f\,\in\, L^\infty(0,T_0;\Pc\cap L^\infty(\Dd)),\qquad \rho_f:=\int_{\R^d}f(\cdot,v)dv \,\in\, L^\infty([0,T_0]\times\R^d),
\end{equation}
and further require
\begin{equation}\label{eq:reg-req-0b}
K\ast f \,\in\, L^\infty(0,T_0;W^{1,\infty}(\R^d)),\qquad f^{-1}\,\in\,L^\infty_\loc([0,T_0]\times\Dd),
\end{equation}
together with the following weighted regularity of logarithmic derivatives (whereas $\nabla_v\sqrt f \in L^1(0,T_0;L^2(\Dd))$ was sufficient in~\cite{BDJ-24}):
\begin{enumerate}[$\bullet$]
\item in the deterministic case $\alpha=0$, for some~$\delta\in(0,\frac12]$,
\begin{gather}
f^\delta,~f^\delta\nabla_v\log f\,\in\,L^\infty\big(0,T_0;L^2_xH^{3/2}_v(\Dd)\cap L^\infty_xH^{3/2}_v(\Dd)\big),\label{eq:reg-req-1}\\
f^{\frac12-\delta}\nabla_v\log f\,\in\, L^\infty([0,T_0]\times\Dd);\nonumber
\end{gather}
\item in the diffusive case $\alpha>0$, for some~$\delta\in(0,\frac12]$,
\begin{gather}
f^\delta,~f^\delta\nabla_v\log f\,\in\, L^\infty\big(0,T_0;L^2(\Dd)\cap L^\infty_xL^2_v(\Dd)\big),\label{eq:reg-req-2}\\
f^{(\frac12-\delta)\wedge\frac\delta2}\big(|\nabla_v\log f|^2+|\triangle_v\log f|\big)\,\in\, L^\infty([0,T_0]\times\Dd),\nonumber\\
\triangle_v\log f
+\delta|\nabla_v\log f|^2
\ge-C
\qquad\text{a.e. on $[0,T_0]\times\Dd$}.
\label{eq:reg-req-2b}
\end{gather}
\end{enumerate}

\begin{theor}\label{th:main}
Let
\[\alpha\ge0,\qquad K\in L^{2d/(d+2\gamma)}_\loc(\R^d),\qquad \sup_{|x|\ge1}|K(x)|<\infty,\]
where $\gamma$ satisfies
\begin{equation}\label{eq:restr-gamma}
0\le\gamma<\gamma_*:=\left\{\begin{array}{lll}
\frac16&:&\alpha>0,\\[1mm]
\frac{1}{16}\big(5d+4-\sqrt{(5d+4)^2-32d}\big)=\tfrac15+O(\tfrac1d)&:&\alpha=0,\,d\ge2,\\[1mm]
\frac12(\sqrt5-2)&:&\alpha=0,\,d=1.
\end{array}\right.
\end{equation}
For every $N\ge2$, consider a symmetric global weak duality solution $F_N\in L^\infty(\R^+;L^1_+(\Dd^N))$ of the Liouville equation~\eqref{eq:Liouville}, in the sense of~\cite[Appendix~A]{BDJ-24}, with initial data $F_N^\circ\in\Pc\cap L^1(\Dd^N)$ satisfying the following chaoticity condition,
\begin{equation}\label{eq:initial}
\limsup_{N\uparrow\infty}\int_{\Dd^N}\frac{|F_N^\circ|^2}{f_\circ^{\otimes N}}<\infty,
\end{equation}
for some $f_\circ\in\Pc\cap L^\infty(\Dd)$. Let $f$ be a weak solution of the Vlasov equation~\eqref{eq:Vlasov} with initial data~$f_\circ$ and assume that, for some~$T_0>0$, it satisfies the regularity requirements stated above, namely~\eqref{eq:reg-req-0}--\eqref{eq:reg-req-0b} together with~\eqref{eq:reg-req-1} if~$\alpha=0$ and with~\eqref{eq:reg-req-2}--\eqref{eq:reg-req-2b} if~$\alpha>0$.
Then, for all $k\ge1$ and $t\in[0,T_0]$,
\begin{equation}\label{eq:concl-mfl}
F_{N,k}(t)\to f(t)^{\otimes k}\quad\text{narrowly as finite measures on $\Dd^k$}.
\end{equation}
\end{theor}

An important feature of this result is its generality with respect to the particle dynamics: as in~\cite{BDJ-24}, it applies to arbitrary \emph{weak duality solutions} of the Liouville equation, without requiring the existence of renormalized solutions. Weak duality solutions always exist under the sole assumption $K\in L^1_\loc$, cf.~\cite[Proposition~A.2]{BDJ-24}, but they need not be unique or even conserve their total mass, which is consistent with possible finite-time breakdown of the underlying particle dynamics through collisions or escape to infinity in phase space. Nevertheless, since the limiting mean-field solution is a probability measure, the narrow convergence~\eqref{eq:concl-mfl} implies, for all $t\in[0,T_0]$,
\begin{equation}\label{eq:mass-to1}
\int_{\Dd^N}F_N(t)=\int_{\Dd^k}F_{N,k}(t)\to1\qquad\text{as $N\uparrow\infty$,}
\end{equation}
so any possible loss of mass at the particle level becomes negligible in the mean-field limit.

Although the admissible range~\eqref{eq:restr-gamma} for $\gamma$ might be improved, see item~(c) below, the present approach could not be pushed to reach Coulomb interactions in dimensions $d\ge3$ in the deterministic setting $\alpha=0$. Indeed, velocity averaging can provide at most half a spatial derivative in the relevant $L^2$ framework, whereas the 3D Coulomb force critically fails to belong to $H^{-1/2}_\loc$.

\begin{samepage}\begin{rems}\label{rem:mfl}\
\begin{enumerate}[(a)]
\item \emph{Initial chaos assumption:}
The initial chaoticity assumption~\eqref{eq:initial} is a uniform bound on the $\chi^2$-divergence of $F_N^\circ$ from the reference product law $f_\circ^{\otimes N}$. As explained in~\cite[Remark~1.2(c)]{BDJ-24}, it is stronger than the usual notions of Kac or entropic chaos, and is the natural condition for our $L^2$-duality method.

\smallskip\item \emph{Regularity of the mean-field solution:}
Taking $\delta=\frac12$ in~\eqref{eq:reg-req-1}--\eqref{eq:reg-req-2b} simplifies the regularity requirements for $f$. In particular, in the diffusive setting, the lower bound~\eqref{eq:reg-req-2b} then follows from the boundedness assumptions in~\eqref{eq:reg-req-2}. However, this choice would also require $\nabla_v\log f$ itself to be bounded, which is incompatible with Maxwellian velocity tails. Allowing for $\delta<\frac12$ provides more flexible weighted conditions at large velocities, which can be checked to hold at least on short times for suitable initial data.

\smallskip\item \emph{Range of $\gamma$:}
The restriction~\eqref{eq:restr-gamma} on $\gamma$ stems from the uniform-in-$N$ regularity estimates of Propositions~\ref{prop:bnd-diff} and~\ref{prop:bnd-nodiff} below, rather than from any apparent criticality of the limiting dual hierarchy itself. Indeed, as explained in Remark~\ref{rem:restr-gamma-uniqueness}, the limiting hierarchy is believed to remain well-posed up to $\gamma\le\frac12$ when $\alpha=0$, and $\gamma\le\frac56$ when $\alpha>0$. At present, however, we do not know how to establish the required uniform-in-$N$ regularity estimates throughout this larger range.\\
Some improvement over~\eqref{eq:restr-gamma} is easily achieved when $\alpha>0$. Indeed, for simplicity, we rely in that case on hypoellipticity in the proof, which yields the restriction~$\gamma<\frac16$. Appealing instead to velocity averaging, as we do in the deterministic case $\alpha=0$, and combining it with the additional regularization in velocity provided by the diffusion, we could extend the admissible range to
\[\gamma<\tfrac14\big(2d+1-\sqrt{4d^2+1}\big)=\tfrac14+O(\tfrac1d).\]
Since this refinement would not bring any additional physically-relevant interactions within the scope of our results, however, we do not pursue it here.
\end{enumerate}
\end{rems}\end{samepage}

\subsection{Main result~2: fluctuations around mean-field}
Beyond the mean-field approximation, fluctuations capture finite-$N$ effects, including microscopic noise and correlations, that are neglected in the Vlasov equation. They play a fundamental role in kinetic theory, in particular in understanding finite-particle corrections and the emergence of collisional features on longer time scales; see e.g.~\cite{Klimontovich-67,HamiltonFouvry-24}.

To describe the leading random corrections to the Vlasov dynamics, we consider the empirical measure $\mu_N$ associated with the particle dynamics~\eqref{eq:manyparticles}, and its fluctuations $\zeta_N$ around the mean-field limit at the central-limit scale,
\[\mu_N(t):=\frac{1}{N}\sum_{i=1}^N \delta_{(X_{i,N}(t),V_{i,N}(t))},\qquad \zeta_N(t):=\sqrt N\,(\mu_N(t)-f(t)).\]
For exchangeable particles, if the particle dynamics is well-posed, propagation of chaos~\eqref{eq:prop-chaos} is equivalently expressed as convergence of $\mu_N(t)$ in probability to~$f(t)$ in the narrow topology. At the next order, the fluctuation field $\zeta_N(t)$ is then formally governed by the associated \emph{fluctuating Vlasov equation}
\begin{equation}\label{eq:fluc-Vlas}
\partial_t\zeta+v\cdot\nabla_x\zeta+(K\ast f)\cdot\nabla_v\zeta+(K\ast\zeta)\cdot\nabla_vf=\alpha\triangle_v\zeta+\sqrt{2\alpha}\,\Div_v(\sqrt{f}\xi),
\end{equation}
where $\xi$ is an $\R^d$-valued space-time white noise on $\R^+\times\Dd$, independent of the initial fluctuation field.

Before stating our fluctuation result, we need to specify a suitable notion of solution for~\eqref{eq:fluc-Vlas}: although this equation is linear, the interaction term $K\ast\zeta$ cannot in general be interpreted pointwise when $\zeta$ is distribution-valued. We adopt a mild formulation in terms of the backward propagator for the dual linearized Vlasov operator. Unlike the usual martingale formulation, as used e.g.\@ in~\cite{Wang-Zhao-Zhu-23}, this does not require interpreting the interaction term in a distributional sense, nor imposing pathwise regularity on the solution. Here and below, $\langle\zeta,h\rangle$ denotes the usual distributional pairing with $h\in C^\infty_c(\Dd)$.

\begin{defin}\label{defin:mild-sol}
Let $f$ be a weak solution of~\eqref{eq:Vlasov} satisfying the regularity assumptions of Theorem~\ref{th:main} on $[0,T_0]$. Let $\zeta_\circ$ be a random distribution on $\Dd$ such that, for some $C_\circ<\infty$,
\begin{equation}\label{eq:zeta0-bndL2f}
\E\big[|\langle\zeta_\circ,h\rangle|^2\big]\le C_\circ\int_\Dd\Big|h-\int_\Dd h f_\circ\Big|^2 f_\circ,\qquad h\in C_c^\infty(\Dd),
\end{equation}
and let $\xi$ be an $\R^d$-valued space-time white noise independent of $\zeta_\circ$.
A \emph{mild solution} of~\eqref{eq:fluc-Vlas} on $[0,T_0]$ with initial data $\zeta_\circ$ is then defined as a family $(\zeta_t)_{t\in[0,T_0]}$ of random distributions on $\Dd$ such that, for all $t\in[0,T_0]$ and $h\in C_c^\infty(\Dd)$,
\begin{equation}\label{eq:mild-fluc-Vlas}
\langle\zeta_t,h\rangle=\langle\zeta_\circ,U_{0,t}h\rangle+\sqrt{2\alpha}\int_0^t\!\int_\Dd\sqrt{f_s(z)}\,\nabla_vU_{s,t}h(z)\cdot\xi(ds,dz),
\end{equation}
where $U_{s,t}$ is the backward propagator generated by the dual linearized Vlasov operator
\[S_fh:=v\cdot\nabla_{x}h+\alpha\triangle_{v}h+(K\ast f)\cdot\nabla_{v}h-\int_\Dd K(x_*-x)\cdot\nabla_{v_*}\log f(z_*)h(z_*)f(z_*)\,dz_*,\]
that is,
\[\partial_sU_{s,t}h+S_{f_s}U_{s,t}h=0,\quad 0\le s\le t,\qquad U_{t,t}h=h.\]
\end{defin}

We refer to Lemma~\ref{lem:primal-lim} for well-posedness of this notion of mild solution. In a nutshell, the initial bound~\eqref{eq:zeta0-bndL2f} allows us to extend the pairing $h\mapsto\langle\zeta_\circ,h\rangle$ continuously to~$L^2(f_\circ)$ modulo constants, with values in $L^2(\Omega)$. The first term in~\eqref{eq:mild-fluc-Vlas} is then well-defined thanks to the weighted energy estimate for the backward propagator,
\begin{equation}\label{eq:energy-Ust}
\|U_{0,t}h-c\|_{L^2(f_\circ)}^2+2\alpha\int_0^t\int_\Dd|\nabla_vU_{s,t}h|^2f_s\,ds\lesssim \|h-c\|_{L^2(f_t)}^2,\qquad c\in\R,
\end{equation}
which holds under the regularity assumptions on $f$. The same estimate ensures that the stochastic integral in~\eqref{eq:mild-fluc-Vlas} is also well-defined with finite variance.

\medskip
In this setting, our second main result establishes the convergence of empirical fluctuations to the mild solution of~\eqref{eq:fluc-Vlas} in the sense of finite-dimensional distributions, that is, convergence of their joint laws at any finite collection of times. This holds under exactly the same assumptions on the interaction as for the mean-field limit, reaching in particular the full 2D Coulomb singularity, both repulsive and attractive, with or without diffusion, for as long as the mean-field solution remains sufficiently regular.

Unlike the mean-field limit in Theorem~\ref{th:main}, the convergence of fluctuations at several times requires a choice of consistent joint-time laws. Such laws are not specified by given solutions of the Liouville equation, and we thus introduce a notion of \emph{consistent multi-time extensions} of weak duality solutions; see Appendix~\ref{app:multitime}. An extension $\Gamma_N$ consists of finite nonnegative measures $\Gamma_N^{t_1,\ldots,t_m}\in\calM_+((\Dd^N)^m)$, which play the role of joint laws of the particle system at times $t_1,\ldots,t_m$. Like the one-time densities, these measures may have total mass less than one, consistent with a possible finite-time breakdown of the particle dynamics. Proposition~\ref{prop:multi-time} constructs a weak duality solution together with such an extension under the sole assumption $K\in L^1_\loc$. The following applies to any such extension without imposing a selection rule.

\begin{theor}\label{th:fluct}
Assume the hypotheses of Theorem~\ref{th:main} and assume that, under~$F_N^\circ$, the initial empirical fluctuation field $\zeta_N^\circ$ converges in law in $\Dc'(\Dd)$ to a random distribution~$\zeta_\circ$. Further assume the following refined convergence, which ensures in particular that~$\zeta_\circ$ satisfies~\eqref{eq:zeta0-bndL2f}: the Hoeffding components $(H_{N,m}^\circ)_{0\le m\le N}$ of $F_N^\circ$ with respect to~$f_\circ^{\otimes N}$, as defined in Section~\ref{sec:Hoeffding}, satisfy for some limit $(\bar H_m^\circ)_{m\ge0}$,
\begin{equation}\label{eq:init-conv}
\sum_{m=0}^N\Big\|\binom Nm^\frac12H_{N,m}^\circ-\bar H_m^\circ\Big\|_{L^2(f_\circ^{\otimes m})}^2\to0.
\end{equation}
Finally, assume that $F_N$ is equipped with a consistent multi-time weak-duality extension~$\Gamma_N$ in the sense of Appendix~\ref{app:multitime}.
Then, for all $k\ge1$, $s>d$, and~$0\le t_1<\ldots<t_k\le T_0$, we have
\[(\zeta_N(t_1),\ldots,\zeta_N(t_k))\Rightarrow(\zeta(t_1),\ldots,\zeta(t_k))\quad\text{in $H^{-s}(\Dd)^k$,\quad under $\Gamma_N^{t_1,\ldots,t_k}$},\]
where $\zeta$ is the unique mild solution of~\eqref{eq:fluc-Vlas} with initial data $\zeta_\circ$, and where the convergence is understood narrowly for the corresponding finite measures.
\end{theor}

An unusual feature of our analysis is that we identify limiting empirical fluctuations at the natural central-limit scale without first establishing any quantitative rate of convergence for particle marginals, or any uniform moment bound or trajectory tightness estimates for the fluctuation field. Indeed, our proof proceeds through a more general result on bounded observables: in Theorem~\ref{th:bounded-obs}, we characterize the limit of arbitrary uniformly bounded sequences of observables of the particle system whenever their rescaled Hoeffding components converge with respect to the reference measure $f^{\otimes N}$. The proof relies on the same dual hierarchical analysis as for the mean-field limit, now applied to a broader class of test functions. Applying this result to characteristic functionals and identifying the probabilistic limit through the dual hierarchy then yields Theorem~\ref{th:fluct}.

In this way, the above result identifies joint laws of fluctuations at finitely many times, but does not establish convergence in law of trajectories: tightness of empirical fluctuations in a suitable path space remains an open problem under our assumptions.
This is very unusual and surprising in the context of many-particle systems: more traditional methods for fluctuation limits~\cite{Sznitman-84,Sznitman-91,Wang-Zhao-Zhu-23,HZZ-26} typically proceed in exactly the opposite direction, first establishing uniform moment bounds and tightness for empirical fluctuations, before being able to identify the limiting process.

\begin{rems}\label{rem:fluct}\
\begin{enumerate}[(a)]
\item \emph{Tightness and moment bounds:}
While tightness of empirical fluctuations $(\zeta_N)_N$ in $C(0,T_0;\Dc'(\Dd))$ remains an open problem, the above implies, at any fixed time $t\in[0,T_0]$, tightness of $(\zeta_N(t))_N$ in~$H^{-s}(\Dd)$ for all $s>d$. In particular, there exists an increasing function $\phi:\R^+\to\R^+$, with $\phi(r)\uparrow\infty$ as $r\uparrow\infty$, such that
\begin{equation}\label{tightzetaN}
\sup_N\E_{F_N(t)}\big[\phi(\|\zeta_N(t)\|_{H^{-s}(\Dd)})\big]<\infty.
\end{equation}
As $\phi$ may grow arbitrarily slowly, however, this does not yield any uniform moment bound for $\zeta_N$, nor corresponding bounds on correlation functions.

\smallskip\item \emph{Convergence rate for the empirical measure:}
The fixed-time tightness in~\eqref{tightzetaN} implies the following central-limit-scale estimate: for all $t\in[0,T_0]$ and $s>d$,
\[\mu_N(t)-f(t)=O_\PP(N^{-\frac12})\qquad\text{in $H^{-s}(\Dd)$, under $F_N(t)$.}\]
More precisely, for every $\eta>0$, there exists $C_{\eta,t,s}<\infty$ such that
\[\limsup_{N\uparrow\infty}\PP_{F_N(t)}\Big[N^\frac12\|\mu_N(t)-f(t)\|_{H^{-s}(\Dd)}>C_{\eta,t,s}\Big]\le\eta.\]
This provides an optimal convergence rate for the empirical measure in probability, but does not imply a corresponding deterministic convergence rate for marginals. Indeed, for the first marginal, we have
\[F_{N,1}(t)-f(t)=\E_{F_N(t)}\big[\mu_N(t)-f(t)\big]+\Big(\int_{\Dd^N}F_N(t)-1\Big)f(t),\]
and~\eqref{tightzetaN} does not give the uniform integrability needed to control the first term at order $N^{-1/2}$, even when the microscopic dynamics conserves mass and the second term vanishes.
A quantitative version of the dual hierarchical method was developed in an earlier version of~\cite{BDJ-24} (see \verb|arXiv:2402.04695v2|), but it would yield at best a much weaker rate $O(N^{-\theta})$, for some small $\theta>0$, for the convergence of marginals; we do not pursue this here.

\smallskip\item \emph{Interpretation of initial assumptions:}
The refined initial convergence assumption~\eqref{eq:init-conv} admits a natural interpretation in terms of Wiener-It\^o chaos. Let $W_\circ$ denote the isonormal Gaussian field over
$L^2_0(f_\circ):=\{h\in L^2(f_\circ):\int_\Dd hf_\circ=0\}$,
which is the limit in law of the empirical fluctuation field~$\zeta_N^\circ$ under the reference product measure~$f_\circ^{\otimes N}$. As explained in Remark~\ref{rem:fluct-init-explan}, the convergence~\eqref{eq:init-conv} is equivalent to
\[\qquad\Big(\frac{F_N^\circ}{f_\circ^{\otimes N}},\zeta_N^\circ\Big)\Rightarrow(H_\circ,W_\circ)\quad\text{under $f_\circ^{\otimes N}$},\qquad\int_{\Dd^N}\frac{|F_N^\circ|^2}{f_\circ^{\otimes N}}\to\E[|H_\circ|^2],\]
for some nonnegative random variable $H_\circ\in L^2(\sigma(W_\circ))$ with $\E[H_\circ]=1$. In this reformulation, $H_\circ$ is the Radon-Nikodym derivative of the law of the limiting initial fluctuation field~$\zeta_\circ$ with respect to the law of~$W_\circ$, and the limiting rescaled Hoeffding components $(\bar H_m^\circ)_m$ coincide with the Wiener-It\^o chaos components of $H_\circ$. In particular, the initial fluctuations need not be Gaussian: the condition~\eqref{eq:init-conv} allows for general $L^2$ changes of measure of the reference Gaussian field $W_\circ$.
This interpretation also explains the initial bound~\eqref{eq:zeta0-bndL2f}: by Cauchy-Schwarz and a Gaussian moment identity, for all $h\in C^\infty_c(\Dd)$,
\begin{eqnarray*}
\qquad\E[|\langle \zeta_\circ,h\rangle|^2]&=&\E[H_\circ|\langle W_\circ,h\rangle|^2]\\
\qquad&\le&\|H_\circ\|_{L^2}\E[\langle W_\circ,h\rangle^4]^\frac12
~=~\sqrt3\|H_\circ\|_{L^2}\int_\Dd\Big|h-\int_\Dd h f_\circ\Big|^2 f_\circ.
\end{eqnarray*}
\end{enumerate}
\end{rems}

\subsection{Extensions to other interacting systems}\label{sec:extensions}
The dual hierarchical method does not require any structural assumption on the interaction kernel $K$, and is likewise not restricted to the kinetic particle system~\eqref{eq:manyparticles} that we focus on in this work. As in~\cite{BDJ-24}, the mean-field analysis extends for instance to more general phase-space interactions, including velocity-alignment models such as the Cucker-Smale system, and the kinetic regularization exploited in the present work would similarly enlarge the range of admissible singularities in this case compared to~\cite{BDJ-24}.

The method also applies to first-order (inertialess) particle systems. In the deterministic first-order setting, however, velocity averaging is not available and we thus expect no improvement over~\cite{BDJ-24}. For systems with spatial diffusion, by contrast, parabolic regularization would allow us to extend the admissible range of singularities substantially further.

As we have shown, an additional strength of the dual hierarchical method is that the same analysis underlying the mean-field limit also suffices to describe fluctuations. Combined with the extensions developed in~\cite{BDJ-24}, this allows us to adapt our fluctuation results to general phase-space interactions and to first-order particle systems, at least under the same $L^2$-based assumptions as in the mean-field results of~\cite[Section~1.2]{BDJ-24}. Of particular interest are Cucker-Smale alignment models with singular communication kernels, for which the fluctuation theory has remained largely unexplored. Moreover, kinetic regularization mechanisms would further enlarge the admissible range of singularities. Although these extensions considerably broaden the scope of our results, we leave their detailed treatment aside here for brevity.

Another natural application concerns the Kac-type particle approximation of the spatially homogeneous Landau equation. Our dual hierarchical approach was recently adapted in~\cite{Feng-Wang-26b} to this setting, establishing propagation of chaos even for Coulomb interactions. In a forthcoming work, we further extend this result to characterize empirical fluctuations around the Landau equation for this particle model.

\subsection{Dual hierarchical approach and fluctuation viewpoint}\label{sec:method}

Rather than estimating the forward BBGKY hierarchy of marginals, which leads to a problematic loss of velocity derivatives, the dual hierarchical approach to mean-field limits that we introduced with D.~Bresch in~\cite{BDJ-24} proceeds by  analyzing the backward Liouville evolution of bounded observables through their Hoeffding components, or so-called dual correlations. An important feature, developed here, is that the same approach allows us to further capture fluctuations beyond the leading-order mean-field limit.

To explain this point, let us briefly set aside the hierarchy and compare our use of duality with other approaches to mean-field limits based on the Liouville equation. Following~\cite{Grunbaum-71}, the $N$-particle law~$F_N$ induces a probability distribution on empirical measures through the map
\[\mu_N(z_1,\ldots,z_N)=\frac1N\sum_{i=1}^N\delta_{z_i},\]
and one may then approach the mean-field limit by comparing the microscopic evolution of functionals of $\mu_N$ with the evolution induced by the nonlinear mean-field flow. This is the viewpoint underlying the generator method of~\cite{Mischler-Mouhot-13,Mischler-Mouhot-Wennberg-15}. For illustration, consider the deterministic case~$\alpha=0$. For $h\in C_c^\infty(\Dd)$ and a smooth bounded function $\Phi$, testing the Liouville equation~\eqref{eq:Liouville} with $\Phi(\langle\mu_N,h\rangle)$ formally gives
\[\frac{d}{dt}\int_{\Dd^N}\Phi(\langle\mu_N,h\rangle)\,F_N=\int_{\Dd^N}\Phi'(\langle\mu_N,h\rangle)\,\Big(\big\langle\mu_N,v\cdot\nabla_xh+(K\ast f)\cdot\nabla_vh\big\rangle+E_N^h\Big)\,F_N,\]
where we recognize the mean-field operator on $h$ and where $E_N^h$ accounts for the discrepancy between microscopic and mean-field forces,
\[E_N^h:=\frac1{N(N-1)}\sum_{i\ne j}^N\big(K(x_i-x_j)-(K\ast f)(x_i)\big)\cdot\nabla_vh(z_i).\]
The latter satisfies $E_N^h=O(N^{-1/2})$ in $L^2(f^{\otimes N})$ when $K\in L^2_\loc$, and $E_N^h=o(1)$ in~$L^1(f^{\otimes N})$ when only~$K\in L^1_\loc$. Controlling its contribution under the interacting law $F_N$, however, is a different matter. To this aim, the method of~\cite{Mischler-Mouhot-13,Mischler-Mouhot-Wennberg-15} exploits some form of duality: it proceeds by transporting observables by the limiting nonlinear mean-field flow and by differentiating their pairing with $F_N$; the resulting discrepancy between $N$-particle and mean-field generators is then controlled through consistency and stability estimates. Another approach is the relative-entropy method of~\cite{Jabin-Wang-16}, which controls a related discrepancy term using exponential concentration under the reference product law.

This viewpoint extends naturally to fluctuations: indeed, $F_N$ also induces a probability distribution on empirical fluctuations, through the map
\[\zeta_N(t,z_1,\ldots,z_N):=\sqrt N\big(\mu_N(z_1,\ldots,z_N)-f\big).\]
Testing the Liouville equation with $\Phi(\langle\zeta_N,h\rangle)$ and using the Vlasov equation for~$f$,
\[\frac{d}{dt}\int_{\Dd^N}\Phi(\langle\zeta_N,h\rangle)\,F_N=\int_{\Dd^N}\Phi'(\langle\zeta_N,h\rangle)\,\Big(\big\langle\zeta_N,v\cdot\nabla_xh+(K\ast f)\cdot\nabla_vh\big\rangle+\sqrt N E_N^h\Big)\,F_N.\]
The discrepancy $\sqrt NE_N^h$ now contributes at order $O(1)$: expanding $\mu_N=f+N^{-\frac12}\zeta_N$,
\begin{multline*}
\Big(1-\frac1N\Big)\sqrt N E_N^h=\int_\Dd(K\ast\zeta_N)(x)\cdot\nabla_vh(z)f(z)\,dz\\[-1mm]
+\frac1{\sqrt N}\iint_{z\ne z'} K(x-x')\cdot\nabla_vh(z)\,\zeta_N(dz)\zeta_N(dz')+O(N^{-\frac12}).
\end{multline*}
The first term yields the linearized interaction in the fluctuating Vlasov equation~\eqref{eq:fluc-Vlas}, while controlling the quadratic remainder requires nontrivial estimates on~$\zeta_N$ under $F_N$. In the spirit of~\cite{Mischler-Mouhot-13,Mischler-Mouhot-Wennberg-15}, this suggests comparing generators at the fluctuation level: this was recently performed in~\cite{CecchinNikolaev-25} for interacting diffusions, using the available quantitative estimates under $F_N$. Yet, for general unbounded forces as we consider here, no sufficient estimate on~$\zeta_N$ is a priori available under $F_N$ to pursue this route.

Our dual hierarchical method exploits duality in a radically different way. Rather than transporting observables by the limiting mean-field semigroup and estimating the resulting discrepancy between generators, we study their backward evolution under the exact $N$-particle Liouville equation. Given a bounded symmetric final observable $\Phi_N^T$ at time $T$, for instance a functional of $\mu_N$ or $\zeta_N(T)$, let $\Phi_N$ solve the backward Liouville equation. Our starting point is the exact duality relation
\begin{equation}\label{eq:duality-0}
\int_{\Dd^N}\Phi_N^TF_N(T)=\int_{\Dd^N}\Phi_N(0)F_N^\circ.
\end{equation}
In this form, the problem is reduced to studying the backward-evolved observable $\Phi_N(0)$. The essential advantage is that the uniform boundedness of observables is preserved by the backward evolution. By orthogonality, it leads to a priori $L^2$ estimates for their Hoeffding components, or dual correlations, in the natural fluctuation scaling; see Sections~\ref{sec:Hoeffding}--\ref{sec:apriori}. The general principle is reminiscent of classical duality arguments in the DiPerna-Lions theory for transport equations~\cite{DiLi}: a rough forward evolution is best studied through its adjoint equation when the latter enjoys stronger a priori estimates. In our case, whereas the $N$-particle distribution $F_N$ is only controlled in $L^1(\Dd^N)$, duality reduces the problem to considering the adjoint equation on the space $L^\infty(\Dd^N)$ of bounded observables, with resulting $L^2$ bounds on Hoeffding components.

When $F_N^\circ=f_\circ^{\otimes N}$, the right-hand side of~\eqref{eq:duality-0} is the expectation of~$\Phi_N(0)$ under~$f_\circ^{\otimes N}$, that is, its zeroth Hoeffding component. In order to obtain the mean-field limit, we might thus try to track only the evolution of this zeroth component. Yet, it does not satisfy a closed equation: from the Liouville equation for $\Phi_N$ and the Vlasov equation for $f$,
\begin{equation}\label{eq:dtC_N0}
\frac{d}{dt}\int_{\Dd^N}\Phi_N\,f^{\otimes N}=\int_{\Dd^N}\bigg(\frac1{N-1}\sum_{i\ne j}^NK_{f}(z_i,z_j)\bigg)\,\Phi_N\, f^{\otimes N},
\end{equation}
where
\[K_f(z,z'):=\big(K(x-x')-(K\ast f)(x)\big)\cdot\nabla_v\log f(z).\]
As $K_f$ is centered in each variable with respect to $f$, the sum in brackets in~\eqref{eq:dtC_N0} is a degenerate $U$-statistics and is bounded in $L^2(f^{\otimes N})$ if $K_f\in L^2(f^{\otimes2})$, but it does not generally vanish as $N\uparrow\infty$. In fact, \eqref{eq:dtC_N0} couples the zeroth Hoeffding component of~$\Phi_N$ to its second component. Even for a simple final observable $\Phi_N^T=\Phi(\mu_N)$, which has a deterministic limit $\Phi(f_T)$ under $f_T^{\otimes N}$, nontrivial higher-order components may thus appear along the evolution, and closing the argument requires controlling the full hierarchy. As a consequence, at no extra cost, we may as well consider more general final observables of empirical fluctuations, $\Phi_N^T=\Phi(\zeta_N(T))$, and the same hierarchical estimates used for the mean-field limit allow us to identify limiting fluctuations at the same time.

The natural $L^2$ control of Hoeffding components explains the square-integrability threshold in our first work~\cite{BDJ-24}: interaction terms such as~\eqref{eq:dtC_N0} involve singular kernels $K_f$ paired with a priori $L^2$ quantities. The main analytic novelty of the present contribution is to get beyond this threshold. To this aim, the key observation is that dual correlations are driven by the linearized mean-field equation, with interaction terms involving singular forces multiplied by \emph{velocity averages} of other correlations. We may then exploit kinetic regularization of these averages, using velocity averaging when $\alpha=0$ and hypoellipticity when $\alpha>0$, and iterate the resulting fractional spatial regularity estimates within the hierarchy. Justifying such estimates is substantially complicated by the different singular error terms that appear in the finite-$N$ hierarchy; see Propositions~\ref{prop:bnd-diff} and~\ref{prop:bnd-nodiff}.

Beyond the $L^2$ threshold, a further difficulty is uniqueness for the limiting dual hierarchy. The available kinetic estimates provide regularity in one particle variable at a time, whereas a direct PDE argument for uniqueness would require mixed regularity in several variables simultaneously. This is reminiscent of the difficulty originally met in~\cite{Bardos-Golse-Mauser-00} in the hierarchical derivation of the quantum mean-field limit for bosons with Coulomb interactions. We overcome this issue by exploiting the explicit triangular structure of the limiting hierarchy; see Section~\ref{sec:uniqueness}.

Finally, we identify the limiting dual hierarchy as the adjoint of the hierarchy governing the Wiener-It\^o chaos coefficients of the fluctuating Vlasov equation~\eqref{eq:fluc-Vlas}; see Section~\ref{sec:fluct-vlas}. This allows us to draw the link between our hierarchical characterization of limiting observables and the expected description of empirical fluctuations.


\section{Dual hierarchical method}\label{sec:dual-method}
We start by recalling the dual hierarchical approach to mean-field limits, which we introduced in~\cite{BDJ-24} with D. Bresch. The method proceeds by reformulating the mean-field limit in terms of the backward Liouville evolution on the space of bounded observables, then applying the Hoeffding chaos decomposition with respect to the reference product law $f^{\otimes N}$, and finally noting that the obstruction related to losses of velocity derivatives in the BBGKY hierarchy is compensated by the fluctuation scaling; see~\cite[Section~1.3]{BDJ-24}.

\subsection{Hoeffding chaos decomposition}\label{sec:Hoeffding}
Given a symmetric function $G_N\in L^2(f^{\otimes N})$, we consider its Hoeffding components with respect to $f^{\otimes N}$: for~$0\le m\le N$,
\[\Pi_{N,m}[G_N]:=\big((\Id-\pi_f)^{\otimes m}\otimes \pi_f^{\otimes N-m}\big)\,G_N,\]
in terms of the projection
\[\pi_fh:=\int_\Dd hf.\]
By definition, $\Pi_{N,m}[G_N]$ belongs to the space
\begin{multline*}
L^2_0(f^{\otimes m}):=\bigg\{G_m\in L^2(f^{\otimes m}):G_m~\text{symmetric}\\[-2mm]
\text{and}~\int_\Dd G_m(z_1,\ldots,z_m)\,f(z_j)\,dz_j=0\quad\text{for all $1\le j\le m$}\bigg\}.
\end{multline*}
The Hoeffding decomposition of $G_N$ then reads
\begin{equation}\label{eq:hoeffd-chaos}
G_N(z_1,\ldots,z_N)\,=\,\sum_{m=0}^NI_{N,m}(\Pi_{N,m}[G_N]),\qquad I_{N,m}(G_m):=\!\!\sum_{\sigma\subset [N]:\,\sharp\sigma=m}G_m(z_{\sigma}),
\end{equation}
with the shorthand notation $[N]=\{1,\ldots,N\}$ and $z_\sigma=(z_{i_1},\ldots,z_{i_\ell})$ for $\sigma=\{i_1,\ldots,i_\ell\}$.
By orthogonality between Hoeffding components,
\begin{equation}\label{eq:orthog}
\int_{\Dd^N}|G_N|^2f^{\otimes N}=\sum_{m=0}^N\binom{N}{m}\int_{\Dd^m}|\Pi_{N,m}[G_N]|^2f^{\otimes m}.
\end{equation}

Applied to the relative density $F_N/f^{\otimes N}$ of the $N$-particle distribution with respect to the mean-field distribution $f^{\otimes N}$,
the Hoeffding components
\begin{equation}\label{eq:def-HNm}
H_{N,m}:=\Pi_{N,m}\Big[\frac{F_N}{f^{\otimes N}}\Big]~\in L^2_0(f^{\otimes m})
\end{equation}
coincide with the linear correlation functions first considered in~\cite{PPS-19} to characterize corrections to the mean-field approximation. In terms of marginals~\eqref{eq:def-marginal}, if $F_N$ has unit mass,
\[H_{N,0}=1,\quad H_{N,1}f=F_{N,1}-f,\quad H_{N,2}f^{\otimes2}=F_{N,2}-2F_{N,1}\otimes_\sigma f+f\otimes f,\]
and so on, where $\otimes_\sigma$ stands for symmetric tensor product.
In these terms, the mean-field limit simply amounts to showing $H_{N,m}\to0$ for all~$m\ge1$. As the norm $\int_{\Dd^N}|F_N|^2/f^{\otimes N}$ is not conserved by the Liouville flow, however, no useful estimate on these correlation functions can be deduced from identity~\eqref{eq:orthog}.

In contrast, for bounded observables, since the~$L^\infty$ norm is non-increasing under the backward Liouville evolution, we can use~\eqref{eq:orthog} in the form
\begin{equation}\label{eq:orthog-re}
\int_{\Dd^m}|\Pi_{N,m}[\Phi_N]|^2f^{\otimes m}\le\binom{N}{m}^{-1}\|\Phi_N\|_{L^\infty}^2\le C_mN^{-m}\|\Phi_N\|_{L^\infty}^2.
\end{equation}
This is what motivates our duality argument: passing to the dual, cf.~\eqref{eq:duality-0}, we can reduce to studying the backward Liouville evolution of bounded observables, and the above then provides nontrivial a priori estimates at the relevant fluctuation scale.

\subsection{Dual reformulation}
We recall the following result from~\cite[Proposition~2.2]{BDJ-24}, which is the basis for our duality approach: the mean-field limit is reduced to proving that bounded backward solutions $\Phi_N$ of the dual Liouville equation with almost deterministic final data remain almost deterministic, uniformly in time. More precisely,  for final data $\Phi_N^T$ satisfying
\[\|\Phi_N^T\|_{L^\infty}\le1,\qquad\lim_{N\uparrow\infty}\|\Phi_N^T-1\|_{L^2(f_T^{\otimes N})}=0,\]
any associated bounded backward Liouville solution $\Phi_N$ still satisfies
\[\sup_{[0,T]}\|\Phi_N-1\|_{L^2(f^{\otimes N})}\to0.\]
As we have shown in~\cite{BDJ-24}, by conservation laws, it actually suffices to show that the zeroth Hoeffding component $\Pi_{N,0}\Phi_N=\int_{\Dd^N}\Phi_Nf^{\otimes N}$ converges to $1$; see~\eqref{eq:conv-CN2-0} below.
In comparison, characterizing fluctuations will simply require to consider a broader range of (non-deterministic) final data; cf.~Section~\ref{sec:general-test}.

\begin{prop}[Dual reformulation; see~\cite{BDJ-24}]\label{lem:duality}
Let $\alpha\ge0$ and $K\in L^1_\loc(\R^d)^d$. Let~$F_N\in L^\infty(\R^+;L^1_+(\Dd^N))$ be a symmetric global weak duality solution of the Liouville equation~\eqref{eq:Liouville}, in the sense of~\cite[Appendix]{BDJ-24}, with initial data $F_N^\circ\in \Pc\cap L^1(\Dd^N)$ satisfying the chaoticity assumption~\eqref{eq:initial} for some $f_\circ\in\Pc\cap L^\infty(\Dd)$. Let~$f$ be a weak solution of the Vlasov equation~\eqref{eq:Vlasov} with initial data $f_\circ$, and assume that, for some $T_0>0$,
\[f\in L^\infty([0,T_0]\times\Dd),\qquad \rho_f\in L^\infty([0,T_0]\times\R^d),\qquad\nabla_vf\in L^1([0,T_0]\times\Dd).\]
Given $T\in(0,T_0]$ and symmetric $\Phi_N^T\in L^\infty(\Dd^N)$ with
\begin{eqnarray}
\|\Phi_N^T\|_{L^\infty(\Dd^N)}&\le&1,\label{eq:conv-PhiNT-ass}\\
\lim_{N\uparrow\infty}\|\Phi_N^T-1\|_{L^2(f_T^{\otimes N})}&=&0,\label{eq:conv-PhiNT-ass-MF}
\end{eqnarray}
let $\Phi_N\in L^\infty([0,T]\times\Dd^N)$ be a symmetric global bounded weak solution of the backward dual Liouville equation on $[0,T]$,
\begin{equation}\label{eq:Liouville-dual}
\partial_t\Phi_N+\sum_{i=1}^N(v_i\cdot\nabla_{x_i}+\alpha\triangle_{v_i})\Phi_N+\frac1{N-1}\sum_{i\ne j}^NK(x_i-x_j)\cdot\nabla_{v_i}\Phi_N\,=\,0,
\end{equation}
with final data $\Phi_N|_{t=T}=\Phi_N^T$ and with
\begin{equation}\label{eq:conser-Linfty}
\sup_{[0,T]}\|\Phi_N\|_{L^\infty(\Dd^N)}\le1.
\end{equation}
Assume that for any such $\Phi_N$ we have
\begin{equation}\label{eq:conv-CN2-0}
\sup_{[0,T]}\Big|\int_{\Dd^N}\Phi_N(t)f(t)^{\otimes N}-1\Big|\,\to\,0\quad\text{as $N\uparrow\infty$}.
\end{equation}
Then, for all $k\ge1$ and $t\in[0,T_0]$,
\begin{equation}\label{eq:concl-MFL}
F_{N,k}(t)\,\to\,f_t^{\otimes k}\quad\text{narrowly as finite measures on $\Dd^k$.}
\end{equation}
\end{prop}

\subsection{A priori estimates}\label{sec:apriori}
Henceforth, we assume that $F_N,f$ are as in Proposition~\ref{lem:duality},
we fix $T\in [0,T_0]$ and a symmetric $\Phi_N^T\in L^\infty(\Dd^N)$ satisfying~\eqref{eq:conv-PhiNT-ass}, and we consider a global bounded weak solution $\Phi_N\in L^\infty([0,T]\times\Dd^N)$ of the dual Liouville equation~\eqref{eq:Liouville-dual} with final data $\Phi_N|_{t=T}=\Phi_N^T$ and with the conservation property~\eqref{eq:conser-Linfty}. This framework is identical for deriving the mean-field limit and for obtaining the fluctuations: the only difference is in the choice of the final data~$\Phi^T_N$, which we shall further assume to be almost deterministic in the sense of~\eqref{eq:conv-PhiNT-ass-MF} for the mean-field limit.

In this setting, we shall study the associated Hoeffding components, or dual correlation functions,
\[C_{N,m}:=\Pi_{N,m}[\Phi_N],\qquad 0\le m\le N.\]
By orthogonality~\eqref{eq:orthog-re}, using~\eqref{eq:conser-Linfty}, we obtain the fluctuation-scale a priori estimates
\begin{equation}\label{eq:est-L2-CNm}
\sup_{[0,T]}\|C_{N,m}\|_{L^2(f^{\otimes m})}\,\le\,\binom{N}{m}^{-\frac12}\le m^\frac m2N^{-\frac m2},\qquad 0\le m\le N.
\end{equation}
In particular, $C_{N,2}=O(N^{-1})$, which is of course not enough to establish the required convergence~\eqref{eq:conv-CN2-0}
and to deduce the mean-field limit even if $K\in L^2_\loc(\R^d)$. Indeed, this can only hold for almost-deterministic final data~\eqref{eq:conv-PhiNT-ass-MF}, which we have not used yet.

In view of these a priori bounds~\eqref{eq:est-L2-CNm}, we naturally consider the rescaled dual correlation functions
\begin{equation}\label{eq:barCNn}
\bar C_{N,m}\,:=\,\binom{N}{m}^\frac12C_{N,m},\qquad0\le m\le N.
\end{equation}
Since $\Phi_N$ is bounded in $L^\infty$, cf.~\eqref{eq:conser-Linfty}, we can interpolate and deduce the following more general mixed a priori estimates.

\begin{lem}[Mixed a priori estimates]\label{lem:apriori-L2}
For all $0\le k\le m\le N$ and $q\ge2$,
\begin{equation*}
\sup_{[0,T]}\|\bar C_{N,m}\|_{L^q(f^{\otimes k},L^2(f^{\otimes m-k}))}\,\le\,\Big(4e\frac Nm\Big)^{k(\frac12-\frac1q)}.
\end{equation*}
\end{lem}

\begin{proof}
By interpolation with~\eqref{eq:est-L2-CNm}, it suffices to prove the result for $q=\infty$.
We argue by duality. For a symmetric test function $g_m\in L^2(f^{\otimes m})$, the definition of correlations reads
\[\int_{\Dd^m}g_mC_{N,m}f^{\otimes m}=\int_{\Dd^N}((\Id-\pi_f)^{\otimes m}g_m)(z_{[m]})\,\Phi_N(z_{[N]})\,f^{\otimes N}(z_{[N]})\,dz_{[N]}.\]
Given $1\le k\le m$, this can be rewritten as follows by symmetry,
\begin{multline*}
\int_{\Dd^m}g_mC_{N,m}f^{\otimes m}\\
=\int_{\Dd^N}\bigg(\binom{N-k}{m-k}^{-1}\!\!\!\sum_{\sigma\subset\{k+1,\ldots,N\}\atop\sharp\sigma=m-k}((\Id-\pi_f)^{\otimes m}g_m)(z_{[k]},z_\sigma)\bigg)\Phi_N(z_{[N]})\,f^{\otimes N}(z_{[N]})\,dz_{[N]}.
\end{multline*}
Hence, by the Cauchy-Schwarz inequality and by~\eqref{eq:conser-Linfty},
\[\Big|\int_{\Dd^m}g_mC_{N,m}f^{\otimes m}\Big|\le\bigg\|\binom{N-k}{m-k}^{-1}\!\!\!\sum_{\sigma\subset\{k+1,\ldots,N\}\atop\sharp\sigma=m-k}((\Id-\pi_f)^{\otimes m}g_m)(z_{[k]},z_\sigma)\bigg\|_{L^1(f^{\otimes k};L^2(f^{\otimes N-k}))},\]
and thus, by orthogonality,
\[\Big|\int_{\Dd^m}g_mC_{N,m}f^{\otimes m}\Big|\le\binom{N-k}{m-k}^{-\frac12}\|(\Id-\pi_f)^{\otimes m}g_m\|_{L^1(f^{\otimes k};L^2(f^{\otimes m-k}))}.\]
The conclusion for $q=\infty$ then follows by the arbitrariness of $g_m$ and by the elementary bound $\binom{N}{m}\binom{N-k}{m-k}{}^{-1}\le(e\frac Nm)^{k}$.
\end{proof}

\subsection{Hierarchy for dual correlations}
As the above a priori bounds~\eqref{eq:est-L2-CNm} do not suffice for the desired convergence~\eqref{eq:conv-CN2-0} to conclude the mean-field limit, we proceed by studying in more detail the actual dynamics of dual correlations.
To this end, we consider the BBGKY-type hierarchy that they satisfy, which can be obtained directly from the dual Liouville equation~\eqref{eq:Liouville-dual} for~$\Phi_N$; see~\cite[Appendix~B]{BDJ-24}. We state it here for the rescaled dual correlations~\eqref{eq:barCNn}.

We emphasize that most of the interaction operators appearing in the hierarchy contain velocity derivatives, and their iteration would thus ordinarily generate an uncontrollable cascade of derivative losses, which is the usual obstruction for classical kinetic hierarchies. Yet, our central observation in~\cite[Section~1.3]{BDJ-24} was that the fluctuation scaling~\eqref{eq:barCNn} exactly compensates for derivative losses as~$N\uparrow\infty$: after rescaling, velocity derivatives are confined to the remainder term $R_{N,m}$ below, which is formally of lower order.

\begin{lem}[Hierarchy for dual correlations; see~\cite{BDJ-24}]\label{lem:hier}
For all $0\le m\le N$, we have in the weak sense on $[0,T]$,
\begin{equation}\label{eq:lemhier}
(\partial_t+L_f^m)\bar C_{N,m}
=R_{N,m}
+\sqrt{(m+1)(m+2)}\Big(\frac{(N-m)(N-m-1)}{(N-1)^2}\Big)^\frac12\,K_f[\bar C_{N,m+2}],
\end{equation}
where $L_f^m$ stands for the following linearized mean-field operator on the $m$-particle space,
\begin{eqnarray*}
L_f^m&:=&\sum_{j=1}^m\Id^{\otimes j-1}\otimes L_f\otimes\Id^{m-j},\\
L_fh&:=& v\cdot\nabla_{x}h+\alpha\triangle_{v}h+(K\ast f)(x)\cdot\nabla_{v}h-\int_\Dd K_f(z_*,z)h(z_*)f(z_*)\,dz_*,
\end{eqnarray*}
where $K_f[\bar C_{N,m+2}]$ is defined as
\begin{equation}\label{eq:def-VC}
K_f[\bar C_{N,m+2}]\,=\,\int_{\Dd^2}K_f(z_*,z_*')\,\bar C_{N,m+2}(\cdot,z_*,z_*')\,f(z_*)f(z_*')\,dz_*dz_*',
\end{equation}
where the remainder term $R_{N,m}$ is given explicitly by
\begin{multline*}
R_{N,m}\,=\,\Big(\frac{N-m+1}{m(N-1)^2}\Big)^\frac12\sum_{i\ne j}^m(K\ast f)(x_i)\cdot\nabla_{v_i}\bar C_{N,[m]\setminus\{j\}}\\
-\Big(\frac{N-m+1}{m(N-1)^2}\Big)^\frac12\sum_{i\ne j}^m\int_\Dd K_f(z_{*},z_j)\bar C_{N,[m]\cup\{*\}\setminus\{i,j\}}\,f(z_{*})\,dz_{*}\\
-\Big(\frac{N-m+1}{m(N-1)^2}\Big)^\frac12\sum_{i\ne j}^mK(x_i-x_j)\cdot\nabla_{v_i}\bar C_{N,[m]\setminus\{j\}}\\
+\frac{m-1}{N-1}\sum_{i=1}^m(K\ast f)(x_i)\cdot\nabla_{v_i}\bar C_{N,m}
-\frac1{N-1}\sum_{i\ne j}^mK(x_i-x_j)\cdot\nabla_{v_i}\bar C_{N,m}\\
-\frac{m-1}{N-1}\sum_{j=1}^m\int_\Dd K_f(z_*,z_j)\bar C_{N,[m]\cup\{*\}\setminus\{j\}}\,f(z_*)\,dz_*\\
+\frac1{N-1}\sum_{i\ne j}^m\int_\Dd K(x_i-x_*)\cdot\nabla_{v_i}\bar C_{N,[m]\cup\{*\}\setminus\{j\}}\,f(z_*)\,dz_*\\
-\frac1{N-1}\sum_{i\ne j}^m\int_\Dd K_f(z_*,z_j)\,\bar C_{N,[m]\cup\{*\}\setminus\{i\}}\,f(z_*)\,dz_*\\
+\frac1{N-1}\sum_{i\ne j}^m\int_{\Dd^2}K_f(z_*,z_*')\,\bar C_{N,[m]\cup\{*,*'\}\setminus\{i,j\}}\,f(z_*)f(z_*')\,dz_*dz_*'\\
+\Big(\frac{(m+1)(N-m)}{(N-1)^2}\Big)^\frac12\sum_{j=1}^m\int_\Dd K_f(z_*,z_j)\,\bar C_{N,[m]\cup\{*\}}\,f(z_*)\,dz_*\\
-\Big(\frac{(m+1)(N-m)}{(N-1)^2}\Big)^\frac12\sum_{i=1}^m\int_\Dd K(x_i-x_*)\cdot\nabla_{v_i}\bar C_{N,[m]\cup\{*\}}\,f(z_*)\,dz_*\\
-2\Big(\frac{(m+1)(N-m)}{(N-1)^2}\Big)^\frac12\sum_{i=1}^m\int_{\Dd^2}K_f(z_*,z_*')\,\bar C_{N,[m]\cup\{*,*'\}\setminus\{i\}}\,f(z_*)f(z_*')\,dz_*dz_*',
\end{multline*}
and where for notational convenience we have set $\bar C_{N,m}\equiv0$ for $m<0$ and $m>N$.
\end{lem}

The hierarchy~\eqref{eq:lemhier} shows that the evolution of dual correlations is governed by the linearized mean-field operator $L_f^m$, while the main interaction term $K_f[\bar C_{N,m+2}]$ involves a velocity average of $\bar C_{N,m+2}$ integrated against the singular interaction kernel $K(x_*-x_*')$; see~\eqref{eq:def-VC}. Since $L_f^m$ is a kinetic transport operator, with additional velocity diffusion when $\alpha>0$, this suggests exploiting kinetic regularization to gain fractional spatial regularity and thereby control the singular interaction terms beyond the $L^2$ threshold of~\cite{BDJ-24}. When $\alpha=0$, the relevant regularity gain comes from velocity averaging; when $\alpha>0$, hypoellipticity provides spatial regularity of dual correlations themselves without relying on velocity averaging (see, however, Remark~\ref{rem:mfl}(c)). This is the topic of the next two sections. The application of kinetic regularity within the hierarchy involves some technical difficulties, notably the iteration of regularity estimates across orders and the treatment of the singular remainder $R_{N,m}$ in~\eqref{eq:lemhier}.


\section{Diffusive case: hierarchical estimates via hypoellipticity}

We start with the diffusive case $\alpha>0$. The analysis relies on the following hypoelliptic estimate, which is essentially classical~\cite{Bouchut-02}. We include a short proof in Appendix~\ref{app:kinetic-reg} for completeness, based on a direct Fourier calculation. Importantly, the estimate is uniform in the number of particle variables $m$, and allows for source terms with negative Sobolev regularity.

We use the standard Fourier multiplier notation $|\nabla_{x_{[m]}}|$ and $\langle\nabla_{x_{[m]}}\rangle$, with symbols $|\xi_{[m]}|=(\sum_{i=1}^m|\xi_i|^2)^{1/2}$ and $(1+|\xi_{[m]}|^2)^{1/2}$, respectively, where $\xi_{[m]}=(\xi_1,\ldots,\xi_m)$ is the dual variable corresponding to $x_{[m]}=(x_1,\ldots, x_m)$, and similarly for~$\nabla_{v_{[m]}}$.

\begin{lem}[Hypoelliptic regularity]\label{lem:hypoell}
Given $\alpha>0$ and $m\ge1$, let $g_m,e_m$ satisfy in the weak sense on~$[0,T]\times\Dd^m$,
\[\partial_tg_m+\sum_{i=1}^m(v_i\cdot\nabla_{x_i}-\alpha\triangle_{v_i})g_m=e_m.\]
Then, for all $0\le r\le\frac13$,
\begin{equation}
\int_0^T\||\nabla_{x_{[m]}}|^{\frac{1}3-r} g_m\|_{L^2}^2
\le C(T+1)\sup_{[0,T]}\|g_m\|_{L^2}^2+C\int_0^T\|\langle\nabla_{x_{[m]}}\rangle^{-r}\langle\nabla_{v_{[m]}}\rangle^{-1}e_m\|_{L^2}^2,
\end{equation}
for some constant $C$ only depending on $\alpha$.
In the case when $e_m=e_m^0+\Div_v(e_m^1)$, we further have for all $0\le r\le\frac13$ and~$0\le s\le\frac23$,
\begin{multline}
\int_0^T\||\nabla_{x_{[m]}}|^{(\frac23-s)\wedge(\frac13-r)}g_m\|_{L^2}^2\\[-2mm]
\le C(T+1)\sup_{[0,T]}\|g_m\|_{L^2}^2+C\int_0^T\|\langle\nabla_{x_{[m]}}\rangle^{-s}e_m^0\|_{L^2}^2+C\int_0^T\|\langle\nabla_{x_{[m]}}\rangle^{-r}e_m^1\|_{L^2}^2. \label{lem:hypoell-re}
\end{multline}
\end{lem}

Applying this result to the hierarchy~\eqref{eq:lemhier} for dual correlations, we establish the following uniform-in-$N$ regularity estimate. Note that we use the weight $f^{1-\delta}$ instead of~$\sqrt f$, with $1-\delta\ge\frac12$, to avoid boundedness requirements on $\nabla_v\log f$ itself, in line with the weighted regularity assumption~\eqref{eq:reg-req-2}.

\begin{prop}\label{prop:bnd-diff}
Let
\[\alpha>0,\qquad K\in L^{2d/(d+2\gamma)}_\loc(\R^d),\qquad\sup_{|x|\ge1}|K(x)|<\infty,\]
for some $0\le\gamma<\frac16$. Also assume that the mean-field solution $f$ satisfies the regularity conditions~\eqref{eq:reg-req-0}, \eqref{eq:reg-req-0b}, and~\eqref{eq:reg-req-2} for some $\delta\in(0,\frac12]$.
Then, for all~$0\le m\le N$,
\[\int_0^T\big\||\nabla_{x_1}|^{\gamma}((f^{\otimes m})^{1-\delta}\bar C_{N,m})\big\|_{L^2}^2\,\le\,C(T+1)(m+1)^c,\]
for some constant $c$ only depending on $\gamma$ and some constant $C$ further depending on~$\alpha,\delta$ and on assumptions on $K,f$.
\end{prop}

\begin{proof}
Smuggling the weight $(f^{\otimes m})^{1-\delta}$, with $\delta$ as in the regularity condition~\eqref{eq:reg-req-2}, and using the Vlasov equation for $f$, the equation~\eqref{eq:lemhier} for $\bar C_{N,m}$ can be reformulated as
\begin{equation}\label{eq:rewr-Cnn-fn}
\Big(\partial_t+\sum_{i=1}^mv_i\cdot\nabla_{x_i}+\alpha\sum_{i=1}^m\triangle_{v_i}\Big)((f^{\otimes m})^{1-\delta}\bar C_{N,m})\\
=A_m=:\sum_{j=1}^4A_m^j,
\end{equation}
where we have set
\begin{eqnarray*}
A_m^1&:=&(f^{\otimes m})^{1-\delta} R_{N,m},\\
A_m^2&:=&\sqrt{(m+1)(m+2)}\Big(\frac{(N-m)(N-m-1)}{(N-1)^2}\Big)^\frac12\,(f^{\otimes m})^{1-\delta}K_f[\bar C_{N,m+2}]\\
&&+(f^{\otimes m})^{1-\delta}\sum_{i=1}^m\int_\Dd K_f(z_*,z_i)\,\bar C_{N,m}(z_{[m]\setminus i},z_*)\,f(z_*)\,dz_*,\\
A_m^3&:=&-\sum_{i=1}^m\Div_{v_i}\Big(\big((K\ast f)(x_i)-2\alpha(1-\delta)(\nabla_{v}\log f)(z_i)\big)(f^{\otimes m})^{1-\delta}\bar C_{N,m}\Big),\\
A_m^4&:=&\alpha\delta(1-\delta) \Big(\sum_{i=1}^m |\nabla_v\log f(z_i)|^2\Big)(f^{\otimes m})^{1-\delta}\bar C_{N,m}.
\end{eqnarray*}
We now appeal to the hypoelliptic regularity in the form of Lemma~\ref{lem:hypoell} (but backward in time), recall the a priori $L^2$-estimate of Lemma~\ref{lem:apriori-L2}, and note that the boundedness of $f$ ensures $f^{1-\delta}\lesssim f^{1/2}$ since $\delta\leq 1/2$. We hence get for all $0\le r\le\frac13$,
\begin{equation}\label{eq:decomp-hypobound}
\int_0^T\|\langle\nabla_{x_1}\rangle^{\frac13-r}((f^{\otimes m})^{1-\delta}\bar C_{N,m})\|_{L^2}^2
\le C(T+1)+C\sum_{i=1}^4\int_0^T\|\langle\nabla_X\rangle^{-r}\langle\nabla_V\rangle^{-1}A_m^i\|_{L^2}^2,
\end{equation}
where we set for abbreviation $Z=(X,V):=(x_{[m]},v_{[m]})\in\Dd^m$.
From here, we split the proof into three steps: in the first two, we examine the different right-hand side terms in~\eqref{eq:decomp-hypobound}, and we conclude in the last step by buckling. In the sequel, $C$ stands for arbitrary constants only depending on $\gamma,\alpha,\delta$, and on assumptions on $K,f$, and we write $\lesssim$ for $\le C\times$.

\medskip\noindent
{\bf Step~1:} Proof that
\begin{multline*}
\sum_{i=2}^4\|\langle\nabla_V\rangle^{-1}A_m^i\|_{L^2}\\
\,\lesssim\,m+(m+1)\|\langle\nabla_{x_1}\rangle^{\gamma}((f^{\otimes m+2})^{1-\delta}\bar C_{N,m+2})\|_{L^2}
+m\|\langle\nabla_{x_1}\rangle^{\gamma}((f^{\otimes m})^{1-\delta}\bar C_{N,m})\|_{L^2}.
\end{multline*}
We start with the contribution of $A_m^2$. By definition, we can bound
\begin{multline}\label{eq:An2-decomp}
\|A_m^2\|_{L^2}
\,\lesssim\,
(m+1) \|(f^{\otimes m})^{1-\delta}K_f[\bar C_{N,m+2}]\|_{L^2}\\
+\Big\|(f^{\otimes m})^{1-\delta}\sum_{i=1}^m\int_\Dd K_f(z_*,z_i)\,\bar C_{N,m}(z_{[m]\setminus i},z_*)\,f(z_*)\,dz_*\Big\|_{L^2}.
\end{multline}
Recalling the definition of the operator $K_f[\cdot]$, cf.~\eqref{eq:def-VC}, we can decompose the first term as follows,
\begin{multline*}
\|(f^{\otimes m})^{1-\delta}K_f[\bar C_{N,m+2}]\|_{L^2}\\
\,\le\,\bigg\|\int_{\Dd^2}K(x'-x'')\cdot\nabla_v\log f(z')\,f(z')^\delta f(z'')^\delta \big((f^{\otimes m+2})^{1-\delta}\bar C_{N,m+2}\big)(\cdot,z',z'')\,dz'dz''\bigg\|_{L^2}\\
+\|K\ast f\|_{L^\infty}\bigg\|\int_{\Dd^2}\nabla_v\log f(z')\,f(z')^\delta f(z'')^\delta \big((f^{\otimes m+2})^{1-\delta}\bar C_{N,m+2}\big)(\cdot,z',z'')\,dz'dz''\bigg\|_{L^2},
\end{multline*}
and thus, by H\"older's inequality and the Sobolev embedding~$H^\gamma\subset L^{2d/(d-2\gamma)}$, splitting the kernel $K(x'-x'')$ for $|x'-x''|\le1$ and $|x'-x''|>1$, and using the exchangeability, we have that
\begin{multline*}
\|(f^{\otimes m})^{1-\delta}K_f[\bar C_{N,m+2}]\|_{L^2}\\
\le \Big(\|K\|_{L^{2d/(d+2\gamma)}(B)}\|f^\delta\|_{L^\infty_xL^2_v}+\big(\|K\|_{L^\infty(B^c)}+\|K\ast f\|_{L^\infty}\big)\|f^\delta\|_{L^2}\Big)\\
\times\|f^\delta\nabla_v\log f\|_{L^2}\|\langle\nabla_{x_1}\rangle^\gamma((f^{\otimes m+2})^{1-\delta}\bar C_{N,m+2})\|_{L^2}.
\end{multline*}
The different factors are precisely controlled by the assumptions on $K$ and $f$, hence
\begin{equation*}
\|(f^{\otimes m})^{1-\delta}K_f[\bar C_{N,m+2}]\|_{L^2}
\lesssim\|\langle\nabla_{x_1}\rangle^\gamma((f^{\otimes m+2})^{1-\delta}\bar C_{N,m+2})\|_{L^2}.
\end{equation*}
Similarly estimating the second term in~\eqref{eq:An2-decomp}, this time using Young's convolution inequality to bound the convolution with the singular part of $K$, we get that
\begin{multline*}
\Big\|(f^{\otimes m})^{1-\delta}\sum_{i=1}^m\int_\Dd K_f(z_*,z_i)\,\bar C_{N,m}(z_{[m]\setminus i},z_*)\,f(z_*)\,dz_*\Big\|_{L^2}\\
\le m\Big(\|K\|_{L^{2d/(d+2\gamma)}(B)}\|f^{1-\delta}\|_{L^\infty_xL^2_v}+\big(\|K\|_{L^\infty(B^c)}+\|K\ast f\|_{L^\infty}\big)\|f^{1-\delta}\|_{L^2}\Big)\\
\times\|f^\delta\nabla_v\log f\|_{L^2}\|\langle\nabla_{x_1}\rangle^\gamma((f^{\otimes m})^{1-\delta}\bar C_{N,m})\|_{L^2}.
\end{multline*}
We are thus led to
\[\|A_m^2\|_{L^2}\,\lesssim\,(m+1)\|\langle\nabla_{x_1}\rangle^\gamma((f^{\otimes m+2})^{1-\delta}\bar C_{N,m+2})\|_{L^2}+m\|\langle\nabla_{x_1}\rangle^\gamma((f^{\otimes m})^{1-\delta} \bar C_{N,m})\|_{L^2}.\]
Finally, for $A_m^3$ and $A_m^4$, the a priori $L^2$-estimates of Lemma~\ref{lem:apriori-L2} and the assumptions on~$f$ immediately ensure
\[\|\langle\nabla_V\rangle^{-1}A_m^3\|_{L^2}\lesssim m,\qquad\|A_m^4\|_{L^2}\lesssim m.\]

\medskip\noindent
{\bf Step~2:} Proof that for all $\gamma\le r\le1$,
\begin{equation}\label{eq:estim-An1}
\|\langle\nabla_X\rangle^{-r}\langle\nabla_V\rangle^{-1}A_m^1\|_{L^2}\,\lesssim\,m^2.
\end{equation}
Recall that $A_m^1=(f^{\otimes m})^{1-\delta}R_{N,m}$, where $R_{N,m}$ is the remainder term defined in Lemma~\ref{lem:hier}. The claim~\eqref{eq:estim-An1} follows from a careful examination of all the terms in the definition of~$R_{N,m}$, using the assumptions on $K$, $f$, and the a priori estimates of Lemma~\ref{lem:apriori-L2}.
For shortness, let us focus on the following three typical terms that are part of $R_{N,m}$,
\begin{eqnarray}
R_{N,m}^1&:=&-\Big(\frac{N-m+1}{m(N-1)^2}\Big)^\frac12\,\sum_{i\neq j}^{m} K(x_i-x_j)\cdot \nabla_{v_i} \bar C_{N,[m]\setminus j},\label{eq:def-RN123}\\
R_{N,m}^2&:=&-\frac1{N-1}\sum_{i\ne j}^mK(x_i-x_j)\cdot\nabla_{v_i}\bar C_{N,m},\nonumber\\
R_{N,m}^3&:=&\Big(\frac{(m+1)(N-m)}{(N-1)^2}\Big)^\frac12\sum_{j=1}^m\int_\Dd K_f(z_*,z_j)\,\bar C_{N,[m]\cup\{*\}}\,f(z_*)\,dz_*.\nonumber
\end{eqnarray}
For the first one, commuting the velocity derivative with the weight, we have by symmetry
\begin{multline*}
\|\langle\nabla_X\rangle^{-r}\langle\nabla_V\rangle^{-1}((f^{\otimes m})^{1-\delta}R_{N,m}^1)\|_{L^2}\\
\,\lesssim\,m^{\frac32}N^{-\frac12}\Big\|\langle\nabla_{x_1}\rangle^{-r}\Big(K(x_2-x_1)(f^{\otimes m})^{1-\delta}\bar C_{N,[2,m]}\Big)\Big\|_{L^2}\\
+m^{\frac32}N^{-\frac12}\Big\|\langle\nabla_{x_1}\rangle^{-r}\Big(K(x_2-x_1)\bar C_{N,[2,m]}\cdot\nabla_{v_2}(f^{\otimes m})^{1-\delta}\Big)\Big\|_{L^2},
\end{multline*}
and thus, for all $r\ge\gamma$, appealing to the Sobolev embedding $L^{2d/(d+2r)}_x\subset H^{-r}_x$, the assumptions on $K,f$, and the a priori $L^2$-estimate of Lemma~\ref{lem:apriori-L2}, we obtain that
\begin{equation*}
\|\langle\nabla_X\rangle^{-r}\langle\nabla_V\rangle^{-1}((f^{\otimes m})^{1-\delta}R_{N,m}^1)\|_{L^2}
\,\lesssim\,m^{\frac32}N^{-\frac12}.
\end{equation*}
We turn to the second typical term,
\begin{multline}\label{eq:decomp-RNm2}
\|\langle\nabla_X\rangle^{-r}\langle\nabla_V\rangle^{-1}((f^{\otimes m})^{1-\delta}R_{N,m}^2)\|_{L^2}\\
\lesssim m^2N^{-1}\Big\|\langle\nabla_{x_1}\rangle^{-r}\Big(K(x_1-x_2)(f^{\otimes m})^{1-\delta}\bar C_{N,m}\Big)\Big\|_{L^2}\\
+\,m^2N^{-1}\|\langle\nabla_{x_1}\rangle^{-r}\Big(K(x_1-x_2)\bar C_{N,m}\cdot\nabla_{v_1}(f^{\otimes m})^{1-\delta}\Big)\Big\|_{L^2}.
\end{multline}
For all $r\ge\gamma$, appealing again to the Sobolev embedding and using the assumptions on~$K,f$, we now find
\begin{multline*}
\lefteqn{\Big\|\langle\nabla_{x_1}\rangle^{-r}\Big(K(x_1-x_2)(f^{\otimes m})^{1-\delta}\bar C_{N,m}\Big)\Big\|_{L^2}}\\
\lesssim\Big\|(f^{\otimes m})^{1-\delta}\bar C_{N,m}\Big\|_{L^2(\Dd;L^\infty_xL^2_v\cap L^2(\Dd;L^2(\Dd^{m-2})))}
\lesssim\|\bar C_{N,m}\|_{L^\infty(\Dd^2;L^2(f^{\otimes m-2}))}.
\end{multline*}
By the mixed a priori estimate of Lemma~\ref{lem:apriori-L2}, this term is dominated by $CN/m$. Similarly estimating the second term in~\eqref{eq:decomp-RNm2}, using the assumptions on $\nabla_v f^{1-\delta}=({1-\delta})f^{1-\delta}\nabla_v\log f$, we are led to
\[\|\langle\nabla_X\rangle^{-r}\langle\nabla_V\rangle^{-1}((f^{\otimes m})^{1-\delta}R_{N,m}^2)\|_{L^2}\lesssim m.\]
Finally, for the third typical term $R_{N,m}^3$, we can similarly bound, using again the mixed a priori estimate of Lemma~\ref{lem:apriori-L2},
\begin{eqnarray*}
\|(f^{\otimes m})^{1-\delta}R_{N,m}^3\|_{L^2}
&\lesssim&m^\frac32N^{-\frac12}\Big\|(f^{\otimes m})^{1-\delta}\int_\Dd K_f(z_*,z_1)\,\bar C_{N,[m]\cup\{*\}}\,f(z_*)\,dz_*\Big\|_{L^2}\\
&\lesssim&m^\frac32N^{-\frac12}\|\bar C_{N,m+1}\|_{L^{2d/(d-2\gamma)}(f^{\otimes2};L^2(f^{\otimes m-1}))}\\
&\lesssim&m^\frac32N^{\frac{2\gamma}{d}-\frac12}.
\end{eqnarray*}
As by assumption $\gamma\le\frac d4$, this is bounded by $m^{\frac32}$.
Arguing similarly for all the other terms in the definition of $R_{N,m}$, the claim~\eqref{eq:estim-An1} follows.

\medskip\noindent
{\bf Step~3:} Conclusion.\\
Inserting the estimates of Steps~1 and~2 into~\eqref{eq:decomp-hypobound}, we get for all $\gamma\le r\le\frac13$,
\begin{multline*}
\int_0^T\|\langle\nabla_{x_1}\rangle^{\frac13-r}((f^{\otimes m})^{1-\delta}\bar C_{N,m})\|_{L^2}^2
\,\lesssim\,(T+1)(m+1)^4\\
+(m+1)^2\int_0^T\|\langle\nabla_{x_1}\rangle^{\gamma}((f^{\otimes m+2})^{1-\delta}\bar C_{N,m+2})\|_{L^2}^2\\
+m^2\int_0^T\|\langle\nabla_{x_1}\rangle^{\gamma}((f^{\otimes m})^{1-\delta}\bar C_{N,m})\|_{L^2}^2,
\end{multline*}
and thus, by interpolation with the a priori  $L^2$-estimate of Lemma~\ref{lem:apriori-L2}, for all $\theta\in(0,1)$, we deduce that
\begin{multline}\label{eq:estim-rec-hypoell}
\int_0^T\|\langle\nabla_{x_1}\rangle^{\frac13-r}((f^{\otimes m})^{1-\delta}\bar C_{N,m})\|_{L^2}^2
\,\lesssim\,(T+1)(m+1)^4\\
+(m+1)^2\int_0^T\|\langle\nabla_{x_1}\rangle^{\frac\gamma{1-\theta}}((f^{\otimes m+2})^{1-\delta}\bar C_{N,m+2})\|_{L^2}^{2(1-\theta)}\\
+m^2\int_0^T\|\langle\nabla_{x_1}\rangle^{\frac\gamma{1-\theta}}((f^{\otimes m})^{1-\delta}\bar C_{N,m})\|_{L^2}^{2(1-\theta)}.
\end{multline}
Provided $\frac{\gamma}{1-\theta}\le\frac13-r$, Young's inequality allows to absorb the last right-hand side term into the left-hand side by bounding
\begin{eqnarray*}
\lefteqn{m^2\int_0^T\|\langle\nabla_{x_1}\rangle^{\frac\gamma{1-\theta}}((f^{\otimes m})^{1-\delta}\bar C_{N,m})\|_{L^2}^{2(1-\theta)}}\\
&\le& m^2T^{\theta}\Big(\int_0^T\|\langle\nabla_{x_1}\rangle^{\frac13 -r}((f^{\otimes m})^{1-\delta}\bar C_{N,m})\|_{L^2}^{2}\Big)^{1-\theta}\\
&\le& Cm^{\frac{2}{\theta}}T+\frac{1}{2}\,\int_0^T\|\langle\nabla_{x_1}\rangle^{\frac13 -r}((f^{\otimes m})^{1-\delta}\bar C_{N,m})\|_{L^2}^{2}.
\end{eqnarray*}
Performing the same estimate for the second right-hand side term in~\eqref{eq:estim-rec-hypoell}, we obtain in this way, denoting $s=1/3-r$, for all $\gamma<s\le\frac13-\gamma$,
\begin{multline*}
\int_0^T\|\langle\nabla_{x_1}\rangle^{s}((f^{\otimes m})^{1-\delta}\bar C_{N,m})\|_{L^2}^2\\
\,\le\,C(T+1)(m+1)^{4\vee\frac{2s}{s-\gamma}}
+\frac12\int_0^T\|\langle\nabla_{x_1}\rangle^{s}((f^{\otimes m+2})^{1-\delta}\bar C_{N,m+2})\|_{L^2}^2,
\end{multline*}
and the conclusion follows from a direct iteration.
\end{proof}


\section{Deterministic case: hierarchical estimates via averaging}\label{sec:averaging}
We turn to the deterministic case without velocity diffusion, $\alpha=0$. A priori estimates on dual correlations will now be obtained by applying averaging lemmas to the hierarchy~\eqref{eq:lemhier}, instead of hypoelliptic regularity. More precisely, we appeal to the following slight variant of the~$L^2$ averaging lemma of~\cite{JLT-22}; for completeness, a short proof is given in Appendix~\ref{app:kinetic-reg}. As for hypoelliptic estimates, the result is uniform in the number of particle variables $m$.

\begin{lem}[Averaging lemma]\label{lem:averaging}
Given $m\ge1$, let $g_m,e_m$ satisfy in the weak sense on~$[0,T]\times\Dd^m$,
\[\partial_tg_m+\sum_{i=1}^mv_i\cdot\nabla_{x_i}g_m=e_m.\]
Then, for all $0\le r\le1$,
\begin{multline*}
\int_0^T\||\nabla_{x_{[m]}}|^\frac{1-r}4\langle\nabla_{v_{[m]}}\rangle^{-\frac32}g_m\|_{L^2}^2\\[-2mm]
\le C\|g_m(0)\|_{L^2}^2+C\int_0^T\|g_m\|_{L^2}^2+C\int_0^T\|\langle\nabla_{x_{[m]}}\rangle^{-r}\langle\nabla_{v_{[m]}}\rangle^{-1}e_m\|_{L^2}^2,
\end{multline*}
for some universal constant $C$.
In the case when $e_m=e_m^0+\Div_v(e_m^1)$, we further have for all $0\le r,s\le 1$,
\begin{multline*}
\int_0^T\||\nabla_{x_{[m]}}|^{\frac{1-s}2\wedge\frac{1-r}4}\langle\nabla_{v_{[m]}}\rangle^{-\frac32}g_m\|_{L^2}^2\\[-2mm]
\le C\,\|g_m(0)\|_{L^2}^2+C\int_0^T\|g_m\|_{L^2}^2+C\int_0^T\|\langle\nabla_{x_{[m]}}\rangle^{-s}e_m^0\|_{L^2}^2+C\int_0^T\|\langle\nabla_{x_{[m]}}\rangle^{-r}e_m^1\|_{L^2}^2.
\end{multline*}
\end{lem}
Applying this lemma directly to the $m$-particle equation would distribute the loss of velocity regularity over all particle variables, which would be problematic to adapt the hierarchical iteration in the proof of Proposition~\ref{prop:bnd-diff}. To remedy this, we appeal to the following convenient Hilbert-valued version.
The idea is to view the last $m-2$ particle variables instead as a Hilbert-valued parameter and to keep their mean-field transport on the left-hand side: since that transport is skew-adjoint in $L^2$ and acts only on the Hilbert-valued variables, it can be removed by a unitary conjugation, after which the two-particle version of the above averaging lemma applies with regularity loss only in $(v_1,v_2)$.

\begin{lem}\label{lem:averaging-hilb}
Let $\mathfrak h$ be a fixed, separable Hilbert space and let $L_t$ be a time-dependent skew-adjoint operator on $\mathfrak h$ generating a unitary propagator. Let $g,e$ be $\mathfrak h$-valued functions satisfying in the weak sense on $[0,T]\times\Dd^2$,
\[\Big(\partial_t+v_1\cdot\nabla_{x_1}+v_2\cdot\nabla_{x_2}+L_t\Big)g=e.\]
Then, for all $0\le r\le1$,
\begin{multline*}
\int_0^T\||\nabla_{x_{[2]}}|^{\frac{1-r}4}\langle\nabla_{v_{[2]}}\rangle^{-\frac32}g\|_{L^2(\Dd^2;\mathfrak h)}^2\,dt\\[-2mm]
\le C(T+1)\sup_{[0,T]}\|g\|_{L^2(\Dd^2;\mathfrak h)}^2
+C\int_0^T\|\langle\nabla_{x_{[2]}}\rangle^{-r}\langle\nabla_{v_{[2]}}\rangle^{-1}e\|_{L^2(\Dd^2;\mathfrak h)}^2\,dt,
\end{multline*}
for some universal constant $C$. In the case when $e=e^0+\Div_{v_{[2]}}(e^1)$, we further have for all $0\le r,s\le 1$,
\begin{multline*}
\int_0^T\||\nabla_{x_{[2]}}|^{\frac{1-s}2\wedge\frac{1-r}4}\langle\nabla_{v_{[2]}}\rangle^{-\frac32}g\|_{L^2(\Dd^2,\mathfrak h)}^2\\[-2mm]
\le C\,(T+1)\,\sup_{[0,\ T]}\|g\|_{L^2(\Dd^2,\mathfrak h)}^2+C\int_0^T\|\langle\nabla_{x_{[2]}}\rangle^{-s}e^0\|_{L^2(\Dd^2,\mathfrak h)}^2+C\int_0^T\|\langle\nabla_{x_{[2]}}\rangle^{-r}e^1\|_{L^2(\Dd^2,\mathfrak h)}^2.
\end{multline*}
\end{lem}

\begin{proof}
Let $(u_n^0)_{n\in \N}$ denote an orthonormal basis of $\mathfrak h$ and, for each $n$, let $u_n$ solve
\[\partial_t u_n-L_t^* u_n=0,\qquad u_n|_{t=0}=u_n^0.\]
From the assumptions, $(u_n(t))_{n\in\N}$ forms another orthonormal basis of $\mathfrak h$ for every $t\in[0,T]$. We now consider $\langle u_n,g\rangle_{\mathfrak h}$, which are only functions of $t,z_1,z_2$. From the equations on $g$ and $u_n$, we obtain
\[(\partial_t+v_1\cdot \nabla_{x_1}+v_2\cdot\nabla_{x_2})\langle u_n,g\rangle_{\mathfrak h}=\langle u_n,e\rangle_{\mathfrak h}.\]
We can hence directly apply Lemma~\ref{lem:averaging} to find that
\begin{multline*}
\int_0^T\||\nabla_{x_{[2]}}|^{\frac{1-r}4}\langle\nabla_{v_{[2]}}\rangle^{-\frac32}\langle u_n,g\rangle_{\mathfrak h}\|_{L^2}^2\,dt\\[-2mm]
\le C\|\langle u_n^0,g(0)\rangle_{\mathfrak h}\|_{L^2}^2+C\int_0^T\|\langle u_n,g\rangle_{\mathfrak h}\|_{L^2}^2
+C\int_0^T\|\langle\nabla_{x_{[2]}}\rangle^{-r}\langle\nabla_{v_{[2]}}\rangle^{-1}\langle u_n,e\rangle_{\mathfrak h}\|_{L^2}^2\,dt.
\end{multline*}
Summing the left-hand side over $n$ and using that $(u_n(t))_n$ is an orthonormal basis for each~$t$, we have
\begin{eqnarray*}
\lefteqn{\sum_n \int_0^T\||\nabla_{x_{[2]}}|^{\frac{1-r}4}\langle\nabla_{v_{[2]}}\rangle^{-\frac32}\langle u_n,g\rangle_{\mathfrak h}\|_{L^2}^2}\\
&=&\int_0^T \int_{\Dd^2} \sum_n\big|\langle u_n,|\nabla_{x_{[2]}}|^{\frac{1-r}4}\langle\nabla_{v_{[2]}}\rangle^{-\frac32} g\rangle_{\mathfrak h}\big|^2\\
&=&\int_0^T \||\nabla_{x_{[2]}}|^{\frac{1-r}4}\langle\nabla_{v_{[2]}}\rangle^{-\frac32} g\|_{L^2(\Dd^2;\mathfrak{h})}^2\,dt,
\end{eqnarray*}
and similarly for the right-hand side terms. The conclusion follows, and the case $e=e^0+\Div_{v_{[2]}} e^1$ is treated in the same way by using the variant of Lemma~\ref{lem:averaging}.
\end{proof}

As the remainder term $R_{N,m}$ in the hierarchy~\eqref{eq:lemhier} contains singularities in every particle variable, we will need to combine the above Hilbert-valued averaging lemma with a mollification argument in the background variables. For $\e\in(0,1]$ to be chosen below, let $M_\e(z):=\e^{-2d}M(z/\e)$ be a family of mollifiers where $M$ is smooth, exponentially decaying, and normalized by $\int_\Dd M=1$. For any $P\subset[m]$, denote by $M_\e^P$ the convolution operator acting only in the variables $(z_j)_{j\in P}$,
\[M_\e^Ph_m(z_{[m]}):=\int_{\Dd^{\sharp P}}\Big(\prod_{j\in P}M_\e(z_j-z_j')\Big)h_m(z_{[m]\setminus P},z_P')\,dz_P'.\]
In these terms, we establish the following uniform-in-$N$ a priori estimate.

\begin{prop}\label{prop:bnd-nodiff}
Let
\[\alpha=0,\qquad K\in L^{2d/(d+2\gamma)}_\loc(\R^d),\qquad\sup_{|x|\ge1}|K(x)|<\infty,\]
for some
\[0\le\gamma<\gamma_*:=\left\{\begin{array}{lll}
\tfrac1{16}\big(5d+4-\sqrt{(5d+4)^2-32d}\big)&:&d\ge2,\\[1mm]
\frac12(\sqrt{5}-2)&:&d=1.
\end{array}\right.\]
Also assume that the mean-field solution~$f$ satisfies the regularity conditions~\eqref{eq:reg-req-0}, \eqref{eq:reg-req-0b}, and~\eqref{eq:reg-req-1} for some~$\delta\in(0,\frac12]$.
Then there is $\theta>0$ such that for the choice $\e=N^{-\theta}$ we have for all~$0\le m\le N$,
\[\int_0^T\||\nabla_{x_{[2]}}|^\gamma\langle\nabla_{v_{[2]}}\rangle^{-\frac32}M_\e^{[3,m]}((f^{\otimes m})^{1-\delta}\bar C_{N,m})\|_{L^2}^2\,\le\,C(T+1)(m+1)^c,\]
for some constant $c$ only depending on $\gamma$ and some constant $C$ further depending on $\delta$ and on assumptions on $K,f$.
\end{prop}

\begin{proof}
For $\alpha=0$, equation~\eqref{eq:rewr-Cnn-fn} can be reorganized as
\begin{multline*}
\Big(\partial_t+v_1\cdot\nabla_{x_1}+v_2\cdot\nabla_{x_2}+\sum_{i=3}^m\big(v_i\cdot\nabla_{x_i}+K\ast f(x_i)\cdot\nabla_{v_i}\Big)((f^{\otimes m})^{1-\delta}\bar C_{N,m})\\[-3mm]
=A_m^1+A_m^2-\sum_{i=1}^2\Div_{v_i}\Big((K\ast f)(x_i)(f^{\otimes m})^{1-\delta}\bar C_{N,m}\Big),
\end{multline*}
where $A_m^1,A_m^2$ are defined as in~\eqref{eq:rewr-Cnn-fn}.
Applying the mollification operator $M_\e^{[3,m]}$ to both sides of this equation, only acting in the variables $z_3,\ldots,z_m$, we obtain
\begin{multline}\label{eq:pde-Meps-f-C}
\Big(\partial_t+v_1\cdot\nabla_{x_1}+v_2\cdot\nabla_{x_2}+\sum_{i=3}^m(v_i\cdot\nabla_{x_i}+K\ast f(x_i)\cdot\nabla_{v_i})\Big)M_\e^{[3,m]}((f^{\otimes m})^{1-\delta}\bar C_{N,m})\\[-3mm]
=R_{\e,m}+M_\e^{[3,m]}A_m^1+M_\e^{[3,m]}A_m^2-\sum_{i=1}^2\Div_{v_i}\Big((K\ast f)(x_i)M_\e^{[3,m]}((f^{\otimes m})^{1-\delta}\bar C_{N,m})\Big),
\end{multline}
in terms of the commutator
\[R_{\e,m}:=\bigg[\sum_{i=3}^m\big(v_i\cdot\nabla_{x_i}+K\ast f(x_i)\cdot\nabla_{v_i}\big),M_\e^{[3,m]}\bigg](f^{\otimes m})^{1-\delta}\bar C_{N,m}.\]
For one particle variable, this commutator is explicitly
\begin{multline*}
\big[v\cdot\nabla_x+K\ast f(x)\cdot\nabla_v,M_\e\big]h(z)
 =\int_{\mathbb D}\Big((v-v')\cdot\nabla_xM_\varepsilon(z-z')\\
+\big(K\ast f(x)-K\ast f(x')\big)\cdot\nabla_vM_\varepsilon(z-z')\Big)h(z')\,dz'.
\end{multline*}
Recalling $M_\e(z)=\e^{-2d}M(z/\e)$ with $M$ exponentially decaying, and using the Lipschitz bound on $K\ast f$, we deduce, uniformly in $\e$,
\[\big\|\big[v\cdot\nabla_x+K\ast f(x)\cdot\nabla_v,M_\e\big]h\big\|_{L^2}\lesssim \|h\|_{L^2}.\]
Applying this to $R_{\e,m}$, together with the a priori $L^2$-estimate of Lemma~\ref{lem:apriori-L2} and the boundedness of~$f$, we get
\[\|R_{\e,m}\|_{L^2}\lesssim m.\]
Now applying Lemma~\ref{lem:averaging-hilb} to equation~\eqref{eq:pde-Meps-f-C} with $\mathfrak h=L^2(\Dd^{m-2})$ and with~$L_t$ given by mean-field transport in the variables $z_3,\ldots,z_m$, using this bound on the commutator term and again the a priori $L^2$-estimate of Lemma~\ref{lem:apriori-L2}, we deduce for any $0\le r\le1$,
\begin{multline}\label{eq:estim-averaging-Dn}
\int_0^T\||\nabla_{x_{[2]}}|^{\frac{1-r}4}\langle\nabla_{v_{[2]}}\rangle^{-\frac32}M_\e^{[3,m]}((f^{\otimes m})^{1-\delta}\bar C_{N,m})\|_{L^2}^2\\[-1mm]
\lesssim T+1+\sum_{j=1}^2\int_0^T\|\langle\nabla_{x_{[2]}}\rangle^{-r}\langle\nabla_{v_{[2]}}\rangle^{-1}M_\e^{[3,m]}A_m^j\|_{L^2}^2.
\end{multline}
It remains to estimate these norms of $A_m^1$ and $A_m^2$, which we separately do in the following two steps.

\medskip\noindent
{\bf Step~1:} We control $A_m^2$ by showing that for all $0\le r\le 1$,
\begin{multline}\label{eq:todo-An2}
\|\langle\nabla_{x_{[2]}}\rangle^{-r}\langle\nabla_{v_{[2]}}\rangle^{-1}M_\e^{[3,m]}A_m^2\|_{L^2}
\,\lesssim\,
1+(m+1)\e^rN^\frac{2\gamma}d\\
+(m+1)\|\langle\nabla_{x_{[2]}}\rangle^\gamma\langle\nabla_{v_{[2]}}\rangle^{-\frac32}M_\e^{[3,m+2]}((f^{\otimes m+2})^{1-\delta}\bar C_{N,m+2})\|_{L^2}\\
+m\|\langle\nabla_{x_{[2]}}\rangle^\gamma\langle\nabla_{v_{[2]}}\rangle^{-\frac32}M_\e^{[3,m]}((f^{\otimes m})^{1-\delta}\bar C_{N,m})\|_{L^2}.
\end{multline}
By definition of $A_m^2$,
\begin{multline}\label{eq:contr-An1-dec}
\|\langle\nabla_{x_{[2]}}\rangle^{-r}\langle\nabla_{v_{[2]}}\rangle^{-1}M_\e^{[3,m]}A_m^2\|_{L^2}\\
\,\lesssim\,(m+1)\|\langle\nabla_{x_{[2]}}\rangle^{-r}\langle\nabla_{v_{[2]}}\rangle^{-1}M_\e^{[3,m]}((f^{\otimes m})^{1-\delta}K_f[\bar C_{N,m+2}])\|_{L^2}\\
+\sum_{i=1}^m\Big\|\langle\nabla_{x_{[2]}}\rangle^{-r}\langle\nabla_{v_{[2]}}\rangle^{-1}M_\e^{[3,m]}\\
\times\Big((f^{\otimes m})^{1-\delta}\int_\Dd K_f(z_*,z_i)\bar C_{N,m}(z_{[m]\setminus\{i\}},z_*)\,f(z_*)\,dz_*\Big)\Big\|_{L^2}.
\end{multline}
Both right-hand side terms can be treated in the same way, so we only discuss the first one for example. As the mollification kernel $M$ has exponential decay and satisfies $\int_\Dd M=1$, we have for any function $h$ and any $0\le r\le1$, uniformly in $\e$,
\begin{equation*}
\|\langle\nabla_{x}\rangle^{-r}\langle\nabla_{v}\rangle^{-1}h\|_{L^2}\,\lesssim\,\|M_\e\ast h\|_{L^2}+\e^r\|h\|_{L^2}.
\end{equation*}
Applying this in the first two variables gives a fully mollified contribution and an $O(\e^r)$ error,
\begin{multline}\label{eq:interpol-Leps}
\|\langle\nabla_{x_{[2]}}\rangle^{-r}\langle\nabla_{v_{[2]}}\rangle^{-1}M_\e^{[3,m]}((f^{\otimes m})^{1-\delta}K_f[\bar C_{N,m+2}])\|_{L^2}\\
\,\lesssim\,\|M_\e^{[m]}((f^{\otimes m})^{1-\delta}K_f[\bar C_{N,m+2}])\|_{L^2}+\e^r\|(f^{\otimes m})^{1-\delta}K_f[\bar C_{N,m+2}]\|_{L^2}.
\end{multline}
For the first term, recalling the definition of the operator~$K_f[\cdot]$, cf.~\eqref{eq:def-VC}, and using H\"older's inequality and the Sobolev embedding $H^\gamma\subset L^{2d/(d-2\gamma)}$, similarly as in Step~1 of the proof of Proposition~\ref{prop:bnd-diff}, we can bound
\begin{multline*}
\|M_\e^{[m]}((f^{\otimes m})^{1-\delta}K_f[\bar C_{N,m+2}])\|_{L^2}\\
\,\le\,\Big(\|K\|_{L^{2d/(d+2\gamma)}(B)}\|f^\delta\|_{L^\infty_xH^{3/2}_v}+\big(\|K\|_{L^\infty(B^c)}+\|K\ast f\|_{L^\infty}\big)\|f^\delta\|_{L^2_xH^{3/2}_v}\Big)\\
\times\|f^\delta\nabla_v\log f\|_{L^2_xH^{3/2}_v}\|\langle\nabla_{x_{[m+1,m+2]}}\rangle^\gamma\langle\nabla_{v_{[m+1,m+2]}}\rangle^{-\frac32}M_\e^{[m]}((f^{\otimes m+2})^{1-\delta}\bar C_{N,m+2})\|_{L^2}.
\end{multline*}
Thus, by the assumptions on $K,f$ and the symmetry of $\bar C_{N,m+2}$,
\begin{equation*}
\|M_\e^{[m]}((f^{\otimes m})^{1-\delta}K_f[\bar C_{N,m+2}])\|_{L^2}
\,\lesssim\,\|\langle\nabla_{x_{[2]}}\rangle^\gamma\langle\nabla_{v_{[2]}}\rangle^{-\frac32}M_\e^{[3,m+2]}((f^{\otimes m+2})^{1-\delta}\bar C_{N,m+2})\|_{L^2}.
\end{equation*}
For the second term in~\eqref{eq:interpol-Leps}, we can simply bound
\begin{multline*}
\|(f^{\otimes m})^{1-\delta}K_f[\bar C_{N,m+2}]\|_{L^2}\\
\le\Big(\|K\|_{L^{2d/(d+2\gamma)}(B)}\|f^\delta\|_{L^\infty_xL^2_v}+\big(\|K\|_{L^\infty(B^c)}+\|K\ast f\|_{L^\infty}\big)\|f^\delta\|_{L^2}\Big)\\
\times \|f^\delta\nabla_v\log f\|_{L^2}\Big(\|(f^{\otimes m+1})^{1-\delta}\bar C_{N,m+2}\|_{L^2(\Dd;L^{2d/(d-2\gamma)}_xL^2_v(\Dd;L^2(\Dd^m)))}\\
+\|(f^{\otimes m+1})^{1-\delta}\bar C_{N,m+2}\|_{L^2}\Big),
\end{multline*}
and thus, by the assumptions on $K,f$, Jensen's inequality with weight $f$, and the mixed a priori estimates of Lemma~\ref{lem:apriori-L2},
\begin{equation*}
\|(f^{\otimes m})^{1-\delta}K_f[\bar C_{N,m+2}]\|_{L^2}
\lesssim\|\bar C_{N,m+2}\|_{L^{2d/(d-2\gamma)}(f^{\otimes2};L^2(f^{\otimes m}))}
\lesssim N^{\frac{2\gamma}d}.
\end{equation*}
Combining these estimates into~\eqref{eq:interpol-Leps}, and arguing similarly for the second term in~\eqref{eq:contr-An1-dec}, the claim~\eqref{eq:todo-An2} follows.

\medskip\noindent
{\bf Step~2:} We control the remainder $A_m^1$ by showing that for all $\gamma\le r\le1$,
\begin{equation}\label{eq:todo-An1-re}
\|\langle\nabla_{x_{[2]}}\rangle^{-r}\langle\nabla_{v_{[2]}}\rangle^{-1}M_\e^{[3,m]}A_m^1\|_{L^2}
\,\lesssim\,m^2\Big(1+N^{-\frac12}\e^{-1-\gamma}+N^{\frac{2\gamma}d-1}\e^{-1-\frac d2}+N^{\frac{2\gamma}d-\frac12}\Big).
\end{equation}
Recall $A_m^1=(f^{\otimes m})^{1-\delta}R_{N,m}$, where $R_{N,m}$ is the remainder term defined in Lemma~\ref{lem:hier}.
The estimate~\eqref{eq:todo-An1-re} follows from a careful examination of all the terms entering the definition of~$R_{N,m}$, using the assumptions on $K,f$ and the mixed a priori estimates of Lemma~\ref{lem:apriori-L2}.
For shortness, we focus again on the three typical contributions~$R_{N,m}^1,R_{N,m}^2,R_{N,m}^3$ in~\eqref{eq:def-RN123}.
For $R_{N,m}^1$, distinguishing the four cases $i,j\le2$, $i\le2<j$, $j\le2<i$, and~$i,j>2$, and commuting the velocity derivative with the weight, we get by symmetry
\[\|\langle\nabla_{x_{[2]}}\rangle^{-r}\langle\nabla_{v_{[2]}}\rangle^{-1}M_\e^{[3,m]}((f^{\otimes m})^{1-\delta}R_{N,m}^1)\|_{L^2}\\
\,\lesssim\,\sum_{j=1}^4T_m^j,\]
in terms of
\begin{eqnarray*}
T_m^1&:=&m^{-\frac12}N^{-\frac12}\Big\|\langle\nabla_{x_1}\rangle^{-r}\Big((f^{\otimes m})^{1-\delta} K(x_2-x_1)\bar C_{N,[2,m]}\Big)\Big\|_{L^2}\\
&&+m^{-\frac12}N^{-\frac12}\Big\|\langle\nabla_{x_1}\rangle^{-r}\Big(K(x_2-x_1)\bar C_{N,[2,m]}\cdot\nabla_{v_2}(f^{\otimes m})^{1-\delta}\Big)\Big\|_{L^2},\\
T_m^2&:=&m^\frac12N^{-\frac12}\Big\|M_\e^{[3,m]}\Big((f^{\otimes m})^{1-\delta} K(x_1-x_3)\bar C_{N,[m]\setminus\{3\}}\Big)\Big\|_{L^2}\\
&&+m^\frac12N^{-\frac12}\Big\|M_\e^{[3,m]}\Big(K(x_1-x_3)\bar C_{N,[m]\setminus\{3\}}\cdot\nabla_{v_1}(f^{\otimes m})^{1-\delta}\Big)\Big\|_{L^2},\\
T_m^3&:=&m^\frac12N^{-\frac12}\Big\|\langle\nabla_{x_1}\rangle^{-r}M_\e^{[3,m]}\,\Div_{v_3}\Big((f^{\otimes m})^{1-\delta} K(x_3-x_1) \bar C_{N,[m]\setminus\{1\}}\Big)\Big\|_{L^2}\\
&&+m^\frac12N^{-\frac12}\Big\|\langle\nabla_{x_1}\rangle^{-r}M_\e^{[3,m]}\Big(K(x_3-x_1) \bar C_{N,[m]\setminus\{1\}}\cdot\nabla_{v_3}(f^{\otimes m})^{1-\delta}\Big)\Big\|_{L^2},\\
T_m^4&:=&m^\frac32N^{-\frac12}\Big\|M_\e^{[3,m]}\,\Div_{v_4}\Big((f^{\otimes m})^{1-\delta} K(x_4-x_3) \bar C_{N,[m]\setminus\{3\}}\Big)\Big\|_{L^2}\\
&&+m^\frac32N^{-\frac12}\Big\|M_\e^{[3,m]}\Big( K(x_4-x_3) \bar C_{N,[m]\setminus\{3\}}\cdot\nabla_{v_4}(f^{\otimes m})^{1-\delta}\Big)\Big\|_{L^2}.
\end{eqnarray*}
Note that the convolution with $M_\e$ can be bounded as follows by Young's convolution inequality, for all $0\le q\le\frac d2$,
\begin{equation}\label{eq:bnd-Lconv}
\|M_\e\ast h\|_{L^2}\le\|M_\e\|_{L^{d/(d-q)}_xL^1_v}\|h\|_{L^{2d/(d+2q)}_xL^2_v}\lesssim\e^{-q}\|h\|_{L^{2d/(d+2q)}_xL^2_v}.
\end{equation}
Using this with $q=\gamma$, appealing to the Sobolev embedding $L^{2d/(d+2r)}_x\subset H^{-r}_x$, using the assumptions on $K,f$ and the a priori estimates of Lemma~\ref{lem:apriori-L2}, we get for all $r\ge\gamma$,
\begin{equation*}
\|\langle\nabla_{x_{[2]}}\rangle^{-r}\langle\nabla_{v_{[2]}}\rangle^{-1}M_\e^{[3,m]}((f^{\otimes m})^{1-\delta}R_{N,m}^1)\|_{L^2}\\
\,\lesssim\,
m^\frac32N^{-\frac12}\e^{-1-\gamma}.
\end{equation*}
For $R_{N,m}^2$, distinguishing again the cases \mbox{$i,j\le2$}, \mbox{$i\le2<j$}, \mbox{$j\le2<i$}, and~$i,j>2$, we get by symmetry
\begin{equation*}
\|\langle\nabla_{x_{[2]}}\rangle^{-r}\langle\nabla_{v_{[2]}}\rangle^{-1}M_\e^{[3,m]}((f^{\otimes m})^{1-\delta}R_{N,m}^2)\|_{L^2}
\,\lesssim\,\sum_{j=1}^4S_m^j,
\end{equation*}
in terms of
\begin{eqnarray*}
S_m^1&:=&
N^{-1}\Big\|\langle\nabla_{x_1}\rangle^{-r}\Big((f^{\otimes m})^{1-\delta}K(x_2-x_1)\bar C_{N,m}\Big)\Big\|_{L^2}\\
&&+N^{-1}\Big\|\langle\nabla_{x_1}\rangle^{-r}\Big(K(x_2-x_1)\bar C_{N,m}\cdot\nabla_{v_2}(f^{\otimes m})^{1-\delta}\Big)\Big\|_{L^2},\\
S_m^2&:=&mN^{-1}\Big\|M_\e^{[3,m]}\Big((f^{\otimes m})^{1-\delta}K(x_1-x_3)\bar C_{N,m}\Big)\Big\|_{L^2}\\
&&+mN^{-1}\Big\|M_\e^{[3,m]}\Big(K(x_1-x_3)\bar C_{N,m}\cdot\nabla_{v_1}(f^{\otimes m})^{1-\delta}\Big)\Big\|_{L^2},\\
S_m^3&:=&mN^{-1}\Big\|\langle\nabla_{x_1}\rangle^{-r}M_\e^{[3,m]}\,\Div_{v_3}\Big((f^{\otimes m})^{1-\delta}K(x_3-x_1)\bar C_{N,m}\Big)\Big\|_{L^2}\\
&&+mN^{-1}\Big\|\langle\nabla_{x_1}\rangle^{-r}M_\e^{[3,m]}\,\Big(K(x_3-x_1)\bar C_{N,m}\cdot\nabla_{v_3}(f^{\otimes m})^{1-\delta}\Big)\Big\|_{L^2},\\
S_m^4&:=&m^2N^{-1}\Big\|M_\e^{[3,m]}\,\Div_{v_4}\Big((f^{\otimes m})^{1-\delta}K(x_4-x_3)\bar C_{N,m}\Big)\Big\|_{L^2}\\
&&+m^2N^{-1}\Big\|M_\e^{[3,m]}\,\Big(K(x_4-x_3)\bar C_{N,m}\cdot\nabla_{v_4}(f^{\otimes m})^{1-\delta}\Big)\Big\|_{L^2}.
\end{eqnarray*}
Using again~\eqref{eq:bnd-Lconv} and the Sobolev embedding $L^{2d/(d+2r)}_x\subset H^{-r}_x$, we can bound for all~$r\ge\gamma$ and $0\le q\le\frac d2$,
\begin{multline*}
\|\langle\nabla_{x_{[2]}}\rangle^{-r}\langle\nabla_{v_{[2]}}\rangle^{-1}M_\e^{[3,m]}((f^{\otimes m})^{1-\delta}R_{N,m}^2)\|_{L^2}\\
\,\lesssim\,
N^{-1}\Big\|(f^{\otimes m})^{1-\delta}K(x_1-x_2)\bar C_{N,m}\Big\|_{L^2_{x_2}L^{2d/(d+2r)}_{x_1}L^2_{\ne x_1,x_2}}\\
+mN^{-1}\Big\|K(x_1-x_2)\bar C_{N,m}\cdot\nabla_{v_2}(f^{\otimes m})^{1-\delta}\Big\|_{L^2_{x_2}L^{2d/(d+2r)}_{x_1}L^2_{\ne x_1,x_2}}\\
+m^2N^{-1}\e^{-1-q}\Big\|(f^{\otimes m})^{1-\delta}K(x_1-x_2)\bar C_{N,m}\Big\|_{L^2_{x_2}L^{2d/(d+2q)}_{x_1}L^2_{\ne x_1,x_2}}\\
+m^2N^{-1}\e^{-q}\Big\|K(x_1-x_2)\bar C_{N,m}\cdot\nabla_{v_2}(f^{\otimes m})^{1-\delta}\Big\|_{L^2_{x_2}L^{2d/(d+2q)}_{x_1}L^2_{\ne x_1,x_2}}.
\end{multline*}
Thus, using the assumptions on $K,f$, Jensen's inequality with weight $f$, and the mixed a priori estimates of Lemma~\ref{lem:apriori-L2},
\begin{eqnarray*}
\lefteqn{\|\langle\nabla_{x_{[2]}}\rangle^{-r}\langle\nabla_{v_{[2]}}\rangle^{-1}M_\e^{[3,m]}((f^{\otimes m})^{1-\delta}R_{N,m}^2)\|_{L^2}}\\
&\lesssim&mN^{-1}\|\bar C_{N,m}\|_{L^\infty(\Dd^2;L^2(f^{\otimes m-2}))}
+m^2N^{-1}\e^{-1-q}\|\bar C_{N,m}\|_{L^{d/(q-\gamma)}(f^{\otimes2};L^2(f^{\otimes m-2}))}\\
&\lesssim&1+m^2\e^{-1-q}N^{-\frac2d(q-\gamma)}.
\end{eqnarray*}
Choosing $q=\frac d2$ gives
\begin{equation*}
\|\langle\nabla_{x_{[2]}}\rangle^{-r}\langle\nabla_{v_{[2]}}\rangle^{-1}M_\e^{[3,m]}((f^{\otimes m})^{1-\delta}R_{N,m}^2)\|_{L^2}
\,\lesssim\,1+m^2\e^{-1-\frac d2}N^{\frac{2\gamma}d-1}.
\end{equation*}
Finally, for $R_{N,m}^3$, we can simply bound
\begin{multline*}
\|\langle\nabla_{x_{[2]}}\rangle^{-r}\langle\nabla_{v_{[2]}}\rangle^{-1}M_\e^{[3,m]}((f^{\otimes m})^{1-\delta}R_{N,m}^3)\|_{L^2}\\
\,\lesssim\,m^\frac32N^{-\frac12}\Big\|(f^{\otimes m})^{1-\delta}\int_\Dd K_f(z_*,z_1)\,\bar C_{N,[m]\cup\{*\}}\, f(z_*)\,dz_*\Big\|_{L^2}.
\end{multline*}
Thus, using again the assumptions on $K,f$, Jensen's inequality with weight $f$, and the mixed a priori estimates of Lemma~\ref{lem:apriori-L2}, we obtain that
\begin{multline*}
\|\langle\nabla_{x_{[2]}}\rangle^{-r}\langle\nabla_{v_{[2]}}\rangle^{-1}M_\e^{[3,m]}((f^{\otimes m})^{1-\delta}R_{N,m}^3)\|_{L^2}\\
\,\lesssim\,m^\frac32N^{-\frac12}\|\bar C_{N,m+1}\|_{L^{2d/(d-2\gamma)}(f^{\otimes2};L^2(f^{\otimes m-1}))}
\,\lesssim\,m^\frac32N^{\frac{2\gamma}d-\frac12}.
\end{multline*}
The other terms in $R_{N,m}$ are handled in the same way, proving~\eqref{eq:todo-An1-re}.

\medskip\noindent
{\bf Step 3:} Conclusion.\\
Inserting the estimates of Steps~1 and~2 into~\eqref{eq:estim-averaging-Dn}, we deduce for all~$\gamma\le r\le1$,
\begin{multline}\label{eq:estim-CNm-rec}
\bigg(\int_0^T\||\nabla_{x_{[2]}}|^{\frac{1-r}4}\langle\nabla_{v_{[2]}}\rangle^{-\frac32}M_\e^{[3,m]}((f^{\otimes m})^{1-\delta}\bar C_{N,m})\|_{L^2}^2\,dt\bigg)^\frac12\\[-1mm]
\lesssim (T+1)^\frac12(m+1)^2\Big(1+\e^rN^\frac{2\gamma}d+N^{-\frac12}\e^{-1-\gamma}+N^{\frac{2\gamma}d-1}\e^{-1-\frac d2}+N^{\frac{2\gamma}d-\frac12}\Big)\\
+(m+1)\bigg(\int_0^T\|\langle\nabla_{x_{[2]}}\rangle^\gamma\langle\nabla_{v_{[2]}}\rangle^{-\frac32}M_\e^{[3,m+2]}((f^{\otimes m+2})^{1-\delta}\bar C_{N,m+2})\|_{L^2}^2\bigg)^\frac12\\
+m\bigg(\int_0^T\|\langle\nabla_{x_{[2]}}\rangle^\gamma\langle\nabla_{v_{[2]}}\rangle^{-\frac32}M_\e^{[3,m]}((f^{\otimes m})^{1-\delta}\bar C_{N,m})\|_{L^2}^2\bigg)^\frac12.
\end{multline}
Provided
\begin{equation}\label{eq:cond-rgamma}
\gamma\le\frac d4\wedge1\qquad\text{and}\qquad\frac{2\gamma}{rd}\le\frac{1}{2(1+\gamma)}\wedge\frac{1-2\gamma/d}{1+d/2},
\end{equation}
we can choose $\e=N^{-\theta}$ for some $\theta>0$ so that the first right-hand side term is bounded by $(T+1)(m+1)^2$ uniformly in~$N$.
It remains to absorb the last two terms.
To this aim, we assume in addition
\begin{equation}\label{assumptiongammar}
\gamma<\frac{1-r}4.
\end{equation}
By interpolation and by the a priori $L^2$-estimate of Lemma~\ref{lem:apriori-L2}, we can then bound
\begin{eqnarray*}
\lefteqn{\bigg(\int_0^T\|\langle\nabla_{x_{[2]}}\rangle^\gamma\langle\nabla_{v_{[2]}}\rangle^{-\frac32}M_\e^{[3,m]}((f^{\otimes m})^{1-\delta}\bar C_{N,m})\|_{L^2}^2\bigg)^\frac12}\\
&\le&\bigg(\int_0^T\|\langle\nabla_{x_{[2]}}\rangle^\frac{1-r}4\langle\nabla_{v_{[2]}}\rangle^{-\frac32}M_\e^{[3,m]}((f^{\otimes m})^{1-\delta}\bar C_{N,m})\|_{L^2}^{2\beta}\bigg)^\frac12\\
&\le& T^{\frac12(1-\beta)}\bigg(\int_0^T\|\langle\nabla_{x_{[2]}}\rangle^\frac{1-r}4\langle\nabla_{v_{[2]}}\rangle^{-\frac32}M_\e^{[3,m]}((f^{\otimes m})^{1-\delta}\bar C_{N,m})\|_{L^2}^2\bigg)^\frac\beta2,
\end{eqnarray*}
with $\beta:=\frac{4\gamma}{1-r}<1$.
Using this to bound the last two terms in~\eqref{eq:estim-CNm-rec}, and setting for abbreviation
\begin{equation*}
X_{N,m}:=\bigg(\int_0^T\big\|\langle\nabla_{x_{[2]}}\rangle^\gamma\langle\nabla_{v_{[2]}}\rangle^{-\frac32}M_\e^{[3,m]}((f^{\otimes m})^{1-\delta}\bar C_{N,m})\big\|_{L^2}^2\bigg)^\frac12,
\end{equation*}
we obtain
\begin{equation*}
X_{N,m}\lesssim (T+1)^\frac12(m+1)^2+(m+1)T^{\frac12(1-\beta)}X_{N,m+2}^\beta+mT^{\frac12(1-\beta)}X_{N,m}^\beta.
\end{equation*}
Hence, by Young's inequality,
\begin{equation*}
X_{N,m}\le C_\beta(T+1)^{\frac12}(m+1)^{C_\beta}+\frac14X_{N,m+2}+\frac12 X_{N,m}.
\end{equation*}
Absorbing the last term and iterating the estimate for $m\le N$, we are led to
\begin{equation*}
X_{N,m}\le C_\beta(T+1)^{\frac12}(m+1)^{C_\beta},
\end{equation*}
and the conclusion follows. Note that the choice $\gamma<\gamma_*$ precisely ensures that~\eqref{eq:cond-rgamma} and~\eqref{assumptiongammar} both hold for some $0\le r<1$.
\end{proof}


\section{Limiting hierarchy and the mean-field limit}
We now use the kinetic regularity estimates of the previous two sections to pass to the limit in the hierarchy~\eqref{eq:lemhier} for dual correlations. We first identify subsequential limits as weak solutions of a limiting hierarchy (Lemma~\ref{lem:limhier}), whose singular interaction terms are well-defined thanks to the inherited regularity. We then establish uniqueness in this class (Proposition~\ref{lem:unique}), exploiting the triangular structure of the hierarchy together with kinetic regularization along the linearized mean-field flow.

These are the two key ingredients for both the mean-field limit and fluctuations. For observables with almost deterministic final data~\eqref{eq:conv-PhiNT-ass-MF}, they allow us to identify the limit with the trivial solution of the limiting hierarchy, thereby concluding the proof of Theorem~\ref{th:main} through the dual reformulation of Proposition~\ref{lem:duality}. The same analysis also applies to more general observables and will provide the basis for our treatment of fluctuations in the next sections.

\subsection{Weak compactness}
We start by passing to the limit in the hierarchy for dual correlations, using the kinetic regularity estimates established in Propositions~\ref{prop:bnd-diff} and~\ref{prop:bnd-nodiff}. For~$K\in L^{2d/(d+2\gamma)}_\loc(\R^d)$, we emphasize that the interaction term $K_f[\bar C_{m+2}]$ in the limiting hierarchy~\eqref{eq:limhier1} is not a priori well-defined: it only makes sense thanks to the additional regularity estimate~\eqref{eq:limhier4}.

\begin{lem}\label{lem:limhier}
Assume the hypotheses of Proposition~\ref{prop:bnd-diff} if $\alpha>0$, those of Proposition~\ref{prop:bnd-nodiff} if~$\alpha=0$, and further assume that rescaled dual correlations converge weakly at final time,
\begin{equation}\label{eq:final-conv-ass}
\bar C_{N,m}^T\rightharpoonup\bar C_m^T\qquad\text{weakly in $L^2(f_T^{\otimes m})$.}
\end{equation}
Then, up to a subsequence, as $N\uparrow\infty$, we have for all $m\ge0$,
\begin{equation}\label{eq:weakconv-t-CNm}
\bar C_{N,m}(t)\rightharpoonup\bar C_m(t)\quad\text{weakly in $L^2(f_t^{\otimes m})$, uniformly for $t\in[0,T]$,}
\end{equation}
for some limit $(\bar C_m)_m$ that satisfies the following estimates: for all $m\ge0$,
\begin{equation}\label{eq:limhier3}
\sup_{[0,T]}\|\bar C_m\|_{L^2(f^{\otimes m})}\le C,
\end{equation}
as well as
\begin{equation}\label{eq:limhier4}
\begin{array}{rlll}
\int_0^T\||\nabla_{x_1}|^\gamma((f^{\otimes m})^{1-\delta}\bar C_m)\|_{L^2}^2&\le& C(T+1)(m+1)^c,&\quad\text{if $\alpha>0$},\\[1mm]
\int_0^T\||\nabla_{x_1}|^\gamma\langle\nabla_{v_{[2]}}\rangle^{-\frac32}((f^{\otimes m})^{1-\delta}\bar C_m)\|_{L^2}^2&\le&C(T+1)(m+1)^c,&\quad\text{if $\alpha=0$}.
\end{array}
\end{equation}
Moreover, the limit satisfies the following limiting hierarchy in the weak sense on~$[0,T]$: for all~$m\ge0$,
\begin{equation}\label{eq:limhier1}
\partial_t\bar C_{m}+L_f^m\bar C_{m}=\sqrt{(m+1)(m+2)}\,K_f[\bar C_{m+2}],
\end{equation}
where $L_f^m,K_f[\cdot]$ are defined in Lemma~\ref{lem:hier},
with final data
\begin{equation}\label{eq:limhier2}
\bar C_m|_{t=T}=\bar C_m^T.
\end{equation}
\end{lem}

\begin{proof}
From the a priori $L^2$-estimates of Lemma~\ref{lem:apriori-L2}, weak compactness ensures that, up to a subsequence as $N\uparrow\infty$, we indeed have for all $m\ge0$,
\begin{equation}\label{eq:apriori-L2-conv}
\bar C_{N,m}\overset*\cvf\bar C_m\quad\text{in $L^\infty(0,T;L^2(f^{\otimes m}))$},
\end{equation}
for some limit $(\bar C_m)_m$.
As all our estimates hold with weight $f$, we shall first show that the limit satisfies the limiting hierarchy~\eqref{eq:limhier1}--\eqref{eq:limhier2} in the following $f$-{\it weighted} weak sense: for all $m\ge0$ and $H_m\in C^\infty_c((0,T]\times\Dd^m)$,
\begin{multline}\label{eq:weak-weighted-form}
\int_0^T\int_{\Dd^m} \bar C_mf^{\otimes m}(\partial_t+\hat L_f^m)H_m\\
+\sqrt{(m+1)(m+2)}\int_0^T\int_{\Dd^{m}} H_m K_f[\bar C_{m+2}] f^{\otimes m}
=\int_{\Dd^m}H_m(T)\,\bar C_m^T\,f_T^{\otimes m},
\end{multline}
where $\hat L_f^m$ stands for the following adjoint operator on the $m$-particle space,
\begin{eqnarray}
\hat L_f^m&:=&\sum_{j=1}^m\Id^{\otimes j-1}\otimes \hat L_f\otimes\Id^{\otimes m-j},\nonumber\\
\hat L_fh&:=&v\cdot\nabla_{x}h-\alpha\triangle_{v}h+(K\ast f)(x)\cdot\nabla_{v}h-2\alpha\nabla_{v}\log f(z)\cdot\nabla_{v}h\nonumber\\
&&+\int_\Dd K_f(z,z_*)\,h(z_*)f(z_*)dz_*.\label{eq:def-hat-Lf}
\end{eqnarray}
Using that $f$ satisfies the Vlasov equation and using the positivity requirement for~$f$ in~\eqref{eq:reg-req-0b}, this easily implies that $(\bar C_m)_m$ satisfies~\eqref{eq:limhier1}--\eqref{eq:limhier2} in the usual weak sense.

We split the proof of~\eqref{eq:weak-weighted-form} into three steps: we start by showing that the remainder term $R_{N,m}$ in the hierarchy of equations for dual correlations derived in Lemma~\ref{lem:hier} tends to $0$ in the sense of distributions, and then we conclude the proof, separately considering the cases $\alpha>0$ and $\alpha=0$. Once~\eqref{eq:weak-weighted-form} is established, the final step of the proof is to show that~\eqref{eq:apriori-L2-conv} can be upgraded into uniform-in-time weak convergence~\eqref{eq:weakconv-t-CNm}.

\medskip\noindent
{\bf Step~1:} Remainder estimate: for all $m\ge0$,
\begin{equation}\label{eq:conv-RNm}
f^{\otimes m}R_{N,m}\to0\quad\text{in $\Dc'([0,T]\times\Dd^m)$}.
\end{equation}
For shortness, we focus again on the three typical terms~$R_{N,m}^1,R_{N,m}^2,R_{N,m}^3$ considered in~\eqref{eq:def-RN123}.
Integrating by parts and using the assumptions on $K,f$, we easily find for any smooth symmetric test function $H_m\in C^\infty_c([0,T]\times\Dd^m)$,
\begin{eqnarray*}
\Big|\int_0^T\int_{\Dd^m} H_mR_{N,m}^1f^{\otimes m}\Big|
&\lesssim_m&N^{-\frac12}\int_0^T\|\bar C_{N,m-1}\|_{L^2(f^{\otimes m-1})},\\
\Big|\int_0^T\int_{\Dd^m} H_mR_{N,m}^2f^{\otimes m}\Big|
&\lesssim_m&N^{-1}\int_0^T\|\bar C_{N,m}\|_{L^{2d/(d-2\gamma)}(f^{\otimes2};L^2(f^{\otimes m-2}))},\\
\Big|\int_0^T\int_{\Dd^m} H_mR_{N,m}^3f^{\otimes m}\Big|
&\lesssim_m&N^{-\frac12}
\int_0^T \|\bar C_{N,m+1}\|_{L^{2d/(d-2\gamma)}(f^{\otimes2};L^2(f^{\otimes m-1}))}.
\end{eqnarray*}
By the mixed a priori estimates of Lemma~\ref{lem:apriori-L2}, all three right-hand sides tend to $0$ as~$\gamma<\frac d4$, thus proving~\eqref{eq:conv-RNm}.

\medskip\noindent
{\bf Step~2:} Deriving the limiting hierarchy: case $\alpha>0$.\\
Proposition~\ref{prop:bnd-diff} further implies, along the extracted subsequence~\eqref{eq:apriori-L2-conv}, for all~$m\ge0$,
\begin{equation*}
(f^{\otimes m})^{1-\delta}\bar C_{N,m}\cvf (f^{\otimes m})^{1-\delta}\bar C_m\quad\text{in $L^2(0,T;H^\gamma_{x}L^2_{v}(\Dd;L^2(\Dd^{m-1})))$.}
\end{equation*}
By definition of the operator $K_f[\cdot]$, cf.~\eqref{eq:def-VC}, and by the assumptions on $K,f$, this entails for all $m\ge0$,
\begin{equation}\label{eq:conv-V}
(f^{\otimes m})^{1-\delta}K_f[\bar C_{N,m+2}]\cvf (f^{\otimes m})^{1-\delta}K_f[\bar C_{m+2}]\quad\text{in $L^2(0,T;L^2(\Dd^{m}))$}.
\end{equation}
This allows us to pass to the limit in the hierarchy of equations for dual correlations. Indeed, as $f$ satisfies the Vlasov equation, the weak formulation of the hierarchy derived in Lemma~\ref{lem:hier} can be written as follows, for all $0\le m\le N$ and $H_m\in C^\infty_c((0,T]\times\Dd^m)$,
\begin{multline}\label{eq:weak-weight-CNm}
\int_0^T\int_{\Dd^m}\bar C_{N,m}f^{\otimes m}(\partial_t+\hat L_f^m)H_m+\int_0^T\int_{\Dd^m}H_mR_{N,m}f^{\otimes m}\\
+\sqrt{(m+1)(m+2)}\Big(\frac{(N-m)(N-m-1)}{(N-1)^2}\Big)^\frac12\int_0^T\int_{\Dd^{m}}H_mK_f[\bar C_{N,m+2}] f^{\otimes m}\\
=\int_{\Dd^m}H_m(T)\bar C_{N,m}^Tf(T)^{\otimes m}.
\end{multline}
Using~\eqref{eq:apriori-L2-conv}, \eqref{eq:conv-RNm}, \eqref{eq:conv-V}, and recalling assumption~\eqref{eq:final-conv-ass},
we conclude that the limit $(\bar C_m)_m$ satisfies the weighted weak formulation~\eqref{eq:weak-weighted-form} of the limiting dual hierarchy~\eqref{eq:limhier1}--\eqref{eq:limhier2}.
Finally, the kinetic regularity estimate~\eqref{eq:limhier4} follows from Proposition~\ref{prop:bnd-diff} by lower semicontinuity.

\medskip\noindent
{\bf Step~3:} Deriving the limiting hierarchy: case $\alpha=0$.\\
As $M_\e\to\Id$ in the strong operator topology on $L^2(\Dd)$ as $\e\downarrow0$, it follows from~\eqref{eq:apriori-L2-conv} and from the boundedness of $f$ that, along the extracted subsequence,
\[M_\e^{[3,m]}((f^{\otimes m})^{1-\delta}\bar C_{N,m})\overset*\cvf (f^{\otimes m})^{1-\delta}\bar C_m\quad\text{in $L^\infty(0,T;L^2(\Dd^m))$,}\]
and Proposition~\ref{prop:bnd-nodiff} then further implies
\begin{equation}\label{eq:conv-reg-Meps-CNm}
M_\e^{[3,m]}((f^{\otimes m})^{1-\delta}\bar C_{N,m})\cvf(f^{\otimes m})^{1-\delta}\bar C_m\quad\text{in $L^2(0,T;H^{\gamma}_{x}H^{-3/2}_{v}(\Dd^2;L^2(\Dd^{m-2})))$}.
\end{equation}
By definition of the operator $K_f[\cdot]$, cf.~\eqref{eq:def-VC}, and by the assumptions on $K,f$, this entails for all $m\ge0$,
\begin{equation}\label{eq:conv-V-re}
M_\e^{[m]}((f^{\otimes m})^{1-\delta}K_f[\bar C_{N,m+2}])\cvf (f^{\otimes m})^{1-\delta}K_f[\bar C_{m+2}]\quad\text{in $L^2(0,T;L^2(\Dd^m))$}.
\end{equation}
Together with~\eqref{eq:apriori-L2-conv}, \eqref{eq:conv-RNm}, and with the convergence assumption on the final data, this allows to pass to the limit in~\eqref{eq:weak-weight-CNm} with test function $H_m$ replaced by $M_\e^{[m]}H_m$, which then implies that the limit $(\bar C_m)_m$ satisfies the weighted weak formulation~\eqref{eq:weak-weighted-form} of the limiting hierarchy~\eqref{eq:limhier1}--\eqref{eq:limhier2}. Finally, the kinetic regularity estimate~\eqref{eq:limhier4} follows from Proposition~\ref{prop:bnd-nodiff} together with~\eqref{eq:conv-reg-Meps-CNm} by lower semicontinuity.

\medskip\noindent
{\bf Step~4:} Uniform-in-time weak convergence.\\
We finally upgrade~\eqref{eq:apriori-L2-conv} to uniform-in-time weak convergence~\eqref{eq:weakconv-t-CNm}.
To this aim, by the Arzel\`a-Ascoli argument, it suffices to show that, for any fixed $m\ge0$ and $H_m\in C_c^\infty(\Dd^m)$, the family of functions
\[a_{N,m}(t):=\int_{\Dd^m}H_m\bar C_{N,m}(t) f_t^{\otimes m}\]
is asymptotically equicontinuous on $[0,T]$ as $N\uparrow\infty$.
The weighted weak formulation~\eqref{eq:weak-weight-CNm} yields for all $0\le s\le t\le T$,
\begin{multline*}
a_{N,m}(t)-a_{N,m}(s)=\int_s^t\int_{\Dd^m}\bar C_{N,m}f^{\otimes m}\hat L_f^mH_m+\int_s^t\int_{\Dd^m}H_m R_{N,m}f^{\otimes m}\\
+\sqrt{(m+1)(m+2)}\Big(\frac{(N-m)(N-m-1)}{(N-1)^2}\Big)^{\frac12}\int_s^t\int_{\Dd^m}H_mK_f[\bar C_{N,m+2}]f^{\otimes m}.
\end{multline*}
By the uniform $L^2$-bounds, the first term is bounded by $O(t-s)$, while the result~\eqref{eq:conv-RNm} of Step~1 ensures that the second term tends to $0$ as $N\uparrow\infty$, uniformly in $0\le s\le t\le T$. For the last, interaction term, appealing to the regularity estimates of Propositions~\ref{prop:bnd-diff} and~\ref{prop:bnd-nodiff} similarly as in Steps~2 and~3, we easily deduce that it is bounded by $O(\sqrt{t-s})$ as $N\uparrow\infty$, uniformly in $0\le s\le t\le T$. This proves the asymptotic equicontinuity of $(a_{N,m})_N$ on~$[0,T]$ and concludes the proof.
\end{proof}

\subsection{Uniqueness for the limiting hierarchy}\label{sec:uniqueness}
A direct uniqueness argument for the limiting dual hierarchy~\eqref{eq:limhier1} would be straightforward if $\bar C_m$ were controlled in the space $L^{2d/(d-2\gamma)}(f^{\otimes m})$. The kinetic regularity estimate~\eqref{eq:limhier4}, however, only yields a fixed gain of spatial regularity in one particle variable at a time. Combining these one-particle estimates into an isotropic Sobolev estimate in the $m$-particle space would only give an~$O(\frac1m)$ gain of spatial integrability for~$\bar C_m$, which degenerates as $m\uparrow\infty$ and is thus a priori insufficient to control the singular interaction operator uniformly along the hierarchy.

This is reminiscent of the hierarchical derivation of the quantum mean-field limit for bosons with Coulomb interactions in~\cite{Bardos-Golse-Mauser-00,Erdos-Yau-01}. The central ingredient in~\cite{Erdos-Yau-01} is the propagation of an iterative Sobolev norm controlling one derivative in every particle variable simultaneously. Such tensorized regularity can be derived from energy estimates based on the conservation of the many-body Hamiltonian and is strong enough to close a direct uniqueness argument for the limiting quantum hierarchy. No analogous tensorized estimate appears to be available in the present classical kinetic setting, where regularity is produced by hypoellipticity or velocity averaging.

We instead exploit the explicit triangular structure of the limiting hierarchy through its Duhamel expansion. In each iterated term, every occurrence of the singular interaction kernel appears to be integrated in time along the linearized mean-field propagator. Kinetic regularization can therefore be applied to this integrated kernel, while all remaining particle variables are propagated merely in $L^2$. This replaces the tensorized regularity that is unavailable for the classical hierarchy.
Note that for this uniqueness result the restriction on $\gamma$, cf.~\eqref{eq:red-gamma}, is much weaker than for the uniform-in-$N$ regularity estimates of the previous sections; see also Remark~\ref{rem:restr-gamma-uniqueness} below.

\begin{prop}\label{lem:unique}
Let
\[\alpha\ge0,\qquad K\in L^{2d/(d+2\gamma)}_\loc(\R^d),\qquad\sup_{|x|\ge1}|K(x)|<\infty,\]
for some
\begin{equation}\label{eq:red-gamma}
0\le\gamma\le\left\{\begin{array}{lll}
1/3&:&\alpha>0,\\
1/4&:&\alpha=0.
\end{array}\right.
\end{equation}
Also assume that the mean-field solution $f$ satisfies the regularity conditions~\eqref{eq:reg-req-0}, \eqref{eq:reg-req-0b}, together with~\eqref{eq:reg-req-2}--\eqref{eq:reg-req-2b} if $\alpha>0$, and with~\eqref{eq:reg-req-1} if $\alpha=0$, for some~$\delta\in(0,\frac12]$.
For $0\le T\le T_0$,  let then~$(C_m^1)_m$ and~$(C_m^2)_m$ satisfy the limiting dual hierarchy~\eqref{eq:limhier1} in the weak sense on~$[0,T]$,
\begin{equation}\label{eq:limhier-re-1}
\partial_tC_{m}^i+L_f^mC_{m}^i=\sqrt{(m+1)(m+2)}\,K_f[C_{m+2}^i],\quad i=1,\;2,
\end{equation}
with the same final data
\[ C_m^1|_{t=T}=C_m^2|_{t=T}\quad\text{for all $m\ge0$},\]
with $C_m^i$ symmetric,
and with the following a priori estimates: for some $L\ge1$,
\begin{equation}\label{uniformL2Cmi}
\sup_{[0,T]}\,\sup_m \|C_m^i\|_{L^2(f^{\otimes m})}\le L,
\end{equation}
as well as
\begin{equation}\label{eq:limhier-ap2}
\begin{array}{rlll}
\int_0^T\||\nabla_{x_1}|^\gamma((f^{\otimes m})^{1-\delta}C_m^i)\|_{L^2}^2&\le& L^{m+1},&\quad\text{if $\alpha>0$},\\[1mm]
\int_0^T\||\nabla_{x_1}|^\gamma\langle\nabla_{v_{[2]}}\rangle^{-\frac32}((f^{\otimes m})^{1-\delta}C_m^i)\|_{L^2}^2&\le&L^{m+1},&\quad\text{if $\alpha=0$}.
\end{array}
\end{equation}
Then
\[C_m^1\equiv C_m^2\quad\text{on $[0,T]$,~~for all~$m\ge0$.}\]
\end{prop}

\begin{rem}\label{rem:restr-gamma-uniqueness}
Using sharper kinetic regularity estimates in the proof of Lemma~\ref{lem:bnd-BK-delta} below would extend the admissible range of $\gamma$ in this uniqueness result to
\begin{equation}\label{eq:restr-gamma-unique-impr}
0\le\gamma\le\left\{\begin{array}{lll}
5/6&:&\alpha>0,\\
1/2&:&\alpha=0,
\end{array}\right.
\end{equation}
up to strengthening the regularity requirements on $f$. We do not pursue this refinement here, since the regularity estimates of Propositions~\ref{prop:bnd-diff} and~\ref{prop:bnd-nodiff} impose substantially stronger restrictions on~$\gamma$ in any case.
\end{rem}

Before we proceed to the proof of the above proposition, we start with two preliminary results. First, we construct the non-autonomous propagator $(B_{t,s})_{0\le s\le t\le T}$ generated by the operator $\hat L_f$, cf.~\eqref{eq:def-hat-Lf}, which is the space-time adjoint of~$L_f$ in the moving $f$-weighted pairing, in the sense that
\[\partial_t\int_\Dd g_th_tf_t=\int_\Dd(\partial_tg_t+L_{f_t}g_t)h_tf_t+\int_{\Dd}g_t(\partial_th_t+\hat L_{f_t}h_t)f_t.\]
The natural space for this evolution is $L^2(f)$, where its construction is standard. What will be important for the sequel is the additional fact that the same propagator acts on~$L^2(f^{2\delta})$, with $\delta$ as in~\eqref{eq:reg-req-1}--\eqref{eq:reg-req-2b}. Some care is needed for this stronger estimate, which will allow us later to pair the propagator with the regularity estimates~\eqref{eq:limhier-ap2} carrying the complementary weight~$f^{1-\delta}$.

\begin{lem}\label{lem:propag-B}
Assume the hypotheses of Proposition~\ref{lem:unique}. There exists an evolution family
\[B_{t,s}:L^2(f_s)\to L^2(f_t),\qquad 0\le s\le t\le T,\]
such that
\begin{equation}\label{eq:bnd-Bts0}
\|B_{t,s}\|_{\Lc(L^2(f_s),L^2(f_t))}\le e^{C(t-s)},
\end{equation}
with the semigroup property
\begin{equation}\label{eq:semigr-Bts}
B_{t,r}B_{r,s}=B_{t,s},\qquad 0\le s\le r\le t\le T,
\end{equation}
and such that, in the weak sense,
\begin{equation}\label{eq:gen-Bts}
\left\{\begin{array}{ll}
(\partial_t+\hat L_{f_t})B_{t,s}=0,&0\le s\le t\le T,\\
B_{s,s}=\Id.&
\end{array}\right.
\end{equation}
Moreover, the restriction of $B_{t,s}$ to $L^2(f_s^{2\delta})$ defines a bounded operator $L^2(f_s^{2\delta})\to L^2(f_t^{2\delta})$ with
\begin{equation}\label{eq:bnd-Bts}
\|B_{t,s}\|_{\Lc(L^2(f_s^{2\delta}),L^2(f_t^{2\delta}))}\le e^{C(t-s)},
\end{equation}
\end{lem}

\begin{proof}[Proof of Lemma~\ref{lem:propag-B}]
By definition~\eqref{eq:def-hat-Lf}, we can decompose the operator as
\begin{equation}\label{eq:decomp-hatLf}
\hat L_f:=\hat L_f^0+A_f,
\end{equation}
where
\begin{eqnarray*}
\hat L_f^0h&:=&v\cdot\nabla_{x}h-\alpha\triangle_{v}h+(K\ast f)(x)\cdot\nabla_{v}h-2\alpha\nabla_{v}\log f(z)\cdot\nabla_{v}h,\\
A_fh&:=&\int_\Dd K_f(z,z_*)h(z_*)f(z_*)dz_*.
\end{eqnarray*}
We argue by a perturbation argument, starting from the propagator $(P_{t,s})_{0\le s\le t\le T}$ generated by~$\hat L_f^0$.
We split the proof into three steps.

\medskip\noindent
{\bf Step~1:} Construction on $L^2(f)$.\\
We first consider the evolution generated by $\hat L_f^0$. Since $K\ast f\in L^1(0,T;W^{1,\infty}(\R^d))$ by assumption, the time-inhomogeneous stochastic differential equation
\begin{equation}\label{eq:markov-XV}
dX_t=V_t\,dt,\qquad dV_t=(K*f_t)(X_t)\,dt+\sqrt{2\alpha}\,dB_t
\end{equation}
is well-posed. We take initial data $Z_0=(X_0,V_0)$ with law $f_\circ$. Since $f$ solves the corresponding Vlasov equation, uniqueness ensures that $Z_t$ has law $f_t$ for all $0\le t\le T$. We then define the backward transition operator
\begin{equation}\label{eq:def-Pts}
(P_{t,s}\varphi)(z):=\E[\varphi(Z_s)\|Z_t=z],\qquad 0\le s\le t\le T,
\end{equation}
which is defined $f_t$-almost everywhere. The Markov property yields
\begin{equation}\label{eq:semigr-Pts}
P_{t,r}P_{r,s}=P_{t,s},\qquad 0\le s\le r\le t\le T,
\end{equation}
while Jensen's inequality gives
\begin{equation}\label{eq:bnd-Pts0}
\|P_{t,s}\|_{\Lc(L^2(f_s),L^2(f_t))}\le1,\qquad0\le s\le t\le T.
\end{equation}
By duality with the forward transition kernel, $(P_{t,s})_{0\le s\le t\le T}$ solves in the weak sense
\begin{equation}\label{eq:gen-Pts}
\left\{\begin{array}{ll}
(\partial_t+\hat L_{f_t}^0)P_{t,s}=0,&0\le s\le t\le T,\\[1mm]
P_{s,s}=\Id.&
\end{array}\right.
\end{equation}

We next add the bounded perturbation $A_f$. To this aim, we first note that this perturbation is bounded: indeed, by defintion~\eqref{eq:decomp-hatLf},
\begin{multline*}
\|A_fh\|_{L^2(f)}
\le\Big(\|K\|_{L^1(B)}\|f^{1/2}\|_{L^\infty_xL^2_v}\|f^{1/2}\nabla_v\log f\|_{L^\infty_xL^2_v}\\
+\big(\|K\|_{L^\infty(B^c)}+\|K\ast f\|_{L^\infty}\big)\|f^{1/2}\|_{L^2}\|f^{1/2}\nabla_v\log f\|_{L^2}\Big)
\|h\|_{L^2(f)},
\end{multline*}
hence, by the assumptions on $K,f$, on $[0,T]$,
\begin{equation}\label{eq:bnd-Af0}
\|A_f\|_{\Lc(L^2(f))}\le C.
\end{equation}
We can now construct the full propagator by the convergent Dyson series
\begin{equation}\label{eq:dyson-B}
B_{t,s}=P_{t,s} +\sum_{n=1}^\infty(-1)^n\int_{s<r_n<\ldots<r_1<t}P_{t,r_1}A_{f_{r_1}}P_{r_1,r_2}\ldots A_{f_{r_n}}P_{r_n,s}\,dr_1\cdots dr_n.
\end{equation}
Indeed, by~\eqref{eq:bnd-Pts0} and~\eqref{eq:bnd-Af0}, the norm of the $n$th term is bounded by $\frac1{n!}(C(t-s))^n$,
thus proving~\eqref{eq:bnd-Bts0} after summation. The semigroup property~\eqref{eq:semigr-Bts} is easily checked from the above series using~\eqref{eq:semigr-Pts}. The series also implies the Duhamel formula
\begin{equation*}
B_{t,s}=P_{t,s}-\int_s^tP_{t,r}A_{f_r}B_{r,s}\,dr,
\end{equation*}
which entails that $B_{t,s}$ satisfies~\eqref{eq:gen-Bts} since $P_{t,s}$ satisfies~\eqref{eq:gen-Pts}.

\medskip\noindent
{\bf Step~2:} Weighted estimate for $(P_{t,s})_{0\le s\le t\le T}$: proof that
\begin{equation}\label{eq:bnd-Pts-delta}
\|P_{t,s}\|_{\Lc(L^2(f_s^{2\delta}),L^2(f_t^{2\delta}))}\le e^{C(t-s)},\qquad 0\le s\le t\le T.
\end{equation}
Since $f$ is bounded and $\delta\le\frac12$, we have $L^2(f^{2\delta})\subset L^2(f)$, so the operator $P_{t,s}$ is already defined on~$L^2(f_s^{2\delta})$, and it remains to estimate its norm.
Given $0\le s\le T$ and $h_s\in L^2(f_s)$, let~$h_t:=P_{t,s}h_s$ for $s\le t\le T$.

We start with the diffusive case $\alpha>0$. Inserting the weight $f^\delta$ in the equation and using the Vlasov equation for~$f$, equation~\eqref{eq:gen-Pts} yields
\begin{multline*}
\Big(\partial_t+v\cdot\nabla_x-\alpha\triangle_v+(K\ast f)(x)\cdot\nabla_v\Big)f^\delta h\\
=2\alpha(1-\delta)\nabla_v\log f\cdot\nabla_v(f^\delta h)-\alpha\delta(1-\delta)|\nabla_v\log f|^2f^\delta h,
\end{multline*}
from which we can deduce the energy identity
\begin{equation}\label{eq:energy-local-delta}
\partial_t\|f^\delta h\|_{L^2}^2+2\alpha\|\nabla_v(f^\delta h)\|_{L^2}^2
+2\alpha(1-\delta)\int_\Dd\Big(\triangle_v\log f+\delta|\nabla_v\log f|^2\Big)|f^\delta h|^2=0.
\end{equation}
By the assumption~\eqref{eq:reg-req-2b} on $f$ and by Gronwall's inequality, we deduce for all $s\le t\le T$,
\[\|f_t^\delta h_t\|_{L^2}\le e^{C(t-s)}\|f_s^\delta h_s\|_{L^2},\]
thus proving~\eqref{eq:bnd-Pts-delta}.

We turn to the deterministic case $\alpha=0$. As $f$ is transported by the volume-preserving characteristic flow 
\[\dot X_t=V_t,\qquad\dot V_t=(K\ast f_t)(X_t),\]
we deduce
\begin{equation*}
\|f_t^{\delta}P_{t,s}h_s\|_{L^2}=\|f_s^{\delta}h_s\|_{L^2},
\end{equation*}
thus proving in particular~\eqref{eq:bnd-Pts-delta}.

\medskip\noindent
{\bf Step~3:} Weighted estimate for the full propagator.\\
It remains to show that the propagator $(B_{t,s})_{0\le s\le t\le T}$ constructed in Step~1 preserves the weighted spaces. Similarly as in~\eqref{eq:bnd-Af0}, the assumptions on $K,f$ ensure
\begin{equation}\label{eq:bnd-Af}
\|A_f\|_{\Lc(L^2(f^{2\delta}))}\le C.
\end{equation}
Using this together with the result~\eqref{eq:bnd-Pts-delta} of Step~2 to estimate the Dyson series~\eqref{eq:dyson-B}, the claim~\eqref{eq:bnd-Bts} follows.
\end{proof}

Next, we appeal to kinetic regularization to establish the following key estimate on the time integral of the singular interaction kernel along the linearized mean-field propagator. For later purposes, we formulate the result in a more general setting, further considering time integrals of arbitrary kernels.

\begin{lem}\label{lem:bnd-BK-delta}
Assume the hypotheses of Proposition~\ref{lem:unique}. 
For all $0\le s\le t\le T$, and any symmetric function $h:[0,T]\times\Dd^2\to\R$ in the appropriate Sobolev space, we have
\begin{multline*}
\left\|\int_s^t B_{t,r}^{\otimes 2} h_r\,dr\right\|_{L^2((f_t^{\otimes 2})^{2\delta})}\\
\leq Ce^{C(t-s)}\,\left\{\begin{array}{lll}
\|\langle\nabla_{x_{1}}\rangle^{-\gamma}((f^{\otimes 2})^\delta h)\|_{L^2(s,t;L^2(\Dd^2))}&:&\alpha>0,\\
\|\langle\nabla_{x_{1}}\rangle^{-\gamma}\langle\nabla_{v_{[2]}}\rangle^{\frac32}((f^{\otimes 2})^\delta h)\|_{L^2(s,t;L^2(\Dd^2))}&:&\alpha=0.
\end{array}\right.
\end{multline*}
In particular, choosing $h=K_f$,
\begin{equation}\label{eq:bnd-BK-delta}
\Big\|\int_s^tB_{t,r}^{\otimes2}K_{f_r}dr\Big\|_{L^2((f_t^{\otimes2})^{2\delta})}\le Ce^{C(t-s)}\sqrt{t-s}.
\end{equation}
\end{lem}

\begin{proof}
By density, it suffices to consider $h\in L^\infty(0,T;L^2((f^{\otimes2})^{2\delta}))$. Let $s=0$ without loss of generality, and set for abbreviation
\[S_t(z_1,z_2):=\int_0^tB_{t,r}^{\otimes2} h_r(z_1,z_2)\,dr,\]
viewed as an element of $C_w(0,T;L^2((f^{\otimes2})^{2\delta}))$, cf.~Lemma~\ref{lem:propag-B}.
By definition of the propagator, it satisfies in the weak sense
\begin{equation}\label{eq:defin-S-eqn}
(\partial_t+\hat L_{f}^2)S=h,\qquad S|_{t=0}=0.
\end{equation}
Performing the same steps as for the energy inequality~\eqref{eq:energy-local-delta} in each variable, we obtain
\begin{multline*}
\partial_t\|(f^{\otimes2})^\delta S\|_{L^2}^2
+2\alpha(1-\delta)\sum_{j=1}^2\int_{\Dd^2}\Big(\triangle_v\log f(z_j)+\delta|\nabla_v\log f(z_j)|^2\Big)|(f^{\otimes2})^\delta S|^2\\
\le2\int_{\Dd^2} Sh(f^{\otimes2})^{2\delta}-2\int_{\Dd^2}S A_f^2(S)(f^{\otimes2})^{2\delta},
\end{multline*}
in terms of
\[A_f^2(S):=\int_\Dd K_f(z_1,z_*)S(z_*,z_2)f(z_*)\,dz_*
+\int_\Dd K_f(z_2,z_*)S(z_1,z_*)f(z_*)\,dz_*.\]
By the assumption~\eqref{eq:reg-req-2b} on $f$ in case $\alpha>0$, and by the operator bound~\eqref{eq:bnd-Af} on $A_f$, we deduce
\begin{equation*}
\partial_t\|(f^{\otimes2})^\delta S\|_{L^2}^2
\le2\int_{\Dd^2} Sh(f^{\otimes2})^{2\delta}+C\|(f^{\otimes2})^\delta S\|_{L^2}^2,
\end{equation*}
and thus, by Gronwall's inequality,
\begin{equation}\label{eq:S-energy}
\|(f_t^{\otimes2})^\delta S_t\|_{L^2}^2
\le2e^{Ct}\int_0^t\Big|\int_{\Dd^2} Sh(f^{\otimes2})^{2\delta}\Big|.
\end{equation}
In order to estimate the right-hand side, we appeal to kinetic regularization. We split the proof into two steps, separately considering the cases $\alpha>0$ and $\alpha=0$.

\medskip\noindent
{\bf Step~1:} Diffusive case $\alpha>0$.\\
We do not apply hypoelliptic regularity directly to $(f^{\otimes2})^\delta S$: instead, we use a slightly stronger velocity weight $(f^{\otimes2})^\beta$ with $\beta=\frac32\delta>\delta$.
Inserting this weight and using the equation for $f$, a direct computation allows to rewrite the equation~\eqref{eq:defin-S-eqn} for $S$ as
\begin{equation}\label{eq:Sbeta-eqn}
\partial_t((f^{\otimes2})^\beta S)+\sum_{j=1}^2(v_j\cdot\nabla_{x_j}-\alpha\triangle_{v_j})((f^{\otimes2})^\beta S)=E+\sum_{j=1}^2\Div_{v_j}(F_j),
\end{equation}
with
\begin{eqnarray*}
E&:=&(f^{\otimes2})^\beta h-(f^{\otimes2})^\beta A_f^2(S)\\[-1mm]
&&-\alpha(1-\beta)(f^{\otimes2})^\beta S\sum_{j=1}^2\Big(2\triangle_v\log f(z_j)+\beta|\nabla_v\log f(z_j)|^2\Big)\\[-2mm]
F_j&:=&-\Big(K\ast f(x_j)-2\alpha(1-\beta)\nabla_v\log f(z_j)\Big)(f^{\otimes2})^\beta S.
\end{eqnarray*}
Since $f^\beta=f^{\delta/2}f^\delta$, we can take advantage of the additional factor $f^{\delta/2}$ to bound the terms in the last two lines. By the assumptions on $K,f$, and recalling the bound~\eqref{eq:bnd-Af} on $A_f$, we get
\begin{equation*}
\|\langle\nabla_X\rangle^{-\gamma}E\|_{L^2}+\|F_j\|_{L^2}\lesssim \|\langle\nabla_{x_{1}}\rangle^{-\gamma}((f^{\otimes 2})^\delta h)\|_{L^2(\Dd^2)}+\|(f^{\otimes2})^\delta S\|_{L^2},
\end{equation*}
where we set for abbreviation $Z=(X,V)=(z_1,z_2)$.
Hence, by the refined hypoelliptic estimate~\eqref{lem:hypoell-re} in Lemma~\ref{lem:hypoell},
\begin{equation*}
\int_0^t\|\langle\nabla_X\rangle^{(\frac23-\gamma)\wedge\frac13}((f^{\otimes2})^\beta S)\|_{L^2}^2
\lesssim (t+1)\sup_{[0,t]}\|(f^{\otimes2})^\delta S\|_{L^2}^2+\int_0^t\|\langle\nabla_{x_{1}}\rangle^{-\gamma}((f^{\otimes 2})^\delta h)\|_{L^2}^2.
\end{equation*}
By the assumptions on $K,f$, decomposing $(f^{\otimes2})^{2\delta}=(f^{\otimes2})^{\beta}(f^{\otimes2})^{\delta/2}$,
and noting that the restriction $\gamma\le\frac13$ ensures $(\frac23-\gamma)\wedge\frac13\ge\gamma$, we thus obtain
\begin{eqnarray*}
\int_0^t\Big|\int_{\Dd^2}Sh(f^{\otimes2})^{2\delta}\Big|
&\lesssim&\|\langle\nabla_X\rangle^\gamma((f^{\otimes2})^\beta S)\|_{L^2(0,t;L^2(\Dd^2))}\|\langle\nabla_{x_{1}}\rangle^{-\gamma}((f^{\otimes 2})^\delta h)\|_{L^2(0,t;L^2(\Dd^2))}\\
&\lesssim& \|\langle\nabla_{x_{1}}\rangle^{-\gamma}((f^{\otimes 2})^\delta h)\|_{L^2(0,t;L^2(\Dd^2))}^2\\
&&+(t+1)\|\langle\nabla_{x_{1}}\rangle^{-\gamma}((f^{\otimes 2})^\delta h)\|_{L^2(0,t;L^2(\Dd^2))}\,\sup_{[0,t]}\|(f^{\otimes2})^\delta S\|_{L^2}.
\end{eqnarray*}
Combined with the energy estimate~\eqref{eq:S-energy}, this leads us to
\begin{equation*}
\|(f_t^{\otimes2})^\delta S_t\|_{L^2}^2\lesssim e^{Ct}\,\|\langle\nabla_{x_{1}}\rangle^{-\gamma}((f^{\otimes 2})^\delta h)\|_{L^2(0,t;L^2(\Dd^2))}^2,
\end{equation*}
and the conclusion follows.

\medskip\noindent
{\bf Step~2:} Deterministic case $\alpha=0$.\\
In this case, we can directly apply the averaging lemma to $(f^{\otimes2})^\delta S$, without needing to modify the weight. Equation~\eqref{eq:Sbeta-eqn} for $\alpha=0$, with $\beta$ replaced by $\delta$, reads
\begin{equation*}
\partial_t((f^{\otimes2})^\delta S)+\sum_{j=1}^2v_j\cdot\nabla_{x_j}((f^{\otimes2})^\delta S)=E+\sum_{j=1}^2\Div_{v_j}F_j,
\end{equation*}
with
\begin{eqnarray*}
E&:=&(f^{\otimes2})^\delta h-(f^{\otimes2})^\delta A_f^2(S),\\
F_j&:=&-(K\ast f)(x_j)(f^{\otimes2})^\delta S.
\end{eqnarray*}
By the assumptions on $K,f$, we find
\begin{equation*}
\|\langle\nabla_X\rangle^{-\gamma}E\|_{L^2}+\|F_j\|_{L^2}
\lesssim \|\langle\nabla_{x_{1}}\rangle^{-\gamma}\langle\nabla_{v_{[2]}}\rangle^{\frac32}((f^{\otimes 2})^\delta h)\|_{L^2}+\|(f^{\otimes2})^\delta S\|_{L^2}.
\end{equation*}
Hence, by the refined averaging estimate in Lemma~\ref{lem:averaging},
\begin{multline*}
\int_0^t\||\nabla_X|^{\frac{1-\gamma}{2}\wedge\frac14}\langle\nabla_V\rangle^{-\frac32}((f^{\otimes2})^\delta S)\|_{L^2}^2\\[-2mm]
\lesssim (t+1)\sup_{[0,t]}\|(f^{\otimes2})^\delta S\|_{L^2}^2+\int_0^t\|\langle\nabla_{x_{1}}\rangle^{-\gamma}\langle\nabla_{v_{[2]}}\rangle^{\frac32}((f^{\otimes 2})^\delta h)\|_{L^2}^2.
\end{multline*}
By the assumptions on $K,f$, noting that the restriction $\gamma\le\frac14$ yields $\frac{1-\gamma}2\wedge\frac14\ge\gamma$, we deduce that
\begin{eqnarray*}
\lefteqn{\int_0^t\Big|\int_{\Dd^2}Sh(f^{\otimes2})^{2\delta}\Big|}\\
&\lesssim&\|\langle\nabla_X\rangle^\gamma\langle\nabla_V\rangle^{-\frac32}((f^{\otimes2})^\delta S)\|_{L^2(0,t;L^2(\Dd^2))}\|\langle\nabla_{x_{1}}\rangle^{-\gamma}\langle\nabla_{V}\rangle^{\frac32}((f^{\otimes 2})^\delta h)\|_{L^2(0,t;L^2(\Dd^2))}\\
&\lesssim&\|\langle\nabla_{x_{1}}\rangle^{-\gamma}\langle\nabla_{V}\rangle^{\frac32}((f^{\otimes 2})^\delta h)\|_{L^2(0,t;L^2(\Dd^2))}^2\\
&&+(t+1)\|\langle\nabla_{x_{1}}\rangle^{-\gamma}\langle\nabla_{V}\rangle^{\frac32}((f^{\otimes 2})^\delta h)\|_{L^2(0,t;L^2(\Dd^2))}\,\sup_{[0,t]}\|(f^{\otimes2})^\delta S\|_{L^2}.
\end{eqnarray*}
Combining with~\eqref{eq:S-energy} yields the conclusion.
\end{proof}

We are now in a position to conclude the proof of Proposition~\ref{lem:unique}. We proceed by an explicit examination of the Duhamel expansion.

\begin{proof}[Proof of Proposition~\ref{lem:unique}]
Let $C_m:= C_m^1- C_m^2$. As the hierarchy is linear, $(C_m)_{m\ge0}$ is still a weak solution, now with final data $C_m|_{t=T}=0$, and it satisfies both~\eqref{uniformL2Cmi} and~\eqref{eq:limhier-ap2} (up to replacing $L$ by $2L$). The conclusion follows if we can prove $C_m\equiv 0$ on the whole interval $[0,T]$.
By the regularity estimate~\eqref{eq:limhier-ap2} and the assumptions on $K,f$, we find that the interaction term in the hierarchy is well-defined, with
\begin{eqnarray}
\|(f^{\otimes m})^{1-\delta}K_f[C_{m}]\|_{L^2}
&\lesssim&
\left\{\begin{array}{ll}
\||\nabla_{x_1}|^\gamma((f^{\otimes m})^{1-\delta}C_{m})\|_{L^2},&\text{if $\alpha>0$},\\
\||\nabla_{x_1}|^\gamma\langle\nabla_{v_{[2]}}\rangle^{-\frac32}((f^{\otimes m})^{1-\delta}C_{m})\|_{L^2},&\text{if $\alpha=0$},
\end{array}\right.\nonumber\\
&\le&(2L)^{\frac12(m+1)}\sqrt{T-t}.\label{eq:interaction-KfC}
\end{eqnarray}
In terms of the propagator $(B_{t,s})_{0\le s\le t\le T}$ constructed in Lemma~\ref{lem:propag-B}, we may then appeal to Duhamel's formula for the weak solution of~\eqref{eq:limhier-re-1} with final data $C_m^T=0$,
\[C_m(t)=-\sqrt{(m+1)(m+2)}\int_t^T(B_{s,t}^{\otimes m})^*K_{f_s}[C_{m+2}^s]\,ds.\]
Here, $B_{t,s}^*:L^2(f_t)\to L^2(f_s)$ denotes the adjoint of $B_{t,s}:L^2(f_s)\to L^2(f_t)$ with respect to the $f$-weighted pairing.
Iterating this identity, we find for all $m\ge0$ and $\ell\ge1$,
\begin{multline*}
C_m(t)=(-1)^\ell\sqrt{(m+1)\ldots(m+2\ell)}\\
\times\int_t^T(B_{s,t}^{\otimes m})^*\bigg\langle \Big(\int_{t<s_1<\ldots<s_{\ell-1}<s}{\textstyle\bigotimes_{j=1}^{\ell-1}}(B_{s,s_j}^{\otimes2}K_{f_{s_j}})\Big),K_{f_{s}}[C_{m+2\ell}^{s}]\bigg\rangle_{L^2(f_{s}^{\otimes2(\ell-1)})}\,ds,
\end{multline*}
that is, by symmetry,
\begin{equation*}
C_m(t)=(-1)^\ell\frac{\sqrt{(m+2\ell)!}}{(\ell-1)!\sqrt{m!}}\int_t^T(B_{s,t}^{\otimes m})^*\bigg\langle \Big(\int_t^sB_{s,r}^{\otimes2}K_{f_{r}}dr\Big)^{\otimes\ell-1}\!\!,K_{f_{s}}[C_{m+2\ell}^{s}]\bigg\rangle_{L^2(f_{s}^{\otimes2(\ell-1)})}\,ds.
\end{equation*}
Taking the norm, decomposing the weight as $f=f^\delta f^{1-\delta}$, and noting that, by duality,~\eqref{eq:bnd-Bts} yields
\[\|B_{s,t}^*\|_{\Lc(L^2(f_s^{2(1-\delta)}),L^2(f_t^{2(1-\delta)}))}\le e^{C(s-t)},\qquad 0\le t\le s\le T,\]
we obtain
\begin{multline*}
\|C_m(t)\|_{L^2((f_t^{\otimes m})^{2(1-\delta)})}\le\frac{\sqrt{(m+2\ell)!}}{(\ell-1)!\sqrt{m!}}\int_t^Te^{Cm(s-t)}\Big\|\int_t^sB_{s,r}^{\otimes2}K_{f_{r}}dr\Big\|_{L^2((f_s^{\otimes2})^{2\delta})}^{\ell-1}\\
\times\|(f_s^{\otimes m+2(\ell-1)})^{1-\delta}K_{f_{s}}[C_{m+2\ell}^{s}]\|_{L^2(f_{s}^{\otimes m+2(\ell-1)})}\,ds.
\end{multline*}
Hence, by~\eqref{eq:bnd-BK-delta} and~\eqref{eq:interaction-KfC},
\begin{equation}\label{eq:bnd-Cm-uniqueness}
\|C_m(t)\|_{L^2((f_t^{\otimes m})^{2(1-\delta)})}\le C_{T,L}^m\frac{\sqrt{(m+2\ell)!}}{(\ell-1)!\sqrt{m!}}\big(CLe^{C(T-t)}\sqrt{T-t}\big)^{\ell-1}.
\end{equation}
Using the elementary bound
\[\frac{\sqrt{(m+2\ell)!}}{(\ell-1)!\sqrt{m!}}\le C^\ell\Big(1+\frac{m^\ell}{\ell!}\Big),\]
we find that there is some $\tau_*>0$ such that for any fixed $m\ge0$ the right-hand side in~\eqref{eq:bnd-Cm-uniqueness} tends to $0$ as $\ell\uparrow\infty$ provided $T-t\le\tau_*$. Hence $C_m\equiv0$ on $[T-\tau_*,T]$ for all $m\ge0$. By iteration, the conclusion follows on the whole interval $[0,T]$.
\end{proof}

\subsection{Proof of Theorem~\ref{th:main}}
To derive the mean-field limit, it is enough to consider observables with almost deterministic final data~$\Phi_N^T$ in the sense of~\eqref{eq:conv-PhiNT-ass-MF},
\[\lim_{N\uparrow\infty}\|\Phi_N^T-1\|_{L^2(f_T^{\otimes N})}=0.\]
For dual correlations, recalling~\eqref{eq:orthog} and~\eqref{eq:barCNn}, this entails
\[\bar C_{N,m}^T\to \bar C_m^T:=\mathds1_{m=0}\quad\text{strongly in $L^2(f^{\otimes m})$.}\]

By Lemma~\ref{lem:limhier}, up to a subsequence, dual correlations then converge weakly in $L^2$, uniformly in time, to a solution of the limiting dual hierarchy~\eqref{eq:limhier1} with $\bar C_m^T=\mathds1_{m=0}$, satisfying~\eqref{eq:limhier3}--\eqref{eq:limhier4}. This problem has a unique solution by Proposition~\ref{lem:unique}. Since a trivial solution is given by $\bar C_m\equiv\mathds1_{m=0}$ on~$[0,T]$, this is necessarily the unique solution. 

We may thus get rid of the extraction of a subsequence and conclude $\bar C_{N,m}(t)\cvf\mathds1_{m=0}$ in $L^2(f_t^{\otimes m})$, uniformly for $t\in[0,T]$. For $m=0$, this proves~\eqref{eq:conv-CN2-0} and the conclusion then follows from Proposition~\ref{lem:duality}.\qed

\section{Convergence of bounded observables}\label{sec:general-test}

The previous section identifies the limiting dual hierarchy and establishes the mean-field limit by focusing on observables with almost deterministic final data~\eqref{eq:conv-PhiNT-ass-MF}. More generally, we now extend the analysis to characterize the limiting expectations of more general sequences of uniformly bounded observables $(\Phi_N^T)_N$, provided their rescaled Hoeffding components converge weakly with respect to the reference product measure $f^{\otimes N}$, up to some additional technical condition~\eqref{eq:CmT-weight}. This result goes beyond the mean-field limit and is the main ingredient in our characterization of empirical fluctuations.

While the previous section identifies the limit through the limiting dual hierarchy~\eqref{eq:limhier1}, we show that it can be re-expressed in terms of the \emph{primal} (or bidual) limiting hierarchy~\eqref{eq:limhier-bidual}, which is the formal limit of the rescaled BBGKY hierarchy for the primal linear correlations~\eqref{eq:def-HNm}.
Since no suitable a priori bounds are available on these correlations to pass directly to the limit on the primal side, we need to first take the limit on the dual side, as in the previous section, and then recover the primal hierarchy through a limiting duality identity. Justifying this duality actually requires some care, hence the exponential summability condition~\eqref{eq:CmT-weight}; see Section~\ref{sec:well-posedness} below. A cleaner, non-hierarchical version of this result in the $L^2$ framework of~\cite{BDJ-24} is postponed to a separate note.

\begin{theor}\label{th:bounded-obs}
Assume the hypotheses of Theorem~\ref{th:main} and the strong initial convergence~\eqref{eq:init-conv} with limit $(\bar H_m^\circ)_{m\ge0}$.
There exists~$\lambda_0\ge1$ such that the following holds. Let $T\in[0,T_0]$ and let $(\Phi_N^T)_N$ be a sequence of symmetric observables $\Phi_N^T:\Dd^N\to\C$ satisfying:
\begin{enumerate}[(i)]
\item $\sup_N\|\Phi_N^T\|_{L^\infty}<\infty$;
\smallskip\item for all $m\ge0$, their rescaled Hoeffding components converge weakly,
\[\bar C_{N,m}^T:=\binom{N}{m}^\frac12\Pi_{N,m}^T[\Phi_N^T]\cvf\bar C_m^T\qquad\text{in $L^2(f_T^{\otimes m})$};\]
\item the limiting coefficients satisfy
\begin{equation}\label{eq:CmT-weight}
\sum_{m\ge0} e^{\lambda_0(T+\sqrt T)m}\|\bar C_m^T\|_{L^2(f_T^{\otimes m})}^2<\infty.
\end{equation}
\end{enumerate}
Then we have
\begin{equation}\label{eq:conv-obs}
\int_{\Dd^N}\Phi_N^TF_N(T)\to\sum_{m\ge0}\int_{\Dd^m}\bar C_m^T\bar H_m(T) f_T^{\otimes m},
\end{equation}
where the series is absolutely convergent and $(\bar H_m)_m$ is the unique weak solution of the limiting primal hierarchy
\begin{equation}\label{eq:limhier-bidual}
\left\{\begin{array}{l}
(\partial_t +\hat L_f^m)\bar H_m=-\sqrt{m(m-1)}\,\Sym(K_f\otimes\bar H_{m-2}),\qquad m\ge0,\\[1mm]
\bar H_m|_{t=0}=\bar H_m^\circ,
\end{array}\right.
\end{equation}
with the convention $\bar H_{-1}=\bar H_{-2}=0$, where we recall that $\hat L_f^m$ is defined in~\eqref{eq:def-hat-Lf}, and~$\Sym$ denotes symmetrization in all particle variables.
\end{theor}

In terms of primal and dual correlations, by orthogonality~\eqref{eq:orthog}, the Hoeffding decomposition allows us to rewrite the convergence~\eqref{eq:conv-obs} as
\begin{equation}\label{eq:conv-obs-re}
\sum_{m=0}^N\int_{\Dd^m}\bar C_{N,m}^T\bar H_{N,m}(T)f_T^{\otimes m}\to \sum_{m\ge0}\int_{\Dd^m}\bar C_m^T\bar H_m(T) f_T^{\otimes m},
\end{equation}
where we recall the definition~\eqref{eq:barCNn} of rescaled dual correlations and similarly, on the primal side, cf.~\eqref{eq:def-HNm},
\begin{equation}\label{eq:def-barHNm}
\bar H_{N,m}:=\binom{N}{m}^{\frac12}H_{N,m}=\binom{N}{m}^{\frac12}\Pi_{N,m}\Big[\frac{F_N}{f^{\otimes N}}\Big],
\end{equation}
Yet, we emphasize that \eqref{eq:conv-obs-re} does not imply weak convergence of $\bar H_{N,m}(T)$ in $L^2(f_T^{\otimes m})$, since the test coefficients $(\bar C_{N,m}^T)_m$ are restricted to arise from uniformly bounded observables satisfying the conditions~(i)--(iii). Although such tests form a dense class in the Fock space, no uniform $L^2$ bound is available for $\bar H_{N,m}(T)$ in our setting, even for fixed $m\ge1$.

\subsection{Well-posedness of the limiting hierarchies}\label{sec:well-posedness}
In order to express the limit in~\eqref{eq:conv-obs} through the limiting primal hierarchy, instead of the dual hierarchy derived in the previous section, we need to establish a duality identity at the level of limiting hierarchies.
Whereas the existence of dual solutions follows from the compactness argument in the previous section, a corresponding argument is not possible on the primal side since we do not have any uniform a priori estimate on primal correlations~$\bar H_{N,m}$. Instead, as the limiting primal hierarchy is lower triangular, we notice that unique solutions are constructed explicitly by iterating the Duhamel formula. In order to derive the duality identity, we then proceed by further deriving a corresponding explicit representation of the dual evolution.

Some care is needed with the functional setting: the $L^\infty$-$L^1$ duality between observables and $N$-particle distributions has no convenient characterization in terms of Hoeffding components in the Fock space $\bigoplus_{m\ge0}L^2(f^{\otimes m})$.
We work instead with exponentially weighted Fock spaces, which provide a tractable substitute, at the price of a loss in exponential weight, hence the exponential summability condition~\eqref{eq:ass-Cm-weight-T} for final data.

\begin{lem}\label{lem:well-posed-dual}
Assume the hypotheses of Proposition~\ref{lem:unique}.
There exists $\lambda_0\ge1$ such that the following holds for all $T\in(0,T_0]$. Given $C_m^T\in L^2_0(f_T^{\otimes m})$ for all $m\ge0$, with
\begin{equation}\label{eq:ass-Cm-weight-T}
\sum_{m\ge0}e^{\lambda_0(T+\sqrt T)m}\|C_m^T\|_{L^2(f_T^{\otimes m})}^2<\infty,
\end{equation}
the limiting dual hierarchy~\eqref{eq:limhier1} admits a unique weak solution $(C_m)_m$ on~$[0,T]$, with final data $C_m|_{t=T}=C_m^T$, such that
\begin{equation}\label{eq:limhier-ap1}
\sup_{[0,T]}\sum_{m\ge0}\|C_m\|_{L^2(f^{\otimes m})}^2<\infty,
\end{equation}
and, for some $L\ge1$,
\begin{equation}\label{eq:reg-sol-Cm}
\begin{array}{rlll}
\int_0^T\||\nabla_{x_1}|^\gamma((f^{\otimes m})^{1-\delta}C_m)\|_{L^2}^2&\le& L^{m+1},&\quad\text{if $\alpha>0$},\\[1mm]
\int_0^T\||\nabla_{x_1}|^\gamma\langle\nabla_{v_{[2]}}\rangle^{-\frac32}((f^{\otimes m})^{1-\delta}C_m)\|_{L^2}^2&\le&L^{m+1},&\quad\text{if $\alpha=0$}.
\end{array}
\end{equation}
This solution is explicitly given by the normally convergent series
\begin{equation}\label{eq:Cm-expl}
C_m(t)=\sum_{\ell\ge0}(-1)^\ell\frac{\sqrt{(m+2\ell)!}}{\ell!\sqrt{m!}}(B_{T,t}^*)^{\otimes m}\bigg\langle\Big(\int_t^TB_{T,s}^{\otimes2}K_{f_s}ds\Big)^{\otimes\ell},C_{m+2\ell}^T\bigg\rangle_{L^2(f^{\otimes 2\ell}_T)}.
\end{equation}
Similarly, given $H_m^\circ\in L^2_0(f_\circ^{\otimes m})$ for all $m\ge0$, with
\begin{equation}
\sum_{m\geq 0} \|H_m^\circ\|_{L^2(f_\circ^{\otimes m})}^2<\infty,
\end{equation}
the primal hierarchy~\eqref{eq:limhier-bidual} also admits a unique weak solution, which is explicitly given by
\begin{equation}\label{eq:explicit-Hm}
\bar H_m(t)=\sum_{\ell=0}^{\lfloor m/2\rfloor}(-1)^\ell\frac{\sqrt{m!}}{\ell!\sqrt{(m-2\ell)!}}\,\Sym\bigg((B_{t,0}^{\otimes m-2\ell}\bar H_{m-2\ell}^\circ)\otimes \Big(\int_0^tB_{t,s}^{\otimes2}K_{f_s}ds\Big)^{\otimes\ell}\bigg),
\end{equation}
and satisfies the weighted Fock-space estimate
\begin{equation}\label{eq:bnd-Hm}
\sup_{0\le t\le T}\sum_{m\ge0}e^{-\lambda_0(t+\sqrt t)m}\|\bar H_m(t)\|_{L^2(f_t^{\otimes m})}^2\le C_{T}\sum_{m\ge0}\|\bar H_{m}^\circ\|_{L^2(f_\circ^{\otimes m})}^2.
\end{equation}
Moreover, we then have the duality identity
\begin{equation}\label{eq:bidual-pairing}
\sum_{m\ge0}\int_{\Dd^m}\bar C_m(0)\bar H_m^\circ f_\circ^{\otimes m}
=\sum_{m\ge0}\int_{\Dd^m}\bar C_m^T\bar H_m(T)f_T^{\otimes m},
\end{equation}
where both sides converge absolutely.
\end{lem}

\begin{proof}
Since~\eqref{eq:limhier-bidual} is lower triangular,  it can be solved at each order by finitely many iterations of the Duhamel formula, with the propagator $B_{t,s}$ constructed in Lemma~\ref{lem:propag-B}, so that any weak solution is explicitly given by~\eqref{eq:explicit-Hm} and is necessarily unique. By Lemmas~\ref{lem:propag-B} and~\ref{lem:bnd-BK-delta}, we may then bound
\[\|\bar H_m(t)\|_{L^2(f_t^{\otimes m})}\leq e^{Cmt}\sum_{\ell=0}^{\lfloor m/2\rfloor}\frac{\sqrt{m!}}{\ell!\sqrt{(m-2\ell)!}}\,L_t^\ell\|\bar H_{m-2\ell}^\circ\|_{L^2(f_\circ^{\otimes m-2\ell})},\]
with the shorthand notation~$L_{t}=Ce^{Ct}\sqrt{t}$. Using
\[\frac{\sqrt{m!}}{\ell!\sqrt{(m-2\ell)!}}\le2^\ell\binom{m}{2\ell}^{\frac12},\]
we deduce by the Cauchy-Schwarz inequality,
\[\|\bar H_m(t)\|_{L^2(f_t^{\otimes m})}^2\le e^{Cmt}(4L_t+1)^{m}\sum_{\ell=0}^{\lfloor m/2\rfloor}2^{-2\ell}\|\bar H_{m-2\ell}^\circ\|_{L^2(f_\circ^{\otimes m-2\ell})}^2.\]
Bounding $4L_t+1\le e^{C(t+\sqrt t)}$, summing up in $m$ against $e^{-\lambda_0(t+\sqrt t)m}$, and choosing $\lambda_0$ large enough, we immediately obtain~\eqref{eq:bnd-Hm}.

We turn to the dual hierarchy. We start by checking that the series~\eqref{eq:Cm-expl} is normally convergent. By Lemmas~\ref{lem:propag-B} and~\ref{lem:bnd-BK-delta}, we have
\begin{equation*}
\|C_m(t)\|_{L^2(f_t^{\otimes m})}
\le e^{CTm}\sum_{\ell\ge0}\frac{\sqrt{(m+2\ell)!}}{\ell!\sqrt{m!}}L_{T}^\ell\|C_{m+2\ell}^T\|_{L^2(f_T^{\otimes m+2\ell})}.
\end{equation*}
Summing over $m$, similarly as above, we then get for any $\lambda\ge0$,
\begin{equation*}
\sum_{m\ge0}e^{\lambda m}\|C_m(t)\|_{L^2(f_t^{\otimes m})}^2
\le C_T\sum_{m\ge0}e^{(\lambda+C(T+\sqrt T))m}\|C_{m}^T\|_{L^2(f_T^{\otimes m})}^2.
\end{equation*}
Choosing $\lambda=0$, recalling the final assumption~\eqref{eq:ass-Cm-weight-T}, and increasing $\lambda_0$ if necessary, we deduce~\eqref{eq:limhier-ap1} and the normal convergence of the series~\eqref{eq:Cm-expl}.

Next, we check that the series defines a weak solution of the limiting dual hierarchy~\eqref{eq:limhier1}. For this, we first establish the regularity estimates~\eqref{eq:reg-sol-Cm}, which allow us to interpret the singular interaction terms in the hierarchy. We argue by duality, using the same kinetic regularization estimate as in the proof of Proposition~\ref{lem:unique}.
For a smooth compactly supported time-dependent test function $H_m$, formula~\eqref{eq:Cm-expl} yields
\begin{equation}\label{eq:dual-HC-est-reg}
\int_0^T\int_{\Dd^m} H_mC_m f^{\otimes m}=\sum_{\ell\ge0}(-1)^\ell\frac{\sqrt{(m+2\ell)!}}{\ell!\sqrt{m!}}\int_{\Dd^{m+2\ell}} D_{m,\ell}^TC_{m+2\ell}^T f_T^{\otimes m+2\ell},
\end{equation}
in terms of
\[D_{m,\ell}^T:=\int_0^T (B_{T,s}^{\otimes m}H_m^s)\otimes\Big(\int_s^TB_{T,r}^{\otimes2}K_{f_r}dr\Big)^{\otimes\ell}ds.\]
To estimate this quantity, we view all variables except $z_1,z_2$ as a Hilbert-valued parameter.
Defining
\[R_{m,\ell}^{T}(t):=\int_0^t ((B_{t,s}^{\otimes 2}\otimes B_{T,s}^{\otimes m-2})H_m^s)\otimes\Big(\int_s^TB_{T,r}^{\otimes2}K_{f_r}dr\Big)^{\otimes\ell}ds,\qquad 0\le t\le T,\]
we can write the above as $D_{m,\ell}^T=R_{m,\ell}^T(T)$.
This quantity satisfies, in the weak sense, the following $L^2(f_T^{\otimes m+2\ell-2})$-valued equation,
\[(\partial_t+\hat L_f^2)R_{m,\ell}^T=(\Id^{\otimes 2}\otimes B_{T,t}^{\otimes m-2})H_m^t\otimes\Big(\int_t^TB_{T,r}^{\otimes2}K_{f_r}dr\Big)^{\otimes\ell},\qquad R_{m,\ell}^T|_{t=0}=0.\]
Let $(e_k)_k$ be an orthonormal basis of $L^2(f_T^{\otimes m+2\ell-2})$. Taking the scalar product of the equation with~$e_k$, we find
\[(\partial_t+\hat L_f^2)\langle e_k,R_{m,\ell}^T\rangle=\left<e_k,\;(\Id^{\otimes 2}\otimes B_{T,t}^{\otimes m-2})H_m^t\otimes\Big(\int_t^TB_{T,r}^{\otimes2}K_{f_r}dr\Big)^{\otimes\ell}\right>,\]
where $\langle\cdot,\cdot\rangle$ denotes the scalar product on $L^2(f_T^{\otimes m+2\ell-2})$.
Now applying Lemma~\ref{lem:bnd-BK-delta} and summing over $k$, we get
\begin{multline*}
\|D_{m,\ell}^T\|_{L^2((f_T^{\otimes m+2\ell})^{2\delta})}\lesssim_T\bigg(\sup_{0\le t\le T}\Big\|\int_t^TB_{T,r}^{\otimes2}K_{f_r}dr\Big\|_{L^2((f_T^{\otimes2})^{2\delta})}\bigg)^\ell\\
\times \left\{\begin{array}{lll}
\|\langle\nabla_{x_{1}}\rangle^{-\gamma}((f^{\otimes m})^\delta H_m)\|_{L^2(0,T;L^2(\Dd^m))}&:&\alpha>0,\\
\|\langle\nabla_{x_{1}}\rangle^{-\gamma}\langle\nabla_{v_{[2]}}\rangle^{\frac32}((f^{\otimes m})^\delta H_m)\|_{L^2(0,T;L^2(\Dd^m))}&:&\alpha=0.
\end{array}\right.
\end{multline*}
By~\eqref{eq:bnd-BK-delta} in Lemma~\ref{lem:bnd-BK-delta}, the first factor is bounded by $L_T^\ell$. From~\eqref{eq:dual-HC-est-reg}, recalling the boundedness of $f$ and $\delta\le\frac12$, we then get
\begin{multline*}
\Big|\int_0^T\int_{\Dd^m} H_mC_mf^{\otimes m}\Big|\le C_T^{m}
\sum_{\ell\ge0}\frac{\sqrt{(m+2\ell)!}}{\ell!\sqrt{m!}}\,L_T^\ell\|C_{m+2\ell}^T\|_{L^2(f^{\otimes m+2\ell}_T)}\\
\times \left\{\begin{array}{lll}
\|\langle\nabla_{x_{1}}\rangle^{-\gamma}((f^{\otimes m})^\delta H_m)\|_{L^2(0,T;L^2(\Dd^m))}&:&\alpha>0,\\
\|\langle\nabla_{x_{1}}\rangle^{-\gamma}\langle\nabla_{v_{[2]}}\rangle^{\frac32}((f^{\otimes m})^\delta H_m)\|_{L^2(0,T;L^2(\Dd^m))}&:&\alpha=0.
\end{array}\right.
\end{multline*}
Using again $CL_T\le e^{C(T+\sqrt T)}$ and
\[\frac{{(m+2\ell)!}}{\ell!^2{m!}}=\binom{m+2\ell}{m,\ell,\ell}\le 3^{m+2\ell},\]
the series is bounded by
\begin{equation*}
\sum_{\ell\ge0}\frac{\sqrt{(m+2\ell)!}}{\ell!\sqrt{m!}}\,L_T^\ell\|C_{m+2\ell}^T\|_{L^2(f_T^{\otimes(m+2\ell)})}
\le C^m\bigg(\sum_{m\ge0}e^{\lambda_0(T+\sqrt T)m}\|C_m^T\|_{L^2(f_T^{\otimes m})}^2\bigg)^{\frac12},
\end{equation*}
provided $\lambda_0$ is chosen large enough. Combined with the summability assumption~\eqref{eq:ass-Cm-weight-T}, this gives the regularity estimates~\eqref{eq:reg-sol-Cm} by duality.

The above bounds allow to conclude that the series~\eqref{eq:Cm-expl} is a weak solution in the class of Proposition~\ref{lem:unique}, and the latter gives uniqueness.
Finally, inserting the explicit forms~\eqref{eq:Cm-expl} and~\eqref{eq:explicit-Hm} of the solutions $C_m$ and $H_m$, and rearranging the absolutely convergent series yields the duality identity~\eqref{eq:bidual-pairing}.
\end{proof}

\subsection{Proof of Theorem~\ref{th:bounded-obs}}
Let $\Phi_N$ be a bounded weak solution of the backward dual Liouville equation~\eqref{eq:Liouville-dual} on $[0,T]$ with final data $\Phi_N^T$ and with the conservation property $\sup_{[0,T]}\|\Phi_N\|_{L^\infty}\le\|\Phi_N^T\|_{L^\infty}$.
Since $F_N$ is a weak duality solution of the Liouville equation in the sense of~\cite[Appendix]{BDJ-24}, we can choose $\Phi_N$ to be in duality with $F_N$, that is,
\begin{equation*}
\int_{\Dd^N}\Phi_N^TF_N(T)=\int_{\Dd^N}\Phi_N(0)F_N^\circ.
\end{equation*}
Decomposing both $\Phi_N(0)$ and $F_N^\circ$ into their Hoeffding components, cf.~\eqref{eq:barCNn} and~\eqref{eq:def-barHNm}, this duality formula becomes, by orthogonality~\eqref{eq:orthog},
\begin{equation}\label{eq:dual-PhiF-re}
\int_{\Dd^N}\Phi_N^TF_N(T)=\sum_{m=0}^N\int_{\Dd^m}\bar C_{N,m}(0)\bar H_{N,m}^\circ f_\circ^{\otimes m}.
\end{equation}
By the uniform boundedness of $\Phi_N^T$ (assumption~(i)) and by the assumed conservation property for $\Phi_N$, the orthogonality~\eqref{eq:orthog} yields
\begin{equation}\label{eq:fock-bnd}
\sup_{[0,T]}\sum_{m=0}^N\|\bar C_{N,m}\|_{L^2(f^{\otimes m})}^2\le\|\Phi_N^T\|_{L^\infty}^2\lesssim1.
\end{equation}
By Lemma~\ref{lem:limhier}, with the assumed weak convergence of $(\bar C_{N,m}^T)_m$ (assumption~(ii)), we obtain up to a subsequence, for all $m\ge0$,
\begin{equation}\label{eq:Cw-conv-boundedobs}
\bar C_{N,m}(t)\cvf\bar C_m(t)\quad\text{weakly in $L^2(f_t^{\otimes m})$, uniformly for $t\in[0,T]$,}
\end{equation}
where the limit $(\bar C_m)_m$ solves the limiting hierarchy~\eqref{eq:limhier1} with final data $(\bar C_m^T)_m$, in the regularity class~\eqref{eq:limhier3}--\eqref{eq:limhier4}. By Proposition~\ref{lem:unique}, this limiting solution is unique, so no subsequence extraction is actually needed.

We may now pass to the limit in~\eqref{eq:dual-PhiF-re}. By the assumed strong initial convergence~\eqref{eq:init-conv} and by the uniform Fock bound~\eqref{eq:fock-bnd}, we deduce
\begin{equation}\label{eq:lim-intPhiF}
\int_{\Dd^N}\Phi_N^TF_N(T)\to\sum_{m\ge0}\int_{\Dd^m}\bar C_m(0)\bar H_m^\circ f_\circ^{\otimes m}.
\end{equation}
By the exponential summability condition~\eqref{eq:CmT-weight} (assumption~(iii)), we may then appeal to Lemma~\ref{lem:well-posed-dual} and the conclusion~\eqref{eq:conv-obs} follows from the duality identity~\eqref{eq:bidual-pairing}.\qed

\section{Probabilistic interpretation of the limiting hierarchy}\label{sec:proba-inter}

This section is devoted to the proof of Theorem~\ref{th:fluct}. Theorem~\ref{th:bounded-obs} already characterizes the limiting expectations of bounded observables in terms of the primal hierarchy~\eqref{eq:limhier-bidual}, and it now remains to identify the latter with the fluctuating Vlasov equation~\eqref{eq:fluc-Vlas}. To this aim, we develop a probabilistic interpretation of the hierarchy through its Wiener representation, viewed as the natural limit framework associated with the finite-$N$ Hoeffding decomposition.

\subsection{Gaussian limit framework}\label{sec:gaussian-lim}
We first recall the Gaussian framework corresponding to the Hoeffding chaos decomposition of Section~\ref{sec:Hoeffding}; see e.g.~\cite[Section~2.7]{NourdinPeccati2012} or~\cite[Section~1.1]{Nualart2006}.
For all~$t\in[0,T_0]$, let $W_t$ be an isonormal Gaussian field over the real Hilbert space $L^2_0(f_t)$, defined on a probability space~$(\Omega,\PP)$. This can be viewed as a random distribution on $\Dd$ such that, for any finite family $(h_i)_i\subset C^\infty_c(\Dd)$, the random variables $W_t(h_i)=\langle W_t,h_i\rangle$ are jointly Gaussian and centered, with covariance structure
\[\E\big[W_t(h_i)\,W_t(h_j)\big]=\int_\Dd\Big(h_i-\int_\Dd h_if_t\Big)\Big(h_j-\int_\Dd h_jf_t\Big)f_t.\]
Note that $W_t(h)$ extends by density in $L^2(\Omega)$ to any $h\in L^2(f_t)$.
This naturally plays the role of the reference fluctuation field since by the central limit theorem the empirical fluctuation field~$\zeta_N(t)$ indeed converges in law in $\Dc'(\Dd)$ to $W_t$ under $f_t^{\otimes N}$.

Let then $I_m^t$ denote the corresponding multiple Wiener-It\^o integral. More precisely, we let $I_0^t(c)=c$, and for $m\ge1$, for real-valued $h\in L^2_0(f_t)$ with $\|h\|_{L^2(f_t)}=1$,
\[I_m^t(h^{\otimes m})=\mathcal H_m(W_t(h)),\]
where $\mathcal H_m$ is the $m$th Hermite polynomial. By linearity, polarization, and density, this extends uniquely to define $I_m^t(H_m)\in L^2(\sigma(W_t))$ for all $H_m\in L^2_0(f_t^{\otimes m})$.

The first key property is that multiple integrals of different orders are orthogonal, with the following isometry (see e.g.~\cite[Proposition~2.7.5]{NourdinPeccati2012}),
\begin{equation}\label{eq:Wiener-Ito-isometry}
\E\big[I_m^t(G_m)I_n^t(H_n)\big]=m!\,\mathds1_{m=n}\int_{\Dd^m} G_mH_mf_t^{\otimes m}.
\end{equation}
The second key property is the famous Wiener-It\^o chaos decomposition (see e.g.~\cite[Corollary~2.7.8]{NourdinPeccati2012}), which states that any $\Psi\in L^2(\sigma(W_t))$ admits a unique expansion
\[\Psi=\Ic_t(G):=\sum_{m\ge0}\frac1{\sqrt{m!}}I_m^t(G_m),\]
for some kernels $G_m\in L^2_0(f_t^{\otimes m})$. From the isometry property, we further have
\[\E[|\Ic_t(G)|^2]=\sum_{m\ge0}\|G_m\|_{L^2(f_t^{\otimes m})}^2,\]
so that $\Ic_t$ defines an isometry from the symmetric Fock space
\[\Gamma(L^2_0(f_t)):=\bigoplus_{m\ge0}L^2_0(f_t^{\otimes m}),\qquad
\|G\|_{\Gamma(L^2_0(f_t))}^2:=\sum_{m\ge0}\|G_m\|_{L^2(f_t^{\otimes m})}^2,\]
onto $L^2(\sigma(W_t))$.

\subsection{Correspondence between Hoeffding and Wiener-It\^o}\label{sec:Hoeffd-Ito}
We now make precise the correspondence between the finite-$N$ Hoeffding decomposition and the Wiener-It\^o chaos expansion.
For fixed $t\in[0,T_0]$, recall that the Hoeffding decomposition of Section~\ref{sec:Hoeffding} defines the isometry
\[\Ic_N^t(G):=\sum_{m=0}^N\binom{N}{m}^{-\frac12}I_{N,m}(G_m),\]
from the finite-$N$ Fock subspace
\[\Gamma_{N}(L^2_0(f_t)):=\bigoplus_{m=0}^N L^2_0(f_t^{\otimes m})~\subset\,\Gamma(L^2_0(f_t))\]
onto the space of symmetric functions in $L^2(f_t^{\otimes N})$, cf.~\eqref{eq:orthog}.
This is the finite-$N$ counterpart of the Wiener-It\^o isometry $\Ic_t$ introduced in the previous subsection, which identifies $\Gamma(L^2_0(f_t))$ with $L^2(\sigma(W_t))$.

The correspondence between the two representations is a consequence of the classical limit theorem for degenerate $U$-statistics~\cite{DynkinMandelbaum-83}, which identifies the asymptotic joint distribution of Hoeffding chaoses with the corresponding Wiener-It\^o chaoses, jointly with the convergence of empirical fluctuations: denoting by $P_N:\Gamma(L^2_0(f_t))\to\Gamma_{N}(L^2_0(f_t))$ the orthogonal projection, we have for all $G\in\Gamma(L^2_0(f_t))$,
\[(\Ic_N^t(P_NG),\zeta_N(t))\Rightarrow(\Ic_t(G),W_t)\quad\text{in $\R\times\Dc'(\Dd)$, under $f_t^{\otimes N}$}.\]
Combined with the Hoeffding and Wiener-It\^o isometries, this yields the following characterization of strong convergence in Fock space.

\begin{lem}\label{lem:equiv-CPhi}
Given $t\in[0,T_0]$, let $(\Psi_N^t)_N$ be a sequence of symmetric functions $\Psi_N^t:\Dd^N\to\R$ with $\sup_N\int_{\Dd^N}|\Psi_N^t|^2f_t^{\otimes N}<\infty$, and denote their Hoeffding components by $G_{N,m}:=\Pi_{N,m}^t[\Psi_N^t]$, with the convention $G_{N,m}=0$ for $m>N$.
Then the following are equivalent:
\begin{enumerate}[(i)]
\item for some $\bar G^t=(\bar G_m^t)_m\in\Gamma(L^2_0(f_t))$, we have
\[\sum_{m\ge0}\bigg\|\binom{N}{m}^\frac12G_{N,m}-\bar G_m^t\bigg\|_{L^2(f_t^{\otimes m})}^2\to0;\]
\item for some $\Psi_t\in L^2(\sigma(W_t))$, we have the joint convergence in law
\[(\Psi_N^t,\zeta_N(t))\Rightarrow(\Psi_t,W_t)\quad\text{in $\R\times\Dc'(\Dd)$, under $f_t^{\otimes N}$,}\]
without loss of $L^2$ mass, that is, $\int_{\Dd^N}|\Psi_N|^2f_t^{\otimes N}\to\E[|\Psi_t|^2]$.
\end{enumerate}
In this case, we have $\Psi_t=\Ic_t(\bar G^t)$.
\end{lem}

\begin{proof}
By the classical limit theorem for degenerate $U$-statistics~\cite{DynkinMandelbaum-83}, finite Hoeffding chaoses converge jointly with the empirical fluctuation field to their Wiener-It\^o counterparts. Together with the Hoeffding and Wiener-It\^o isometries, the convergence in Fock space ensures the convergence of $L^2$ norms, so that (i) implies (ii). Conversely, joint convergence in law identifies the weak limits of the Hoeffding coefficients by testing against finite chaoses. The convergence of $L^2$ norms then yields strong convergence in Fock space by the Radon-Riesz property, so that (ii) implies (i).
\end{proof}

\begin{rem}[Initial assumptions]\label{rem:fluct-init-explan}
As a consequence, we recover the equivalence stated in Remark~\ref{rem:fluct}(c) regarding the initial assumptions in Theorem~\ref{th:fluct}. Indeed, by the above lemma, the assumption on the strong convergence~\eqref{eq:init-conv} of the Hoeffding coefficients is equivalent to the joint convergence in law $(F_N^\circ/f_\circ^{\otimes N},\zeta_N^\circ)\Rightarrow(\Psi_\circ,W_\circ)$ under $f_\circ^{\otimes N}$, together with $\int_{\Dd^N}|F_N^\circ|^2/f_\circ^{\otimes N}\to\E[|\Psi_\circ|^2]$, where the limit is given by
\begin{equation}\label{eq:RNder-zeta0}
\Psi_\circ=\sum_{m\ge0}\frac1{\sqrt{m!}}I_m^\circ(\bar H_m^\circ)~\in L^2(\sigma(W_\circ)).
\end{equation}
Since $F_N^\circ$ is a probability density, we deduce $\Psi_\circ\ge0$ and $\E[\Psi_\circ]=1$. Moreover, for any test function $h\in C^\infty_c(\Dd)$, the joint convergence implies
\begin{equation*}
\E_{F_N^\circ}\Big[e^{i\langle \zeta_N^\circ,h\rangle}\Big]=\E_{f_\circ^{\otimes N}}\bigg[e^{i\langle \zeta_N^\circ,h\rangle}\frac{F_N^\circ}{f_\circ^{\otimes N}}\bigg]\to\E\Big[e^{iW_\circ(h)}\Psi_\circ\Big].
\end{equation*}
By L\'evy's continuity theorem, this entails $\zeta_N^\circ\Rightarrow\zeta_\circ$ under $F_N^\circ$, where the law of $\zeta_\circ$ is absolutely continuous with respect to the law of $W_\circ$, with Radon-Nikodym derivative $\Psi_\circ$.
\end{rem}

\subsection{Fluctuating Vlasov equation and primal hierarchy}\label{sec:fluct-vlas}
With the above probabilistic setup at hand, we can now identify the limiting primal hierarchy~\eqref{eq:limhier-bidual} as the Wiener-It\^o chaos representation of the fluctuating Vlasov equation. First, we check that the notion of mild solution in Definition~\ref{defin:mild-sol} is well-defined.

\begin{lem}\label{lem:primal-lim}
Let $\zeta_\circ$ be a random distribution on $\Dd$ satisfying~\eqref{eq:zeta0-bndL2f}.
Then there is a unique mild solution~$\zeta$ of~\eqref{eq:fluc-Vlas} in the sense of Definition~\ref{defin:mild-sol} with initial data~$\zeta_\circ$.
\end{lem}

\begin{proof}
The operator $S_f$ in Definition~\ref{defin:mild-sol} coincides with $L_f$ up to a constant:
\[L_fh:=S_fh+c_f(h),\qquad c_f(h):=\int_\Dd (K\ast f)(x_*)\cdot\nabla_vf(z_*)\,h(z_*)\,dz_*.\]
Since $\langle\zeta_\circ,1\rangle=0$, and since the white-noise term only involves gradients, the propagator~$U_{s,t}$ in~\eqref{eq:mild-fluc-Vlas} can be replaced by~$B_{t,s}^*$. The definition of mild solution then reads, for $h\in C^\infty_c(\Dd)$,
\begin{equation}\label{eq:mild-fluc-Vlas-re}
\langle\zeta_t,h\rangle=\langle\zeta_\circ,B_{t,0}^*h\rangle+\sqrt{2\alpha}\int_0^t\int_{\Dd}\sqrt{f_s(z)}\,\nabla_vB_{t,s}^*h(z)\cdot\xi(ds,dz),
\end{equation}
where $\xi$ is a vector-valued space-time white noise independent of $\zeta_\circ$.
From the equations for $f$ and for $s\mapsto B_{t,s}^*h$, the energy identity gives
\begin{multline}
\lefteqn{\int_\Dd \Big(B_{t,0}^*h-\int_\Dd (B_{t,0}^*h)f_\circ\Big)^2f_\circ+2\alpha\int_0^t\int_\Dd|\nabla_vB_{t,s}^*h|^2f_s\,ds}\\
=\int_\Dd\Big(h-\int_\Dd hf_t\Big)^2f_t-2\int_0^t\int_{\Dd^2}K_{f_s}(B_{t,s}^*h)^{\otimes2}f_s^{\otimes2}\,ds,
\label{eq:energy-U-fluct}
\end{multline}
and Lemmas~\ref{lem:propag-B} and~\ref{lem:bnd-BK-delta} ensure that the right-hand side is bounded by $C_t\int_\Dd|h|^2f_t$.
Together with~\eqref{eq:zeta0-bndL2f}, this shows that both terms in~\eqref{eq:mild-fluc-Vlas-re} are well-defined and that
\[\E[|\langle\zeta_t,h\rangle|^2]\le C_t\|h\|_{L^2(f_t)}^2.\]
The representation~\eqref{eq:mild-fluc-Vlas-re} thus ensures existence and uniqueness of the mild solution.
\end{proof}

Next, the link to the limiting primal hierarchy is drawn by expanding the characteristic functional of the above-defined mild solution.

\begin{lem}\label{lem:limfluct-char}
Let $\zeta$ be the unique mild solution of~\eqref{eq:fluc-Vlas} as constructed above, and further assume that the law of $\zeta_\circ$ is absolutely continuous with respect to the law of $W_\circ$, with Radon-Nikodym density~\eqref{eq:RNder-zeta0}. Then, for all $h\in C_c^\infty(\Dd)$, its characteristic functional takes the explicit form
\begin{equation}\label{eq:char-zetat}
\E\Big[e^{i\langle \zeta(t),h\rangle}\Big]=e^{-\frac12\int_\Dd(h-\int_\Dd hf_t)^2f_t}\sum_{m\ge0}\frac{i^m}{\sqrt{m!}}\int_{\Dd^m}h^{\otimes m}\bar H_m(t) f_t^{\otimes m},
\end{equation}
where $(\bar H_m)_m$ is the weak solution of the primal hierarchy~\eqref{eq:limhier-bidual} with initial data~$(\bar H_m^\circ)_m$.
\end{lem}

\begin{proof}
By independence of $\zeta_\circ$ and $\xi$, the characteristic function of~\eqref{eq:mild-fluc-Vlas-re} reads
\begin{equation}\label{eq:zeta0-char}
\E\Big[e^{i\langle\zeta_t,h\rangle}\Big]=\E\Big[e^{i\langle\zeta_\circ,B_{t,0}^*h\rangle}\Big]\exp\Big(-\alpha\int_0^t\int_{\Dd}|\nabla_vB_{t,s}^*h|^2f_s\,ds\Big).
\end{equation}
Since by assumption the law of $\zeta_\circ$ is absolutely continuous with respect to the law of $W_\circ$, with Radon-Nikodym density $\Psi_\circ$, we have for all $h\in L^2(f_\circ)$,
\begin{equation*}
\E\Big[e^{i\langle\zeta_\circ,h\rangle}\Big]=\E\Big[e^{iW_\circ(h)}\Psi_\circ\Big].
\end{equation*}
By the classical Wiener-It\^o chaos expansion of the Gaussian exponential (see e.g.~\cite[Section~2.7]{NourdinPeccati2012}), we can write
\begin{equation}\label{eq:expon-chaos}
e^{iW_\circ(h)}=e^{-\frac12\int_\Dd (h-\int_\Dd hf_\circ)^2f_\circ}\sum_{m\ge0}\frac{i^m}{m!}\,I_m^\circ\Big(\Big(h-\int_\Dd hf_\circ\Big)^{\otimes m}\Big),
\end{equation}
hence, by the chaos expansion of~$\Psi_\circ$ in~\eqref{eq:RNder-zeta0} and by the Wiener-It\^o isometry~\eqref{eq:Wiener-Ito-isometry},
\begin{equation*}
\E\Big[e^{i\langle\zeta_\circ,h\rangle}\Big]=\E\Big[e^{iW_\circ(h)}\Psi_\circ\Big]=e^{-\frac12\int_\Dd (h-\int_\Dd hf_\circ)^2f_\circ}\sum_{m\ge0}\frac{i^m}{\sqrt{m!}}\int_{\Dd^m}h^{\otimes m}\bar H_m^\circ f_\circ^{\otimes m}.
\end{equation*}
Inserting this into~\eqref{eq:zeta0-char} and using the energy identity~\eqref{eq:energy-U-fluct}, we deduce
\begin{multline}\label{eq:char-zetat-0}
\E\Big[e^{i\langle\zeta(t),h\rangle}\Big]=\exp\bigg(-\frac12\int_\Dd\Big(h-\int_\Dd hf_t\Big)^2f_t+\int_0^t\int_{\Dd^2}K_{f_s}(B_{t,s}^*h)^{\otimes2}f_s^{\otimes2}\,ds\bigg)\\
\times\sum_{m\ge0}\frac{i^m}{\sqrt{m!}}\int_{\Dd^m}(B_{t,0}^*h)^{\otimes m}\bar H_m^\circ f_\circ^{\otimes m}.
\end{multline}
It remains to connect this expression with the primal hierarchy.
In terms of
\[J_m(s):=\int_{\Dd^m}(B_{t,s}^*h)^{\otimes m}\bar H_m(s)f_s^{\otimes m},\]
the equation for $s\mapsto B_{t,s}^*h$ and the primal hierarchy for $\bar H_m$ yield
\[\partial_sJ_m(s)=-\sqrt{m(m-1)}\,\Big(\int_{\Dd^2}K_{f_s}(B_{t,s}^*h)^{\otimes2}f_s^{\otimes2}\Big)J_{m-2}(s).\]
Therefore
\begin{equation}\label{eq:genfun-H}
\partial_s\sum_{m\ge0}\frac{i^m}{\sqrt{m!}}J_m(s)
=\Big(\int_{\Dd^2}K_{f_s}(B_{t,s}^*h)^{\otimes2}f_s^{\otimes2}\Big)
\sum_{m\ge0}\frac{i^m}{\sqrt{m!}}J_m(s),
\end{equation}
where the summation is justified by~\eqref{eq:bnd-Hm} and the propagator bounds.
Integrating in time, we deduce
\begin{multline*}
\sum_{m\ge0}\frac{i^m}{\sqrt{m!}}\int_{\Dd^m}h^{\otimes m}\bar H_m(t)f_t^{\otimes m}\\[-3mm]
=\exp\Big(\int_0^t\int_{\Dd^2}K_{f_s}(B_{t,s}^*h)^{\otimes2}f_s^{\otimes2}\,ds\Big)\sum_{m\ge0}\frac{i^m}{\sqrt{m!}}\int_{\Dd^m}(B_{t,0}^*h)^{\otimes m}\bar H_m^\circ f_\circ^{\otimes m}.
\end{multline*}
Compared with~\eqref{eq:char-zetat-0}, this proves~\eqref{eq:char-zetat}.
\end{proof}

\subsection{Empirical fluctuations at one fixed time}
We now turn to the proof of Theorem~\ref{th:fluct} for empirical fluctuations at one fixed time $T\in[0,T_0]$ under $F_N(T)$. Applying Theorem~\ref{th:bounded-obs} to their characteristic functional and then using Lemma~\ref{lem:limfluct-char} to identify the limit, we obtain the following.

\begin{lem}\label{cor:fluct-tfixed0}
Assume the hypotheses of Theorem~\ref{th:fluct}. Then, for all $T\in[0,T_0]$, we have under~$F_N(T)$,
\[\zeta_N(T)\Rightarrow\zeta(T)\quad\text{in $\Dc'(\Dd)$}.\]
\end{lem}

\begin{proof}
For $T\in[0,T_0]$ and $h\in C_c^\infty(\Dd)$, consider
\begin{equation}\label{eq:terminal-char-observable}
\Phi_N^T(z_{[N]}):=e^{i\langle\zeta_N(T,z_{[N]}),h\rangle}
=\prod_{j=1}^N e^{\frac{i}{\sqrt N}(h(z_j)-\int_\Dd hf_T)}.
\end{equation}
The product structure gives for the associated rescaled Hoeffding coefficients:
\[\bar C_{N,m}^{T}=\binom Nm^{1/2}a_N^{N-m}b_N^{\otimes m},\]
in terms of
\[a_N:=\int_\Dd e^{\frac{i}{\sqrt N}(h-\int_\Dd hf_T)}f_T,\qquad b_N(z):=e^{\frac{i}{\sqrt N}(h(z)-\int_\Dd hf_T)}-a_N.\]
Since
\[a_N^N\to e^{-\frac12\int_\Dd(h-\int_\Dd hf_T)^2f_T},\qquad \sqrt Nb_N\to i\Big(h-\int_\Dd hf_T\Big)\quad\text{in $L^2(f_T)$},\]
we deduce for all $m\ge0$,
\begin{equation}\label{eq:char-hoeffding-limit}
\bar C_{N,m}^{T}\to\bar C_m^{T}:=
\frac{i^m}{\sqrt{m!}}e^{-\frac12\int_\Dd(h-\int_\Dd hf_T)^2f_T}\Big(h-\int_\Dd hf_T\Big)^{\otimes m}\quad\text{in $L^2(f_T^{\otimes m})$}.
\end{equation}
As the limiting coefficients satisfy the exponential summability condition~\eqref{eq:CmT-weight} for any $\lambda_0>0$, we can apply Theorem~\ref{th:bounded-obs} and deduce
\[\E_{F_N(T)}\Big[e^{i\langle\zeta_N(T),h\rangle}\Big]
\to\sum_{m\ge0}\int_{\Dd^m}\bar C_m^{T}\bar H_m(T)f_T^{\otimes m},\]
where $(\bar H_m)_m$ is the unique weak solution of the limiting primal hierarchy~\eqref{eq:limhier-bidual}.
Inserting the explicit form~\eqref{eq:char-hoeffding-limit} of $(\bar C_m^T)_m$, Lemma~\ref{lem:limfluct-char} shows that the right-hand side is precisely equal to $\E[e^{i\langle\zeta(T),h\rangle}]$. The same argument applies to every linear combination of finitely many test functions, and L\'evy's continuity theorem then yields the conclusion.
\end{proof}

Next, we show that the convergence in law holds not only in~$\Dc'(\Dd)$, but also in $H^{-s}(\Dd)$ for $s>d$. This requires some additional care and is deduced from another application of Theorem~\ref{th:bounded-obs}, showing that the norm $\|\zeta_N(T)\|_{H^{-s}}$ also converges in law. Together with the convergence in law in $\Dc'(\Dd)$, this allows us to conclude by the probabilistic version of the Radon-Riesz property (see e.g.~\cite{Merkle1989}): if random elements in a separable Hilbert space converge in law with respect to the weak topology, and if their norms also converge in law, then the convergence in law holds with respect to the strong topology.

\begin{lem}\label{cor:fluct-tfixed}
Assume the hypotheses of Theorem~\ref{th:fluct}. Then, for all $T\in[0,T_0]$ and~$s>d$, we have under~$F_N(T)$,
\[\zeta_N(T)\Rightarrow\zeta(T)\quad\text{in $H^{-s}(\Dd)$}.\]
\end{lem}

\begin{proof}
It remains to upgrade the convergence in Lemma~\ref{cor:fluct-tfixed0} to $H^{-s}(\Dd)$ for $s>d$. We proceed by examining the convergence in law of the norm $\|\zeta_N(T)\|_{H^{-s}}$. To this aim, we study the Laplace transform of the squared norm and apply Theorem~\ref{th:bounded-obs} to the bounded symmetric observables
\[\Phi_{N,a}^T:=e^{-a\|\zeta_N(T)\|_{H^{-s}}^2},\qquad a\ge0.\]
In order to check the assumptions of Theorem~\ref{th:bounded-obs}, we need to compute the limiting Hoeffding components, for which we appeal to the Gaussian characteristic function identity. Let $G$ be a centered Gaussian random field on $\Dd$ with covariance operator~$(1-\triangle_z)^{-s}$, independent of the particle system, so that $\E_G[\langle h,G\rangle^2]=\|h\|_{H^{-s}}^2$ for $h\in C^\infty_c(\Dd)$, where~$\E_G$ denotes expectation with respect to $G$. Since $s>d$, we may choose a continuous version of $G$, and the pairing $\langle h,G\rangle$ extends in $L^2(\sigma(G))$ to $h\in H^{-s}(\Dd)$. By the Gaussian characteristic function identity, we can then represent
\[\Phi_{N,a}^T=\E_G\Big[e^{i\sqrt{2a}\langle\zeta_N(T),G\rangle}\Big].\]
Using~\eqref{eq:char-hoeffding-limit} conditionally on $G$, we deduce for all $m\ge0$,
\[\bar C_{N,a,m}^{T}:=\binom{N}{m}^\frac12\Pi_{N,m}^T[\Phi_{N,a}^T]\to\bar C_{a,m}^{T}
\quad\text{in $L^2(f_T^{\otimes m})$},\]
where
\begin{equation}\label{eq:explicit-Cam}
\bar C_{a,m}^{T}:=\frac{(i\sqrt{2a})^m}{\sqrt{m!}}\E_G\Big[e^{-a\|\tilde G\|_{L^2(f_T)}^2}\tilde G^{\otimes m}\Big],\qquad\tilde G:=G-\int_\Dd Gf_T.
\end{equation}
From this explicit expression, Jensen's inequality yields for $\lambda\ge0$,
\begin{multline}\label{eq:fock-weight-norm}
\sum_{m\ge0}e^{\lambda m}\|\bar C_{a,m}^{T}\|_{L^2(f_T^{\otimes m})}^2
\,\le\,\E_G\Big[e^{-2a\|\tilde G\|_{L^2(f_T)}^2}\sum_{m\ge0}\frac{(2a e^{\lambda})^{m}}{m!}\|\tilde G\|_{L^2(f_T)}^{2m}\Big]\\
\,=\,\E_G\Big[e^{2a (e^{\lambda}-1)\|\tilde G\|_{L^2(f_T)}^2}\Big]
\,\le\,\E_G\Big[e^{2a (e^{\lambda}-1)\|G\|_{L^2(f_T)}^2}\Big].
\end{multline}
Now note that the choice $s>d$ precisely ensures that the kernel $\Gc_s$ of $(1-\triangle_z)^{-s}$ is bounded on the diagonal, hence
\begin{equation}\label{eq:norm-bnd-GL2f}
\E_G\big[\|G\|_{L^2(f_T)}^2\big]=\int_\Dd \Gc_s(z,z)f_T(z)dz=\Gc_s(0)<\infty,
\end{equation}
which shows that the Gaussian field $G$ admits an $L^2(f_T)$-valued realization. Fernique's theorem then entails that~\eqref{eq:fock-weight-norm} is finite for $a>0$ small enough depending on $\lambda$.
Choosing $\lambda=\lambda_0(T+\sqrt T)$, we can therefore apply Theorem~\ref{th:bounded-obs} and conclude that there is some~$a_*(T)>0$ such that for all $0\le a<a_*(T)$,
\begin{equation}\label{eq:laplace-lim}
\E_{F_N(T)}\Big[e^{-a\|\zeta_N(T)\|_{H^{-s}(\Dd)}^2}\Big]
\to L_T(a):=\sum_{m\ge0}\int_{\Dd^m}\bar C_{a,m}^{T}\bar H_m(T)f_T^{\otimes m},
\end{equation}
where $(\bar H_m)_m$ is the unique weak solution of the limiting primal hierarchy~\eqref{eq:limhier-bidual}.
Using~\eqref{eq:bnd-Hm} and~\eqref{eq:fock-weight-norm}, and observing that $\bar H_0(T)=1$, that $\bar C_{a,0}^T\to1$, and that the weighted Fock norm of $(\bar C_{a,m}^T)_{m\ge1}$ tends to zero, we obtain
\begin{equation}\label{eq:limLT1}
L_T(a)\to1\quad\text{as $a\downarrow0$.}
\end{equation}
In particular, \eqref{eq:laplace-lim} and~\eqref{eq:limLT1} imply tightness of $\|\zeta_N(T)\|_{H^{-s}}$. Together with Lemma~\ref{cor:fluct-tfixed0}, this shows that $\zeta(T)\in H^{-s}(\Dd)$ almost surely.

Next, we identify the limit $L_T(a)$ with the Laplace transform of~$\|\zeta(T)\|_{H^{-s}}^2$.
Applying~\eqref{eq:char-zetat} conditionally on $G$ with~$h=\sqrt{2a}G$, and then taking the expectation with respect to $G$, we obtain for all $a\ge0$,
\begin{eqnarray*}
\E\Big[e^{-a\|\zeta(T)\|_{H^{-s}}^2}\Big]
&=&\E_G\E\Big[e^{i\sqrt{2a}\langle \zeta(T),G\rangle}\Big]\\
&=&\E_G\bigg[e^{-a\|\tilde G\|_{L^2(f_T)}^2}\sum_{m\ge0}\frac{(i\sqrt{2a})^m}{\sqrt{m!}}\int_{\Dd^m}\tilde G^{\otimes m}\bar H_m(T) f_T^{\otimes m}\bigg].
\end{eqnarray*}
and thus, recognizing~\eqref{eq:explicit-Cam},
\[\E\Big[e^{-a\|\zeta(T)\|_{H^{-s}(\Dd)}^2}\Big]
=\sum_{m\ge0}\int_{\Dd^m}\bar C_{a,m}^{T}\bar H_m(T) f_T^{\otimes m}.\]
Hence, \eqref{eq:laplace-lim} becomes, for all $0\le a<a_*(T)$,
\[\E_{F_N(T)}\Big[e^{-a\|\zeta_N(T)\|_{H^{-s}(\Dd)}^2}\Big]\to L_T(a)=\E\Big[e^{-a\|\zeta(T)\|_{H^{-s}(\Dd)}^2}\Big].\]
Recalling~\eqref{eq:limLT1}, the continuity theorem for Laplace transforms then yields
\begin{equation}\label{eq:conv-norm}
\|\zeta_N(T)\|_{H^{-s}}\Rightarrow\|\zeta(T)\|_{H^{-s}}\quad\text{under $F_N(T)$}.
\end{equation}
The conclusion now follows from a probabilistic version of the Radon-Riesz property on the Hilbert space~$H^{-s}(\Dd)$. Indeed,~\eqref{eq:conv-norm} implies tightness of the norms and hence tightness of~$\zeta_N(T)$ for the weak Hilbert topology. By Lemma~\ref{cor:fluct-tfixed0}, any joint subsequential limit $(Z,Y)$ of $(\zeta_N(T),\|\zeta_N(T)\|_{H^{-s}})$ has $Z$ distributed as $\zeta(T)$, and weak lower semicontinuity yields $Y\ge\|Z\|_{H^{-s}}$ almost surely. Since~\eqref{eq:conv-norm} ensures that $Y$ and $\|Z\|_{H^{-s}}$ have the same law, we deduce they must be equal almost surely, and the usual Radon-Riesz property then upgrades weak to strong convergence, concluding the proof.
\end{proof}

\subsection{Finite-dimensional distributions}\label{sec:several-times}
It remains to iterate the preceding duality argument across observation times. We give the two-time argument in detail; the general case follows by a direct iteration of the proof, thus concluding the proof of Theorem~\ref{th:fluct}.

\begin{prop}\label{prop:fdd-fluct}
Assume the hypotheses of Theorem~\ref{th:fluct} and denote by $\Gamma_N$ a consistent multi-time weak duality extension of $F_N$ in the sense of Appendix~\ref{app:multitime}. Then, for all $t_0,t_1\in[0,T_0]$ and $s>d$, we have
\[(\zeta_N(t_0),\zeta_N(t_1))\Rightarrow(\zeta(t_0),\zeta(t_1))\quad\text{in $H^{-s}(\Dd)^2$, under $\Gamma_N^{t_0,t_1}$}.\]
\end{prop}

\begin{proof}
By Lemma~\ref{cor:fluct-tfixed}, it suffices to prove joint convergence in law in $\Dc'(\Dd)^2$.
Let $0\le t_0<t_1\le T_0$ and $h_0,h_1\in C_c^\infty(\Dd)$, and consider
\[L_N^{t_1,t_0}(h_1,h_0):=\int_{\Dd^N\times\Dd^N}e^{i\langle\zeta_N(t_1),h_1\rangle+i\langle\zeta_N(t_0),h_0\rangle}d\Gamma_N^{t_0,t_1},\]
where the empirical fluctuation fields are understood as functions of the configurations at the corresponding observation times. We express this quantity in terms of iterated backward Liouville solutions.
By definition of a consistent multi-time extension, cf.~Definition~\ref{def:multi-time-extension}, there exists a bounded backward weak Liouville solution $\Phi_{N}^1$ on $(-\infty,t_1]$ with final data
\[\Phi_{N}^1(t_1)=e^{i\langle\zeta_N(t_1),h_1\rangle},\]
and there exists a bounded backward weak Liouville solution $\Phi_{N}^0$ on $(-\infty,t_0]$ with final data
\[\Phi_{N}^0(t_0)=e^{i\langle\zeta_N(t_0),h_0\rangle}\Phi_{N}^1(t_0),\]
such that the following duality identity holds,
\begin{equation}\label{eq:LN-dual}
L_N^{t_1,t_0}(h_1,h_0)=\int_{\Dd^N}\Phi_{N}^0(0)F_N^\circ.
\end{equation}

Since the final data $\Phi_{N}^1(t_1)$ is of the form~\eqref{eq:terminal-char-observable}, its rescaled Hoeffding coefficients converge strongly, cf.~\eqref{eq:char-hoeffding-limit}, and we can then apply Theorem~\ref{th:bounded-obs} as in the previous section: for the rescaled Hoeffding components
\begin{equation}\label{eq:CNt1-def}
\bar C_{N,m}^1:=\binom{N}{m}^\frac12\Pi_{N,m}\Phi_{N}^1,
\end{equation}
we obtain for all $m\ge0$,
\begin{equation}\label{eq:CNt1-conv}
\bar C_{N,m}^1(t_0)\cvf\bar C_{m}^1(t_0)\quad\text{in $L^2(f_{t_0}^{\otimes m})$},
\end{equation}
where $\bar C^1=(\bar C_{m}^1)_m$ is the solution of the limiting dual hierarchy with final data
\[\bar C_{m}^1(t_1)=\frac{i^m}{\sqrt{m!}}e^{-\frac12\int_\Dd(h_1-\int_\Dd h_1f_{t_1})^2f_{t_1}}\Big(h_1-\int_\Dd h_1f_{t_1}\Big)^{\otimes m}.\]
Recalling the explicit form~\eqref{eq:Cm-expl} of the solution of the dual hierarchy, applied with these final data, and noting that the series in $\ell$ sums to an exponential, we get
\begin{multline*}
\bar C_{m}^1(t_0)=\frac{i^m}{\sqrt{m!}}\Big(B_{t_1,t_0}^*h_1-\int_\Dd (B_{t_1,t_0}^*h_1)f_{t_0}\Big)^{\otimes m}\\
\times\exp\bigg(\int_{t_0}^{t_1}\int_{\Dd^2} K_{f_s}(B_{t_1,s}^*h_1)^{\otimes2}f_{s}^{\otimes2}\,ds-\frac12\int_\Dd\Big(h_1-\int_\Dd h_1f_{t_1}\Big)^2f_{t_1}\bigg).
\end{multline*}
Summing Wiener-It\^o chaoses, recalling the chaos expansion of the Gaussian exponential, cf.~\eqref{eq:expon-chaos}, and using the energy identity~\eqref{eq:energy-U-fluct}, we deduce
\begin{equation}\label{eq:CNt1-form}
\Ic_{t_0}(\bar C^1(t_0))=
e^{iW_{t_0}(B_{t_1,t_0}^*h_1)}\exp\Big(-\alpha\int_{t_0}^{t_1}\int_\Dd|\nabla_vB_{t_1,s}^*h_1|^2f_s\,ds\Big).
\end{equation}

Let us now turn to $\Phi_{N}^0$ and examine its final condition $e^{i\langle\zeta_N(t_0),h_0\rangle}\Phi_{N}^1(t_0)$. We need to identify its limiting Hoeffding components, which will require a suitable weak-strong argument. This can easily become cumbersome to relate the components of the product to the product of the separate components. To bypass long unpleasant calculations,  we work on the Fock space and use the correspondence between Hoeffding and Wiener-It\^o isometries, cf.~Section~\ref{sec:Hoeffd-Ito}.
For $h\in C^\infty_c(\Dd)$, consider the multiplication operator
\[M_N(h)G:=(\Ic_N^{t_0})^{-1}\Big(e^{i\langle\zeta_N(t_0),h\rangle}\Ic_N^{t_0}(G)\Big),\qquad G\in\Gamma_N(L^2_0(f_{t_0})),\]
and its limiting counterpart
\[M(h)G:=(\Ic_{t_0})^{-1}\Big(e^{iW_{t_0}(h)}\Ic_{t_0}(G)\Big),\qquad G\in\Gamma(L^2_0(f_{t_0})).\]
These are unitary operators on $\Gamma_N(L^2_0(f_{t_0}))$ and $\Gamma(L^2_0(f_{t_0}))$, respectively.
Denoting by $P_N:\Gamma(L^2_0(f_{t_0}))\to\Gamma_N(L^2_0(f_{t_0}))$ the orthogonal projection, we claim for all $G\in \Gamma(L^2_0(f_{t_0}))$,
\begin{equation}\label{eq:conv-MN(h)}
M_N(h)P_NG\to M(h)G\quad\text{strongly in $\Gamma(L^2_0(f_{t_0}))$}.
\end{equation}
To prove this, we first consider Gaussian exponentials $e^{iW_{t_0}(g)}$, $g\in C^\infty_c(\Dd)$, and consider the corresponding Hoeffding and Wiener-It\^o chaos coefficients,
\[E_N(g):=(\Ic_N^{t_0})^{-1}e^{i\langle\zeta_N(t_0),g\rangle}\in \Gamma_N(L^2_0(f_{t_0})),\quad E(g):=(\Ic_{t_0})^{-1}(e^{iW_{t_0}(g)})\in\Gamma(L^2_0(f_{t_0})).\]
By~\eqref{eq:char-hoeffding-limit}, together with the trivial convergence of the norms
\[\int_{\Dd^N}|e^{i\langle\zeta_N(t_0),g\rangle}|^2f_{t_0}^{\otimes N}=1=\E\Big[|e^{iW_{t_0}(g)}|^2\Big],\]
we already know
\[E_N(g)\to E(g)\quad\text{in $\Gamma(L^2_0(f_{t_0}))$}.\]
Since $M_N(h)E_N(g)=E_N(h+g)$ and $M(h)E(g)=E(h+g)$, the unitarity of $M_N(h)$ and the convergence $P_NE(g)\to E(g)$ in $\Gamma(L^2_0(f_{t_0}))$ imply~\eqref{eq:conv-MN(h)} for any $G=E(g)$, $g\in C^\infty_c(\Dd)$. The claim~\eqref{eq:conv-MN(h)} then follows by density of finite linear combinations of such Gaussian exponential vectors in $\Gamma(L^2_0(f_{t_0}))$.

With this observation at hand, we may now characterize the limit of Hoeffding coefficients of the final condition $\Phi_N^0(t_0)=e^{i\langle\zeta_N(t_0),h_0\rangle}\Phi_{N}^1(t_0)$,
\[\bar C_{N}^0(t_0):=\bigg(\binom{N}{m}^\frac12\Pi_{N,m}^{t_0}\Big[e^{i\langle\zeta_N(t_0),h_0\rangle}\Phi_{N}^1(t_0)\Big]\bigg)_{0\le m\le N}.\]
In terms of the above multiplication operator $M_N(h_0)$, this can be written as
\[\bar C_{N}^0(t_0)=M_N(h_0)\bar C_{N}^1(t_0).\]
By~\eqref{eq:CNt1-conv}, together with the boundedness of $\bar C_{N}^1$ in the Fock space, we have
\[\bar C_{N}^1(t_0)\cvf\bar C^1(t_0)\quad\text{weakly in $\Gamma(L^2_0(f_{t_0}))$}.\]
By~\eqref{eq:conv-MN(h)}, we may then deduce
\begin{equation}\label{eq:conv-t1t0-cum}
\bar C_{N}^0(t_0)\cvf \bar C^0(t_0):=M(h_0)\bar C^1(t_0)\quad\text{in $\Gamma(L^2_0(f_{t_0}))$}.
\end{equation}
Moreover, the limit can be identified explicitly: by~\eqref{eq:CNt1-form} and the definition of $M(h_0)$,
\[\Ic_{t_0}(\bar C^0(t_0))
=e^{iW_{t_0}(h_0+B_{t_1,t_0}^*h_1)}\exp\Big(-\alpha\int_{t_0}^{t_1}\int_\Dd|\nabla_vB_{t_1,s}^*h_1|^2f_s\,ds\Big),\]
so that the chaos coefficients are given by
\begin{equation}\label{eq:form-t1t0-cum}
\bar C_{m}^0(t_0)=\frac{i^m}{\sqrt{m!}}q_0^{\otimes m}\exp\Big(-\frac12\int_\Dd|q_0|^2f_{t_0}-\alpha\int_{t_0}^{t_1}\int_\Dd|\nabla_vB_{t_1,s}^*h_1|^2f_s\,ds\Big),
\end{equation}
in terms of
\[q_0:=(h_0+B_{t_1,t_0}^*h_1)-\int_\Dd(h_0+B_{t_1,t_0}^*h_1)f_{t_0}.\]
In particular, $\bar C^0(t_0)$ belongs to every exponentially weighted Fock space.

We can thus apply Theorem~\ref{th:bounded-obs} on $[0,t_0]$ with the terminal observable $e^{i\langle\zeta_N(t_0),h_0\rangle}\Phi_{N}^1(t_0)$. Starting from~\eqref{eq:LN-dual} and using~\eqref{eq:conv-t1t0-cum}, it leads us to
\[L_N^{t_1,t_0}(h_1,h_0)\to L^{t_1,t_0}(h_1,h_0):=\sum_{m\ge0}\int_{\Dd^m} \bar C_{m}^0(t_0)\bar H_m(t_0)f_{t_0}^{\otimes m}.\]
It remains to express the right-hand side in terms of $\zeta$. By~\eqref{eq:form-t1t0-cum} and Lemma~\ref{lem:limfluct-char},
\[L^{t_1,t_0}(h_1,h_0)=\E\Big[e^{i\langle\zeta(t_0),h_0+B_{t_1,t_0}^*h_1\rangle}\Big]\exp\Big(-\alpha\int_{t_0}^{t_1}\int_\Dd|\nabla_vB_{t_1,s}^*h|^2f_s\,ds\Big).\]
The mild formula for $\zeta$ gives
\[\langle\zeta(t_1),h_1\rangle=\langle\zeta(t_0),B_{t_1,t_0}^*h_1\rangle+\sqrt{2\alpha}\int_{t_0}^{t_1}\int_\Dd\sqrt{f_s}\,\nabla_vB_{t_1,s}^*h_1\cdot\xi(ds,dz),\]
where the last term is independent of $\zeta(t_0)$. Hence,
\[\E\Big[e^{i\langle\zeta(t_1),h_1\rangle}\Big\|\zeta(t_0)\Big]=e^{i\langle\zeta(t_0),B_{t_1,t_0}^*h_1\rangle}\exp\Big(-\alpha\int_{t_0}^{t_1}\int_\Dd|\nabla_vB_{t_1,s}^*h_1|^2f_s\,ds\Big),\]
and the above becomes
\[L^{t_1,t_0}(h_1,h_0)=\E\Big[e^{i\langle\zeta(t_0),h_0\rangle+i\langle\zeta(t_1),h_1\rangle}\Big].\]
By L\'evy's continuity theorem, this concludes the proof.
\end{proof}

\subsection*{Acknowledgements}
MD acknowledges financial support from the European Union (ERC, PASTIS, Grant Agreement n$^\circ$101075879).\footnote{{Views and opinions expressed are however those of the authors only and do not necessarily reflect those of the European Union or the European Research Council Executive Agency. Neither the European Union nor the granting authority can be held responsible for them.}} PEJ was partially supported by NSF DMS Grant 2508570.
The authors thank Xuanrui Feng for pointing out a mistake in a previous proof of Proposition~\ref{prop:bnd-nodiff}.

\appendix
\section{Kinetic regularity}\label{app:kinetic-reg}
For completeness, in this appendix, we include short self-contained proofs of the specific kinetic regularity statements of Lemmas~\ref{lem:hypoell} and~\ref{lem:averaging}.

\subsection{Proof of Lemma~\ref{lem:hypoell}}
The result essentially follows from~\cite{Bouchut-02}. Yet, to account for non-zero initial data and to emphasize that multiplicative constants are uniform in $m$, we include a direct proof based on explicit Fourier calculation.
Denote by $\xi=(\xi_1,\ldots,\xi_m)$ and $\eta=(\eta_1,\ldots,\eta_m)$ the Fourier variables associated with $(x_1,\ldots,x_m)$ and $(v_1,\ldots,v_m)$, respectively.
Let $P$ denote the Fourier projection onto $|\xi|>1$. The low-frequency contribution is simply bounded by
\[\int_0^T\||\nabla_{x_{[m]}}|^{r} (1-P)g_m(t)\|_{L^2}^2\,dt \,\le\, \int_0^T\|g_m(t)\|_{L^2}^2\qquad\text{for $r\ge0$}.\]
It remains to prove the following for the high-frequency part: for $0\le r\le\frac13$,
\begin{equation*}
\int_0^T\||\nabla_{x_{[m]}}|^{\frac{1}3-r} Pg_m\|_{L^2}^2
\le C\|g_m(0)\|_{L^2}^2+C\int_0^T\|\langle\nabla_{x_{[m]}}\rangle^{-r}\langle\nabla_{v_{[m]}}\rangle^{-1}e_m\|_{L^2}^2,
\end{equation*}
for some constant $C$ only depending on $\alpha$.
Applying $\langle\nabla_{x_{[m]}}\rangle^{-r}$ to both sides of the equation for $g_m$, we find that it suffices to prove the above with $r=0$.
By Duhamel's formula, the solution $g_m$ reads
\begin{equation*}
g_m(t)=S_t^mg_m(0)+\int_0^t S_{t-s}^me_m(s)\,ds,
\end{equation*}
where $(S_t^m)_{t\ge0}$ is the semigroup generated by $-\sum_{j=1}^m(v_j\cdot\nabla_{x_j}-\alpha\triangle_{v_j})$.
We split the proof into two steps, separately estimating the contribution of initial data and the contribution of the source term.

\medskip\noindent
{\bf Step 1:} Contribution of initial data.\\
We claim that for all $r\ge0$,
\begin{equation}\label{eq:initial-hypo}
{\int_0^T \| |\nabla_{x_{[m]}}|^{r}PS_t^mg_m(0)\|_{L^2}^2\,dt}
\le C\|\langle\nabla_{x_{[m]}}\rangle^{r-\frac13}g_m(0)\|_{L^2(\Dd^m)}^2.
\end{equation}
Taking Fourier transform, a classical calculation of the semigroup $S_t^m$ yields explicitly
\begin{equation}\label{eq:explicit-Stm}
\widehat{S_t^m f}(\xi,\eta)=e^{-\alpha\int_0^t |\eta+s\xi|^2\,ds}\,\hat f(\xi,\eta+t\xi).
\end{equation}
Changing variables $\zeta=\eta+t\xi$, we then get
\begin{equation*}
\||\nabla_{x_{[m]}}|^{r}PS_t^mg_m(0)\|_{L^2}^2
=\int_{\Dd^m} \mathds1_{|\xi|>1}|\xi|^{2r}e^{-2\alpha\int_0^t |\zeta-s\xi|^2\,ds}\,|\hat g_m(0,\xi,\zeta)|^2\,d\xi d\zeta.
\end{equation*}
Since
\[\int_0^t |\zeta-s\xi|^2\,ds=t\Big|\zeta-\frac t2\xi\Big|^2+\frac{t^3}{12}|\xi|^2\ge \frac{t^3}{12}|\xi|^2,\]
we obtain
\begin{eqnarray*}
{\int_0^T \| |\nabla_{x_{[m]}}|^{r}PS_t^mg_m(0)\|_{L^2}^2\,dt}
&\le& \int_{\Dd^m}\mathds1_{|\xi|>1}|\xi|^{2r}\Big(\int_0^T e^{-\frac16\alpha t^3|\xi|^2}dt\Big)|\hat g_m(0,\xi,\zeta)|^2\,d\xi d\zeta\\
&\le& C\||\nabla_{x_{[m]}}|^{r-\frac13}Pg_m(0)\|_{L^2}^2,
\end{eqnarray*}
for some constant $C$ only depending on $\alpha$. The claim~\eqref{eq:initial-hypo} follows.

\medskip\noindent
{\bf Step 2:} Contribution of the source term.\\
We claim that
\begin{equation}\label{eq:forced-hypo}
\int_0^T\bigg\||\nabla_{x_{[m]}}|^{\frac13}P\Big(\int_0^t S_{t-s}^m e_m(s)\,ds\Big)\bigg\|_{L^2}^2\,dt\le C\int_0^T\|\langle \nabla_{v_{[m]}} \rangle^{-1}e_m(t)\|_{L^2}^2\,dt,
\end{equation}
for some constant $C$ only depending on $\alpha$.
Taking Fourier transform, inserting the explicit form~\eqref{eq:explicit-Stm} of the semigroup, and again changing variables $\zeta=\eta+t\xi$, it suffices to prove
\begin{equation*}
\int_0^T\int_{\Dd^m}\mathds1_{|\xi|>1}\Big|\int_0^T K_{\xi,\zeta}(t,s)\,\hat E_m(s,\xi,\zeta)\,ds\Big|^2d\xi d\zeta\,dt
\le C\int_0^T\int_{\Dd^m}|\hat E_m(t,\xi,\zeta)|^2d\xi d\zeta\,dt.
\end{equation*}
where we have set for abbreviation
\begin{eqnarray*}
\hat E_m(s,{\xi,\zeta})&:=&\langle \zeta-s\xi \rangle^{-1}\hat e_m(s,\xi,\zeta-s\xi),\\
K_{{\xi,\zeta}}(t,s)&:=&\mathds1_{0\le s\le t\le T}|\xi|^{\frac13}\langle \zeta-s\xi\rangle e^{-\alpha\int_s^t |\zeta-u\xi|^2\,du}.
\end{eqnarray*}
Hence, it suffices to show that, uniformly for all $\xi,\zeta$ with $|\xi|>1$, the kernel $K_{\xi,\zeta}(t,s)$ defines an operator that is bounded on~$L^2(0,T)$, with norm bounded by a constant $C$ only depending on $\alpha$.
To this aim, by Schur's test, it suffices to show
\begin{equation}\label{eq:todo-schur}
\sup_{s\in[0,T]}\int_s^T K_{\xi,\zeta}(t,s)\,dt\le C,\qquad \sup_{t\in[0,T]}\int_0^t K_{\xi,\zeta}(t,s)\,ds\le C.
\end{equation}
We focus on the first bound, while the second one can be checked similarly.
Let $\xi,\zeta$ be fixed with $|\xi|>1$. Using that
\[\int_s^t|\zeta-u\xi|^2\,du=(t-s)|\zeta_{t,s}|^2+\frac{1}{12}(t-s)^3|\xi|^2,\qquad \zeta_{t,s}:=\zeta-s\xi-\frac12(t-s)\xi,\]
and bounding
\[\langle \zeta-s\xi\rangle\le 1+|\zeta_{t,s}|+(t-s)|\xi|,\]
we get
\begin{equation*}
\langle \zeta-s\xi\rangle \int_s^\infty e^{-\alpha\int_s^t|\zeta-u\xi|^2\,du}\,dt
\le\int_s^\infty \big(1+|\zeta_{t,s}|+(t-s)|\xi|\big)\, e^{-\alpha(t-s)|\zeta_{t,s}|^2-\frac{\alpha}{12}(t-s)^3|\xi|^2}\,dt.
\end{equation*}
Using the elementary bound $|\zeta_{t,s}|e^{-\alpha(t-s)|\zeta_{t,s}|^2}\le C(t-s)^{-\frac12}$ and estimating the integrals, we obtain
\begin{eqnarray*}
\langle \zeta-s\xi\rangle \int_s^\infty e^{-\alpha\int_s^t|\zeta-u\xi|^2\,du}\,dt
&\le&\int_s^\infty \Big(1+C (t-s)^{-\frac12}+|\xi| (t-s)\Big)e^{-\frac{\alpha}{12}(t-s)^3|\xi|^2}\,dt\\
&\le&C(|\xi|^{-\frac23}+|\xi|^{-\frac13}).
\end{eqnarray*}
For $|\xi|>1$, this is controlled by $|\xi|^{-\frac13}$, and the first bound in~\eqref{eq:todo-schur} follows.
This concludes the proof of~\eqref{eq:forced-hypo}.\qed

\subsection{Proof of Lemma~\ref{lem:averaging}}
We follow the commutator approach of~\cite{JLT-22}.
Denote again by $\xi=(\xi_1,\ldots,\xi_m)$ and $\eta=(\eta_1,\ldots,\eta_m)$ the Fourier variables associated with $(x_1,\ldots,x_m)$ and $(v_1,\ldots,v_m)$, respectively. We shall show for any $r\in[0,1]$,
\begin{multline}\label{eq:todo-averaging}
\int_0^T\int_{\Dd^m}|\xi|^{\frac{1-r}2}\langle\eta\rangle^{-3}|\hat g_m(\xi,\eta)|^2d\xi d\eta\\
\lesssim \|g_m(0)\|_{L^2}^2+\int_0^T\|g_m\|_{L^2}^2
+\int_0^T\int_{\Dd^m}\langle\xi\rangle^{-2r}\langle\eta\rangle^{-2}|\hat e_m(\xi,\eta)|^2d\xi d\eta.
\end{multline}
First note that we can focus on estimating the high-frequency part in physical space. More precisely, given a cut-off function $\chi\in C^\infty_c(\R^+)$ with $\chi(a)=1$ for $a\le1$, $\chi(a)=0$ for $a\ge2$, and $0\le \chi\le1$, the low-frequency part of the norm is simply controlled by
\begin{equation}\label{eq:low-freq-aver}
\int_0^T\int_{\Dd^m}\chi(|\xi|)|\xi|^\frac{1-r}2\langle\eta\rangle^{-3}|\hat g_m(\xi,\eta)|^2d\xi d\eta
\lesssim \int_0^T \|g_m\|_{L^2}^2.
\end{equation}
Next, let us regularize in velocity at the scale $\langle\xi\rangle^{-\frac{1-r}2}$. More precisely, consider the cut-off function $\gamma(\xi,\eta):=\chi(\langle\xi\rangle^{-\frac{1-r}2}|\eta|)$. Noting that
\[(1-\gamma(\xi,\eta))^2{|\xi|^\frac{1-r}2}{\langle\eta\rangle^{-3}}\le {|\xi|^\frac{1-r}2}{\langle\eta\rangle^{-3}}\mathds1_{|\eta|\ge\langle\xi\rangle^{(1-r)/2}}\le \langle\eta\rangle^{-2}\le1,\]
we find
\[\int_0^T\int_{\Dd^m}(1-\gamma(\xi,\eta))^2|\xi|^\frac{1-r}2\langle\eta\rangle^{-3}|\hat g_m(\xi,\eta)|^2d\xi d\eta
\lesssim \int_0^T\|g_m\|_{L^2}^2.\]
From this together with~\eqref{eq:low-freq-aver}, we deduce that it remains to prove the following estimate, instead of~\eqref{eq:todo-averaging}, for any $r\in[0,1]$,
\begin{multline}\label{eq:todo-averaging-re}
\int_0^T\int_{\Dd^m}(1-\chi(|\xi|))|\xi|^{\frac{1-r}2}\langle\eta\rangle^{-3}|(\gamma\hat g_m)(\xi,\eta)|^2d\xi d\eta\\
\lesssim \|g_m(0)\|_{L^2}^2
+\int_0^T\|g_m\|_{L^2}^2
+\int_0^T\int_{\Dd^m}\langle\xi\rangle^{-2r}\langle\eta\rangle^{-2}|\hat e_m(\xi,\eta)|^2d\xi d\eta.
\end{multline}
Taking Fourier transform, the equation for $g_m$ reads
\[\partial_t\hat g_m-\xi\cdot\nabla_\eta\hat g_m=\hat e_m,\]
and thus, multiplying by the cut-off $\gamma$,
\[\partial_t(\gamma\hat g_m)-\xi\cdot\nabla_\eta(\gamma\hat g_m)=\gamma\hat e_m-(\xi\cdot\nabla_\eta\gamma)\hat g_m.\]
As in~\cite{JLT-22}, consider the Fourier multiplier
\[M(\xi,\eta)=1+\frac{\xi\cdot\eta}{|\xi|\langle\eta\rangle},\]
which satisfies $0\le M(\xi,\eta)\le 2$ and
\begin{equation}\label{eq:commut-M}
\xi\cdot\nabla_\eta M(\xi,\eta)=\frac{|\xi|}{\langle\eta\rangle}-\frac{(\xi\cdot\eta)^2}{|\xi|\langle\eta\rangle^3}\ge\frac{|\xi|}{\langle\eta\rangle}-\frac{|\xi||\eta|^2}{\langle\eta\rangle^3}=\frac{|\xi|}{\langle\eta\rangle^3}.
\end{equation}
Multiplying the above equation for $\gamma\hat g_m$ with $2(1-\chi(|\xi|))|\xi|^{\frac{1-r}2-1}M(\xi,\eta)(\overline{\gamma\hat g_m})(\xi,\eta)$, taking the real part, and using $0\le M\le2$, we get
\begin{multline*}
\int_0^T\int_{\Dd^m}(1-\chi(|\xi|))|\xi|^{\frac{1-r}2-1}(\xi\cdot\nabla_\eta M)(\xi,\eta)|(\gamma\hat g_m)(\xi,\eta)|^2d\xi d\eta\\
\le2\int_{\Dd^m} (1-\chi(|\xi|))|\xi|^{\frac{1-r}2-1}|\gamma\hat g_m(0,\xi,\eta)|^2\\
+4\int_0^T\int_{\Dd^m} (1-\chi(|\xi|))|\xi|^{\frac{1-r}2-1}|(\gamma\hat g_m)(\xi,\eta)||(\xi\cdot\nabla_\eta\gamma)\hat g_m(\xi,\eta)|d\xi d\eta\\
+4\int_0^T\int_{\Dd^m} (1-\chi(|\xi|))|\xi|^{\frac{1-r}2-1}|(\gamma\hat g_m)(\xi,\eta)||(\gamma\hat e_m)(\xi,\eta)|d\xi d\eta.
\end{multline*}
Using~\eqref{eq:commut-M} to bound the left-hand side from below, noting that $|\xi|\ge1$ on the support of $1-\chi(|\xi|)$, that $|\xi\cdot\nabla_\eta\gamma|\lesssim\langle\xi\rangle^{\frac{1+r}2}$, and that $|\eta|\le2\langle\xi\rangle^\frac{1-r}2$ on the support of $\gamma(\xi,\eta)$, we deduce
\begin{multline*}
\int_0^T\int_{\Dd^m}(1-\chi(|\xi|))|\xi|^{\frac{1-r}2}\langle\eta\rangle^{-3}|(\gamma\hat g_m)(\xi,\eta)|^2d\xi d\eta\\
\lesssim \|g_m(0)\|_{L^2}^2
+\int_0^T\|g_m\|_{L^2}^2
+\int_0^T\int_{\Dd^m}\langle\xi\rangle^{-r}\langle\eta\rangle^{-1}|\hat g_m(\xi,\eta)||\hat e_m(\xi,\eta)|d\xi d\eta,
\end{multline*}
and the claim~\eqref{eq:todo-averaging-re} follows.

Assume finally that $e_m=e_m^0+\mbox{div}_v e_m^1$. Going back to the above calculation, we keep the same estimate for $e_m^1$ and we note that for $e_m^0$ we immediately have
\begin{multline*}
\int_0^T\int_{\Dd^m} (1-\chi(|\xi|))|\xi|^{\frac{1-r}2-1}|(\gamma\hat g_m)(\xi,\eta)||(\gamma\hat e_m^0)(\xi,\eta)|d\xi d\eta\\
\le\int_0^T\int_{\Dd^m} |\xi|^{\frac{1-r}2-1}|\hat g_m(\xi,\eta)|\,|\hat e_m^0(\xi,\eta)|d\xi d\eta\\
\le\int_0^T \|g_m\|_{L^2(\Dd^m)}^2\,dt+\int_0^T \|\langle\nabla_x\rangle^{-\frac{1+r}{2}} e_m^0\|_{L^2(\Dd^m)}^2\,dt,
\end{multline*}
which yields the result by choosing $s=\frac{1+r}{2}$.
\qed

\section{Multi-time extensions of weak duality solutions}\label{app:multitime}

The notion of weak duality solution of the Liouville equation~\eqref{eq:Liouville}, as introduced in~\cite[Appendix~A]{BDJ-24}, only concerns one-time densities of the particle system and does not specify joint-time laws. To study fluctuations at several times, we complement it with a notion of consistent joint-time measures. As for weak duality solutions themselves, neither uniqueness nor a globally defined flow of microscopic trajectories is required. The inequality in~(iii) may be strict, reflecting the possible loss of mass in finite time. The iterated duality in~(iv) provides precisely the compatibility property needed in the multi-time characteristic-function argument in the proof of Theorem~\ref{th:fluct}; see Section~\ref{sec:several-times}.

\begin{defin}\label{def:multi-time-extension}
Let $F_N\in L^\infty(\R^+;L^1_+(\Dd^N))$ be a symmetric global weak duality solution of~\eqref{eq:Liouville}, in the sense of~\cite[Appendix~A]{BDJ-24}, with initial condition $F_N^\circ\in \Pc\cap L^1(\Dd^N)$. A \emph{consistent multi-time weak-duality extension} $\Gamma_N$
of $F_N$ is a family of nonnegative finite measures
\[\Gamma_N^{t_1,\ldots,t_m}\in\mathcal M_+((\Dd^N)^m),
\qquad 0\le t_1<\cdots<t_m,\quad m\ge1,\]
with the following properties:
\begin{enumerate}[(i)]
\item \emph{One-time laws:} $\Gamma_N^t=F_N(t)$ for all $t\ge0$;
\smallskip\item \emph{Exchangeability:} $\Gamma_N^{t_1,\ldots,t_m}$ is invariant under simultaneous permutations of particle labels at all observation times;
\smallskip\item \emph{Consistency:} For any two finite sets of observation times $I\subset J$,
\begin{equation}\label{eq:multi-subconsistency}
(\pi_{I,J})_*\Gamma_N^J\le\Gamma_N^I.
\end{equation}
where $\pi_{I,J}:(\Dd^N)^J\to(\Dd^N)^I$ is the canonical projection.
\smallskip\item \emph{Iterated duality:}
For all $0\le t_1<\ldots<t_m$ and all symmetric $h_1,\ldots,h_m\in C_b(\Dd^N;\C)$, there exist $\Phi_N^1,\ldots,\Phi_N^m$,
where $\Phi_N^j\in L^\infty((-\infty,t_j]\times\Dd^N)$ is a global symmetric bounded weak solution of the backward Liouville equation on $(-\infty,t_j]$, in the sense of~\cite[Definition~A.1(i)]{BDJ-24}, with recursively defined final conditions
\begin{equation}\label{eq:multi-recursive-terminal}
\Phi_N^m(t_m)=h_m,\qquad
\Phi_N^j(t_j)=h_j\Phi_N^{j+1}(t_j),\quad 1\le j<m,
\end{equation}
such that
\begin{equation}\label{eq:multi-dual}
\int_{(\Dd^N)^m}\Big(\prod_{j=1}^m h_j(Z_N^j)\Big)\,d\Gamma_N^{t_1,\ldots,t_m}(Z_N^1,\ldots,Z_N^m)
=\int_{\Dd^N}\Phi_N^1(0) F_N^\circ.
\end{equation}
\end{enumerate}
\end{defin}

Whenever the Liouville equation is well-posed in the renormalized sense and admits a unique global regular Lagrangian flow, the latter automatically provides a canonical multi-time extension: the joint laws $\Gamma_N^{t_1,\ldots,t_m}$ are simply the push-forwards of~$F_N^\circ$ by the flow at the different observation times, and the iterated backward duality~\eqref{eq:multi-dual} follows directly. In the deterministic case $\alpha=0$, such well-posedness is known, for instance, when $K\in BV_\loc(\R^d;\R^d)$, cf.~\cite{Bouchut-01}, or for the singular Hamiltonian interactions considered in~\cite{Hauray}, with $K=-\nabla V\in L^1_\loc(\R^d;\R^d)\cap BV_\loc(\R^d\setminus\{0\};\R^d)$ and $V(x)\ge-C(1+|x|^2)$. We refer to~\cite{LeBris-Lions} for the diffusive case $\alpha>0$.

As the following result shows, just like weak duality solutions themselves, cf.~\cite[Proposition~A.2]{BDJ-24}, multi-time weak duality extensions exist whenever $K\in L^1_\loc(\R^d;\R^d)$ and $F_N^\circ\in L^1\cap L^p(\Dd^N)$ for some $p>1$.

\begin{prop}\label{prop:multi-time}
Let
\[\alpha\ge0,\qquad K\in L^1_{\loc}(\R^d;\R^d),\qquad F_N^\circ\in\Pc\cap L^p(\Dd^N),\]
for some $p>1$, with $F_N^\circ$ symmetric.
Then there exist a symmetric global weak duality solution $F_N\in L^\infty(\R^+;L^1_+\cap L^p(\Dd^N))$ of~\eqref{eq:Liouville}, in the sense of~\cite[Appendix~A]{BDJ-24}, with initial condition $F_N^\circ$, as well as a consistent multi-time weak-duality extension $\Gamma_N$.
\end{prop}

\begin{proof}
We argue by approximation as in~\cite[Proposition~A.2]{BDJ-24}.
Let $K_\e\in C^\infty_b(\R^d;\R^d)$ and let $F_{N,\e}^\circ\in \Pc\cap L^\infty(\Dd^N)$ symmetric, with $K_\eps\to K$ in $L^1_\loc(\R^d;\R^d)$ and $F_{N,\e}^\circ\to F_N^\circ$ in $L^1\cap L^p(\Dd^N)$ as $\e\downarrow0$.
Denote by $Z_{N,\eps}$ the regularized $N$-particle process started with initial density $F_{N,\eps}^\circ$, and denote by $F_{N,\eps}(t)$ its one-time density. Since the phase-space drift is divergence-free, we have
\begin{equation}\label{eq:multi-density-bound}
\|F_{N,\eps}(t)\|_{L^1(\Dd^N)}=1,\qquad
\|F_{N,\eps}(t)\|_{L^p(\Dd^N)}\le\|F_{N,\eps}^\circ\|_{L^p(\Dd^N)}\le C_N,
\end{equation}
uniformly in $\e>0$ and $t\ge0$. By the argument of~\cite[Proposition~A.2]{BDJ-24}, up to extraction, $F_{N,\eps}$ converges weakly in $L^1_\loc(\Dd^N)\cap L^p(\Dd^N)$, locally uniformly in time, to a symmetric global weak duality solution $F_N\in L^\infty(\R_+;L^1_+\cap L^p(\Dd^N))$ of~\eqref{eq:Liouville} with initial data $F_N^\circ$.

Next, define $\Gamma_{N,\eps}^{t_1,\ldots,t_m}$ as the law of $(Z_{N,\e}(t_1),\ldots,Z_{N,\e}(t_m))$. It remains to examine its limit along the extraction. As the limit needs to be extracted jointly for all times $0\le t_1<\ldots<t_m$, a nontrivial compactness argument is first needed. We split the proof into two steps.

\medskip\noindent
{\bf Step 1:} Compactness.\\
Given $h\in C_c^\infty(\Dd^N)$, we claim that for all $t,s\ge0$,
\begin{equation}\label{eq:tightness}
\E\big[|h(Z_{N,\e}(t))-h(Z_{N,\e}(s))|\big]\le C_{N,h}\Big(|t-s|^{1-\frac1p}+\sqrt{\alpha|t-s|}\Big),
\end{equation}
uniformly in $\e>0$.
By It\^o's formula, we can decompose for $0\le s\le t$,
\begin{multline*}
|h(Z_{N,\e}(t))-h(Z_{N,\e}(s))|
\le\sqrt{2\alpha}\bigg|\sum_{i=1}^N\int_s^t\nabla_{v_i}h(Z_{N,\e}(r))\cdot dB_{i,N}(r)\bigg|\\
+\int_s^t\bigg|\Big(\sum_{i=1}^N(v_i\cdot\nabla_{x_i}+\alpha\triangle_{v_i})h+\frac1{N-1}\sum_{i\ne j}^NK_\e(x_i-x_j)\cdot\nabla_{v_i}h\Big)(Z_{N,\e}(r))\bigg|dr,
\end{multline*}
and thus, since $h$ is smooth and compactly supported and since $K_\e$ is bounded in $L^1_\loc$,
\begin{equation}\label{eq:tight-bnd}
\E\big[|h(Z_{N,\e}(t))-h(Z_{N,\e}(s))|\big]\le C_{N,h}\Big(|t-s|\|F_{N,\e}^\circ\|_{L^\infty(\Dd^N)}+\sqrt{\alpha|t-s|}\Big).
\end{equation}
This bound is not directly useful as $F_{N,\e}$ is not a priori bounded in $L^\infty(\Dd^N)$ as $\e\downarrow0$. Instead, given $M>0$, decompose the initial density as
\[F_{N,\eps}^\circ=(F_{N,\eps}^\circ\wedge M)+(F_{N,\eps}^\circ-M)_+.\]
Applying~\eqref{eq:tight-bnd} for the truncated initial density $F_{N,\eps}^\circ\wedge M$, while noting that the second term has mass
\[\int_{\Dd^N}(F_{N,\eps}^\circ-M)_+\le M^{1-p}\|F_{N,\eps}^\circ\|_{L^p(\Dd^N)}^p\le C_N^pM^{1-p},\]
we can deduce
\[\E\big[|h(Z_{N,\e}(t))-h(Z_{N,\e}(s))|\big]\le C_{N,h}\Big(M|t-s|+\sqrt{t-s}\Big)+2\|h\|_{L^\infty(\Dd^N)}C_N^pM^{1-p},\]
and the claim~\eqref{eq:tightness} then follows after optimizing in $M>0$.

Given $m\ge1$, $T>0$, and $h_1,\ldots,h_m\in C_c^\infty(\Dd^N)$, we may then deduce from~\eqref{eq:tightness} the uniform equicontinuity on $[0,T]^m$ of the maps
\[(t_1,\ldots,t_m)\mapsto\E\bigg[\prod_{j=1}^mh_j(Z_{N,\e}(t_j))\bigg].\]
By Arzel\`a-Ascoli and a diagonal extraction, we can deduce, along a single subsequence as~$\e\downarrow0$, the uniform convergence of these expectations on $[0,T]^m$ for all $m\ge1$, $T>0$, and $h_1,\ldots,h_m\in C_c^\infty(\Dd^N)$. As the definition of $\Gamma_{N,\e}$ precisely reads
\[\int_{(\Dd^N)^m}\big(h_1\otimes\ldots\otimes h_m\big)\,d\Gamma_{N,\eps}^{t_1,\ldots,t_m}=\E\bigg[\prod_{j=1}^mh_j(Z_{N,\e}(t_j))\bigg],\]
we conclude, along the extracted subsequence, for all $m\ge1$ and $t_1,\ldots,t_m\ge0$,
\begin{equation}\label{eq:conv-GammaNeps}
\Gamma_{N,\e}^{t_1,\ldots,t_m}\cvf\Gamma_{N}^{t_1,\ldots,t_m}\quad\text{vaguely on $(\Dd^N)^m$},
\end{equation}
for some limit $\Gamma_{N}^{t_1,\ldots,t_m}\in\calM_+((\Dd^N)^m)$.
Passing to the limit in the properties of $\Gamma_{N,\e}$ easily ensures that $\Gamma_N$ satisfies the properties~(i)--(iii) of Definition~\ref{def:multi-time-extension}.

\medskip\noindent
{\bf Step 2:} Iterated duality identity.\\
Consider a cut-off sequence $(\chi_R)_R\subset C_c^\infty(\Dd^N)$ with $0\le\chi_R\le1$ and $\chi_R\uparrow1$ as $R\uparrow\infty$.
Given $0\le t_1<\cdots<t_m$ and $h_1,\ldots,h_m\in C_b(\Dd^N;\C)$, construct recursively $\Phi_{N,\e,R}^1,\ldots,\Phi_{N,\e,R}^m$ as follows: let $\Phi_{N,\e,R}^j\in L^\infty((-\infty,t_j]\times\Dd^N)$ be the unique global symmetric bounded weak solution of the regularized backward Liouville equation on $(-\infty,t_j]$, with recursively defined final conditions
\[\Phi_{N,\e,R}^m(t_m)=\chi_Rh_m,\qquad
\Phi_{N,\e,R}^j(t_j)=\chi_Rh_j\Phi_{N,\e,R}^{j+1}(t_j),\quad 1\le j<m.\]
For fixed $R$, the weak-$*$ compactness argument for bounded backward Liouville solutions in~\cite[Proposition~A.2]{BDJ-24} allows to deduce, up to a further extraction as $\e\downarrow0$,
\begin{equation}\label{eq:conv-PhiNeps}
\Phi_{N,\e,R}^j\to\Phi_{N,R}^j\quad\text{in $C_{\loc}((-\infty,t_j];w^*L^\infty(\Dd^N))$},\qquad1\le j\le m,
\end{equation}
where the limit $\Phi_{N,R}^j$ is a symmetric bounded weak solution of the backward Liouville equation on $(-\infty,t_j]$, with final conditions
\[\Phi_{N,R}^m(t_m)=\chi_Rh_m,\qquad
\Phi_{N,R}^j(t_j)=\chi_Rh_j\Phi_{N,R}^{j+1}(t_j),\quad 1\le j<m.\]
For the smooth dynamics, repeated conditioning gives
\begin{equation}\label{eq:multi-regularized-duality}
\int_{\Dd^N}\Phi_{N,\eps,R}^1(0)F_{N,\eps}^\circ=\int_{(\Dd^N)^m}\big((\chi_Rh_1)\otimes\ldots\otimes(\chi_Rh_m)\big)\,d\Gamma_{N,\eps}^{t_1,\ldots,t_m}.
\end{equation}
The right-hand side converges by~\eqref{eq:conv-GammaNeps} since the integrand is compactly supported.
To pass to the limit in the left-hand side, we recall~\eqref{eq:conv-PhiNeps} and $F_{N,\eps}^\circ\to F_N^\circ$ in $L^1(\Dd^N)$, which leads us to
\begin{equation}\label{eq:multi-cutoff-duality}
\int_{\Dd^N}\Phi_{N,R}^1(0)F_N^\circ=\int_{(\Dd^N)^m}\big((\chi_Rh_1)\otimes\ldots\otimes(\chi_Rh_m)\big)\,d\Gamma_N^{t_1,\ldots,t_m}.
\end{equation}
Finally, passing to the limit $R\uparrow\infty$, the desired iterated duality identity~\eqref{eq:multi-dual} easily follows.
\end{proof}

\bibliographystyle{plain}
\bibliography{biblio}

\end{document}